\documentclass[a4paper, 11pt]{amsart}

\usepackage{amsmath,amssymb}
\usepackage{amsthm}
\usepackage{enumitem}

\theoremstyle{definition}
\newtheorem{thm}{Theorem}[section]
\newtheorem{prop}[thm]{Proposition}
\newtheorem{lem}[thm]{Lemma}
\newtheorem{cor}[thm]{Corollary}
\newtheorem{clm}[thm]{Claim}
\newtheorem{dfn}[thm]{Definition}
\newtheorem{ex}[thm]{Example}
\newtheorem{rem}[thm]{Remark}

\numberwithin{equation}{section}

\newcommand{\obsdiam}[2]{\mathrm{ObsDiam}({#1};-{#2})}
\newcommand{\diam}[1]{\mathrm{diam}({#1})}
\newcommand{\supp}[1]{\mathrm{supp}\,{#1}}

\newcommand{\cov}[3]{\mathrm{Cov}({#1};{#2},{#3})}

\newcommand{\Sep}{\mathrm{Sep}}
\newcommand{\dP}{d_{\mathrm{P}}}
\newcommand{\dPH}{(d_{\mathrm{P}})_{\mathrm{H}}}
\newcommand{\dTV}{d_{\mathrm{TV}}}
\newcommand{\extN}{\overline{\mathbb{N}}}

\newcommand{\kl}{k(\ell)}
\newcommand{\kpl}{k^{\prime}(\ell)}

\newcommand{\A}{\mathcal{A}}

\newcommand{\M}{\mathcal{M}}
\newcommand{\Q}{\mathcal{Q}}
\newcommand{\X}{\mathcal{X}}

\newcommand{\N}{\mathbb{N}}
\newcommand{\R}{\mathbb{R}}

\newcommand{\ON}{\overline{\mathbb{N}}}

\newcommand{\TX}{\widetilde{X}}

\allowdisplaybreaks

\title{direct sums and decompositions of Gromov's pyramids}
\author[Toshiaki Miyamoto]{Toshiaki Miyamoto$^{*}$}
\address{Mathematical Institute, Tohoku University,
Sendai 980-8578, Japan}
\email{miyamoto.toshiaki.s2@dc.tohoku.ac.jp}
\date{\today}
\thanks{${}^{*}$The author is a JSPS Research Fellow.}

\begin{document}
\begin{abstract}
    Gromov introduced the notion of a pyramid as a generalization of a metric measure space, based on the idea of the concentration of measure phenomenon.
    In this paper, we introduce the concept of a \textit{direct sum} of pyramids, which naturally appears as a limit of a sequence of metric measure spaces whose measures concentrate on finitely or countably many regions, with the distances between these regions diverging to infinity.
    As one of our main results, we prove that any pyramid admits a unique direct sum decomposition.
    Moreover, as an application, we establish the method for checking whether a given pyramid is an extended metric measure space.
\end{abstract}
\maketitle
\tableofcontents
\section{Introduction}
The geometry of metric measure spaces has its origin in the study of convergence theory and analysis on Riemannian manifolds with lower Ricci curvature bounds.
In particular, convergence theory is an important topic in this field, and various convergence notions, such as \textit{measured Gromov--Hausdorff convergence} (see~\cite{Fukaya}), have been studied.

Gromov~\cite[Chapter 3$\frac{1}{2}_{+}$]{Gromov} introduced the \textit{observable distance} $d_{\mathrm{conc}}(X,Y)$ between two mm-spaces $X$ and $Y$ and developed a new convergence theory based on the idea of \textit{concentration of measure phenomenon}.
We here call a triple $(X, d_{X}, \mu_{X})$ an \textit{mm-space}, where $(X, d_{X})$ is a complete and separable metric space and $\mu_{X}$ a Borel probability measure on $X$.
The observable distance induces a topology, called the \textit{concentration topology}, on the set $\X$ of all isomorphism classes of mm-spaces.

The concentration of measure phenomenon, due to L\'{e}vy and Milman, states that 1-Lipschitz functions are almost constant with high probability (see~\cite{Ledoux}).
For example, it is well known that a 1-Lipschitz function on a high-dimensional unit sphere is highly concentrated around its median (or \textit{L\'{e}vy mean}) (for details, see~\cite{Ledoux}).
Gromov and Milman~\cite{GrM} formulated the concentration of measure phenomenon as a property of sequences of mm-spaces, and a sequence with this property is called a \textit{L\'{e}vy family}.
For example, the sequence of the $n$-dimensional unit spheres $\{S^{n}(1)\}_{n=1}^{\infty}$ is a L\'{e}vy family.
It is known that a L\'{e}vy family converges to the one-point mm-space in the concentration topology.

Gromov also introduced the notion of a \textit{pyramid} defined as a certain directed subset of $\X$ with respect to the Lipschitz order $\prec$ (see Definition~\ref{def:Lipschitz order}).
There exists a natural embedding $\iota:\X \ni X \mapsto \mathcal{P}_{X} \in \Pi$, where $\Pi$ is the space of all pyramids, and \[\mathcal{P}_{X} := \{Y \in \X \mid Y \prec X\},\]which is called the \textit{pyramid associated with} $X$ (for details, see Section 2.4). 
Shioya~\cite{OS2015_limit_formula,S2016book,Shioya_2017} introduced a metric $\rho$ on the space $\Pi$ and proved that the map $\iota$ yields a natural compactification of $\X$ with the concentration topology.
Therefore, the notion of a pyramid is important to understand the concentration topology.
We sometimes identify the mm-space $X$ with the pyramid $\mathcal{P}_{X}$ by this map $\iota$.
Moreover, the study of pyramids itself is also of independent interest.
The limit of a sequence of mm-spaces in $\Pi$ is not necessarily an mm-space.
For example, the sequence $\{S^{n}(\sqrt{n})\}_{n=1}^{\infty}$ of the $n$-dimensional sphere with radius $\sqrt{n}$ converges in $\Pi$ to the \textit{infinite-dimensional virtual standard Gaussian space} $\Gamma^{\infty}$, which is a pyramid but not an mm-space (for details, see~\cite{K2022, S2016book, Shioya_2017}).

In this paper, we introduce the notion of a \textit{direct sum of pyramids} as a generalization of the notion of a direct sum of mm-spaces and study its properties.
The notion of a direct sum of pyramids is essential for studying sequences of mm-spaces whose measures concentrate on finitely or countably many regions, with the distances between these regions diverging to infinity.
As a main theorem, we prove that any pyramid of infinite observable diameter (see Definition~\ref{def:observable diameter}) admits a uniquely determined direct sum decomposition.
As an application, we investigate a general method for determining whether a given pyramid is an extended mm-space or not by introducing the \textit{covering number} of a pyramid.

In the following, we describe the setting and the main results.
We set \[\overline{\mathbb{N}}:= \mathbb{N} \cup \{\infty\} = \{1,2,\dots, \} \cup \{\infty\},\]
and for $N \in \overline{\mathbb{N}}$, 
\begin{equation*}
    [N] := \begin{cases}
                \{1,2,\dots, N\} &\text{if $N < \infty$,}\\
                \N &\text{if $N = \infty$.} 
            \end{cases}
\end{equation*}
Furthermore, we set
\[\mathcal{A}_{1} := \bigcup_{N \in \overline{\mathbb{N}}}\{(a_{n})_{n=1}^{N} \in (0,1]^{N} \mid \|(a_{n})_{n=1}^{N}\|_{1} := \textstyle\sum_{n=1}^{N} a_{n} = 1\}, \] and \[\A := \{(0)\}\cup \bigcup_{N \in \extN}\left\{(a_{n})_{n=1}^{N} \in (0,1]^{N} \,\middle\vert\, \begin{aligned}&a_{n} \ge a_{n+1}\hspace{1mm}\text{for}\hspace{1mm}n \in [N-1]\\ &\text{and}\hspace{1mm} \|(a_{n})_{n=1}^{N}\|_{1} \le 1 \end{aligned} \right\}.\]
Note that whenever we write $\{X_{n}\}_{n=1}^{N}$, $(a_{n})_{n=1}^{N} \in \mathcal{A}$, or $(a_{n})_{n=1}^{N} \in \mathcal{A}_{1}$, we always assume $N \in \overline{\mathbb{N}}$.

For $(a_{n})_{n=1}^{N} \in \mathcal{A}_{1}$ and a sequence of pyramids $\{\mathcal{P}_{n}\}_{n=1}^{N}$, we define
\[\sum_{n=1}^{N} \mathcal{P}_{n}^{a_{n}} := \left\{X \in \mathcal{X} \,\middle|\, \begin{aligned} &\text{For}\hspace{1mm}n \in [N],\hspace{1mm}\text{there exist an mm-space}\hspace{1mm} X_{n} \in \mathcal{P}_{n}\\ &\text{and a 1-Lipschitz map}\hspace{1mm} f_{n}:X_{n} \to X\hspace{1mm}\text{such that}\\ &\mu_{X} = \sum_{n=1}^{N} a_{n}(f_{n})_{*}\mu_{X_{n}} \end{aligned}\right\}.\]
We call this the \textit{direct sum of pyramids} $\{\mathcal{P}_{n}\}_{n=1}^{N}$ \textit{with weight} $(a_{n})_{n=1}^{N}$.
If $N < \infty$, then we sometimes write $\mathcal{P}_{1}^{a_{1}} + \mathcal{P}_{2}^{a_{2}}+ \cdots + \mathcal{P}_{N}^{a_{N}}$ instead of $\textstyle\sum_{n=1}^{N}\mathcal{P}^{a_{n}}$. 

We obtain the following theorem.
\begin{thm}\label{thm:direct sum of pyramids is a pyramid}
    For any $(a_{n})_{n=1}^{N} \in \mathcal{A}_{1}$ and any sequence $\{\mathcal{P}_{n}\}_{n=1}^{N}$ of pyramids, $\sum_{n=1}^{N} \mathcal{P}_{n}^{a_{n}}$ is a pyramid.
\end{thm}
Roughly speaking, the direct sum $\sum_{n=1}^{N}\mathcal{P}_{n}^{a_{n}}$ is a pyramid consisting of the pyramids $\mathcal{P}_{n}$ with weights $a_{n}$, where the distance between $\mathcal{P}_{n}$ and $\mathcal{P}_{m}$ is infinite for $n \neq m$, although this distance is not defined in a strict sense.

We remark that the direct sum of pyramids defined above is a generalization of pyramids generated by atoms observed in~\cite{EKM2024, KNS2024} (see~\eqref{eq:pyramid generated by atoms is a direct sum of pyramids}).

We calculate a nontrivial example of a sequence of mm-spaces that converges in $\Pi$ to a direct sum of pyramids in Section 4.
For example, the wedge sum $T^{n}\vee T^{n}$ of two copies of the $n$-torus $T^{n} := S^{1}(1)\times S^{1}(1)\times \cdots \times S^{1}(1)$ converges in $\Pi$ to the direct sum of two copies of the pyramid $T^{\infty}$, which is the limit of the sequence $\{T^{n}\}_{n=1}^{\infty}$ in $\Pi$.

Next, we describe the main results.
We say that the observable diameter of a pyramid is \textit{finite} if and only if its $\kappa$-observable diameter is finite for every $0 < \kappa < 1$ (see Definition~\ref{def:finiteness of the observable diameter of pyramids}).
We obtain the following decomposition of pyramids.
\begin{thm}\label{thm:decomposition of pyramids}
    For any pyramid $\mathcal{P}$, there exist $A=(a_{n})_{n=1}^{N} \in \mathcal{A}$ and a sequence $\{\mathcal{P}_{n}\}_{n=1}^{N}$ of pyramids of finite observable diameter such that \[\mathcal{P} = \mathcal{X}^{1-\|A\|_{1}} + \sum_{n=1}^{N} \mathcal{P}_{n}^{a_{n}}.\]
\end{thm}
Note that this theorem is nontrivial only in the case when a pyramid $\mathcal{P}$ has infinite observable diameter.
We here explain briefly how Theorem~\ref{thm:decomposition of pyramids} is proved.
Considering the scale transformation $t\mathcal{P}$, $t > 0$, and letting $t \to 0+$, we observe that $t\mathcal{P}$ converges weakly to some scale-invariant pyramid $\mathcal{Q}$, which is a pyramid generated by atoms (see~\cite[Theorem 1.7]{KNS2024}).
By a simple observation, we see that for any pyramid $\mathcal{R}$ of finite observable diameter, $t\mathcal{R}$ converges weakly to the pyramid consisting of only a one-point mm-space.
Based on the above observation, we expect that the atoms which appear in the limit $\mathcal{Q}$ come from the one-point mm-space which is the limit of the part of finite observable diameter in $\mathcal{P}$.
Reconstructing the parts that collapsed to a one point, we will obtain the decomposition of $\mathcal{P}$.

In addition, we obtain the following theorem on the uniqueness of the decomposition.
\begin{thm}\label{thm:the uniquness of decomposition of pyramid}
    Let $A=(a_{n})_{n=1}^{N}, B=(b_{n})_{n=1}^{M}$ be elements of $\mathcal{A}$, and let $\{\mathcal{P}_{n}\}_{n=1}^{N}, \{\mathcal{Q}_{n}\}_{n=1}^{M}$ be sequences of pyramids of finite observable diameter. 
    If \[\mathcal{X}^{1-\|A\|_{1}} + \sum_{n=1}^{N}\mathcal{P}_{n}^{a_{n}} = \mathcal{X}^{1-\|B\|_{1}} + \sum_{n=1}^{M} \mathcal{Q}_{n}^{b_{n}},\] then the following (i) and (ii) hold.
    \begin{enumerate}
        \item $N=M$.
        \item There exists a bijective map $f:[N] \to [N]$ such that $\mathcal{P}_{n} = \mathcal{Q}_{f(n)}$ and $a_{n} = b_{f(n)}$ for each $n \in [N]$.
    \end{enumerate}
\end{thm}
The uniqueness result in the case of pyramids generated by atoms was established in~\cite[Proposition 6.9]{EKM2024} and~\cite[Lemma 4.4]{KNS2024}.
In the proof of Theorem~\ref{thm:the uniquness of decomposition of pyramid}, we use this result.

Next, we study the \textit{covering number} of a pyramid.
\begin{dfn}[\cite{VL_mm_entropy}]
    For $r > 0$, $0 < \kappa < 1$, and for an mm-space $X$, we define \[\cov{X}{r}{\kappa} := \min \hspace{1mm}\{\#\mathcal{N}\mid \mathcal{N} \subset X,\hspace{1mm}\mu_{X}(B_{r}(\mathcal{N})) \ge 1-\kappa\},\] and call it the $(r,\kappa)$-\textit{covering number} of $X$. 
    Note that, for a set $A$, we denote by $\# A$ the cardinality of $A$. 
\end{dfn}
The covering number is an invariant for mm-spaces, and is monotone increasing with respect to the Lipschitz order relation.
This enables us to define the covering number of a pyramid (for details, see~\cite{EKM2024}).
\begin{dfn}
    For $r > 0$, $0 < \kappa < 1$, and for a pyramid $\mathcal{P}$, we define \[\cov{\mathcal{P}}{r}{\kappa}:=\sup_{X \in \mathcal{P}} \cov{X}{r}{\kappa},\] and call it the $(r,\kappa)$-\textit{covering number} of $\mathcal{P}$.
\end{dfn}
As an application of Theorem~\ref{thm:decomposition of pyramids}, we obtain the following theorem.
\begin{thm}\label{thm:when a pyrmid is an ext. mm sp.}
    Let $\mathcal{P}$ be a pyramid.
    Then the following (i) and (ii) are equivalent to each other.
    \begin{enumerate}
        \item For any $r > 0$ and any $0 < \kappa < 1$, we have $\cov{\mathcal{P}}{r}{\kappa} < +\infty$.
        \item There exists an extended mm-space $X$ such that $\mathcal{P} = \mathcal{P}_{X}$.
    \end{enumerate}
\end{thm}
The term \textit{extended mm-space} means that its metric takes its values in $[0, +\infty]$ (for details, see Definition~\ref{def:extended mm-space}).
Moreover, for an extended mm-space $X$, $\mathcal{P}_{X}$ is the \textit{pyramid associated with} $X$ (for details, see Section 3.2).
Theorem~\ref{thm:when a pyrmid is an ext. mm sp.} provides a new general method for deciding whether a given pyramid is an extended mm-space.
For example, if a pyramid $\mathcal{Q}$ satisfies $\mathcal{Q} \subset \mathcal{P}_{X}$ for some extended mm-space $X$, then $\mathcal{Q}$ is also an extended mm-space, that is, we can find an extended mm-space $Y$ such that $\mathcal{Q} = \mathcal{P}_{Y}$ by Theorem~\ref{thm:when a pyrmid is an ext. mm sp.} since the covering number of $\mathcal{Q}$ is bounded above by that of $\mathcal{P}_{X}$, which is finite (see Corollary~\ref{cor:Q < P_X, X; ext. mm-sp. => Q=P_Y}).
We use Theorem~\ref{thm:decomposition of pyramids} only in the proof of (i) $\Rightarrow$ (ii), while (i) $\Leftarrow$ (ii) follows from a simple argument.
By Theorem~\ref{thm:when a pyrmid is an ext. mm sp.} (i) $\Leftarrow$ (ii), we provide an alternative proof that any infinite product space of a nontrivial mm-space is not an mm-space (see Corollary~\ref{cor:infinite product is not mm-sp.}).

The organization of this paper is as follows.
In Section 2, we summarize the preliminaries.
In Section 3, we study properties of the direct sum of mm-spaces and pyramids.
After that, in Section 4, we calculate an example of a sequence of mm-spaces that converges in $\Pi$ to a direct sum of pyramids.
In Section 5, we prove Theorem~\ref{thm:decomposition of pyramids} and Theorem~\ref{thm:the uniquness of decomposition of pyramid}.
Finally, in Section 6, we prove Theorem~\ref{thm:when a pyrmid is an ext. mm sp.}. 

\,

\noindent\textbf{Acknowledgments.}
The author would like to thank Professor Takashi Shioya for his valuable comments and helpful support.
The author is also grateful to Shigeaki Yokota for helpful discussions on the proof of Lemma~\ref{lem:the measure of r-balls of infinite product space are zero}.
The author would like to thank Daisuke Kazukawa for his helpful comments.

\section*{Notation}
\begin{itemize}
    \item $\N:= \{1,2,\dots\}$;
    \item $\overline{\N} := \N \cup \{\infty\}$;
    \item For a set $A$, we denote by $\#A$ the cardinality of $A$;
    \item Let $(X, d_{X})$ be a metric space. For a subset $A \subset X$ and a real number $r >0$, we set \[U_{r}^{X}(A) = U_{r}(A) := \{ x \in X \mid d_{X}(x, A) < r\} \] and
        \begin{equation*}
            B_{r}^{X}(A) = B_{r}(A) := \{x \in X \mid d_{X}(x,\,A) \le r\},
        \end{equation*}
        where $d_{X}(x,\, A) := \inf \{d_{X}(x,y) \mid y \in A\}$. 
    \item For $N \in \ON$, we set 
            \begin{equation*}
                [N] := \begin{cases}
                            \{1,2,\dots,N\} &\text{if $N < \infty$},\\
                            \N &\text{if $N=\infty$};
                        \end{cases}
            \end{equation*}
    \item For a (finite or infinite) sequence $(a_{n})_{n=1}^{N}$, $N \in \ON$, of real numbers, we define \[\|(a_{n})_{n=1}^{N}\|_{1} := \textstyle\sum_{n=1}^{N}\,\lvert a_{n}\rvert.\]
    \item $\mathcal{A}_{1} := \bigcup_{N \in \extN}\{A=(a_{n})_{n=1}^{N} \in (0,1]^{N} \mid \|A\|_{1}  = 1\}$;
    \item $\A := \{(0)\}\cup \displaystyle\bigcup_{N \in \overline{\N}}\left\{(a_{n})_{n=1}^{N} \in (0,1]^{N} \,\middle\vert\, \begin{aligned}&a_{n} \ge a_{n+1}\hspace{1mm}\text{for}\hspace{1mm}n \in [N-1]\\ &\text{and}\hspace{1mm} \|(a_{n})_{n=1}^{N}\|_{1} \le 1 \end{aligned} \right\}$;
\end{itemize}
Note that whenever we write $\{X_{n}\}_{n=1}^{N}$, $(a_{n})_{n=1}^{N} \in \mathcal{A}$, or $(a_{n})_{n=1}^{N} \in \mathcal{A}_{1}$, we always assume $N \in \overline{\mathbb{N}}$.

\section{Preliminaries}

In this section, we recall some definitions and properties of metric measure spaces and Gromov's pyramids. 
Our main references are~\cite[Chapter 3$\frac{1}{2}_{+}$]{Gromov} and~\cite{S2016book}, and we follow the terminology used therein.
\subsection{Metric Measure spaces}
\begin{dfn}[mm-space]\label{def:mm-space}
    Let $(X, d_{X})$ be a complete and separable metric space and $\mu_{X}$ a Borel probability measure on $X$. 
    We call the triple $(X, d_{X}, \mu_{X})$ an \textit{mm-space}. 
    For any mm-space $X$, we denote its metric by $d_{X}$ and its Borel probability measure by $\mu_{X}$.
\end{dfn}

\begin{dfn}[mm-isomorphism]\label{def:mm-isomorphism}
    Let $X$ and $Y$ be mm-spaces. 
    We say that $X$ and $Y$ are \textit{mm-isomorphic} to each other if there exists an isometry $f:\supp{\mu_{X}} \to \supp{\mu_{Y}}$ such that $f_{*}\mu_{X} = \mu_{Y}$, where $\supp{\mu_{X}}$ and $\supp{\mu_{Y}}$ denote the support of $\mu_{X}$ and $\mu_{Y}$, respectively, and $f_{*}\mu_{X}$ is the push-forward measure of $\mu_{X}$ by $f$.
    We call such a map $f$ an \textit{mm-isomorphism}. 
    We denote by $\mathcal{X}$ the set of all mm-isomorphism classes of mm-spaces.
\end{dfn}
Since any mm-space $X$ is mm-isomorphic to $(\supp{\mu_{X}}, d_{X}, \mu_{X})$, we assume that $X = \supp{\mu_{X}}$ for any mm-space $X$ unless otherwise stated.
\begin{dfn}[Nontrivial mm-space]\label{def:nontrivial mm-space}
    We say that an mm-space $X$ is \textit{nontrivial} if it is not mm-isomorphic to the one-point mm-space $*$, equipped with the trivial metric and the Dirac measure.
\end{dfn}

\begin{dfn}[Domination]\label{def:Lipschitz order}
    For two mm-spaces $X$ and $Y$, we write $Y \prec X$ if there exists a 1-Lipschitz map $f:X \to Y$ such that $f_{*}\mu_{X} = \mu_{Y}$. 
    We call such an $f$ a \textit{domination map} and the relation $\prec$ on $\mathcal{X}$ the \textit{Lipschitz order}.
\end{dfn}

\begin{rem}
    The Lipschitz order is a partial order relation on $\mathcal{X}$ (for details, see~\cite[Proposition 2.11]{S2016book}).
\end{rem}

\subsection{Box distance}
\ 

Let $(X, d_{X})$ be a metric space. For a subset $A \subset X$ and a real number $r >0$, we set \[U_{r}^{X}(A) = U_{r}(A) := \{ x \in X \mid d_{X}(x, A) < r\} \] and
\begin{equation*}
    B_{r}^{X}(A) = B_{r}(A) := \{x \in X \mid d_{X}(x,\,A) \le r\},
\end{equation*}
where $d_{X}(x,\, A) := \inf \{d_{X}(x,y) \mid y \in A\}$.

For a metric space $X$, we write $P(X)$ for the space of all Borel probability measures on $X$.
\begin{dfn}[Prokhorov distance]\label{def:Prokhorov distance}
    Let $(X,\, d_{X})$ be a metric space. For $\mu,\, \nu \in P(X)$, the \textit{Prokhorov distance} $d_{\mathrm{P}}(\mu, \nu)$ between $\mu$ and $\nu$ is defined by \[\dP(\mu, \nu) := \inf \{\varepsilon > 0 \mid \mu(U_{\varepsilon}(A)) \ge \nu(A) -\varepsilon \hspace{2mm}\text{for any Borel subset}\hspace{2mm}A \subset X\}. \]
\end{dfn}

The Prokhorov distance defines a metric on $P(X)$. 
If $X$ is separable, then $d_{\mathrm{P}}$ is compatible with the weak convergence of Borel probability measures (for details, see~\cite[Section 6]{Bill}).

\begin{dfn}[Total variation distance]\label{def:the total variation distance}
    Let $X$ be a metric space.
    For $\mu, \nu \in P(X)$, we define the \textit{total variation distance} $\dTV(\mu, \nu)$ between $\mu$ and $\nu$ by \[\dTV(\mu, \nu) := \sup \{\lvert \mu(A) - \nu(A)\rvert\mid A \hspace{1mm}\text{is a Borel subset of}\hspace{1mm}X\}.\]
\end{dfn}
The following lemma is well known (see, for example,~\cite{Huber}).
\begin{lem}\label{lem:dP < dTV}
    For any Borel probability measures $\mu$ and $\nu$ on a metric space $X$, we have \[\dP(\mu, \nu) \le \dTV(\mu, \nu).\]
\end{lem}

\begin{dfn}[Distortion]\label{distortion}
    Let $(X,\, d_{X})$ and $(Y,\, d_{Y})$ be metric spaces. 
    For a subset $S \subset X\times Y$, we define 
    \begin{equation*}
        \mathrm{dis}(S) := \begin{cases}
                                \sup\{|d_{X}(x,x^{\prime})-d_{Y}(y,y^{\prime})|\mid (x,y), (x^{\prime},y^{\prime}) \in S\} &\text{if $S \neq \emptyset$}, \\
                                0 &\text{if $S=\emptyset$} ,
                            \end{cases}
    \end{equation*}
    and we call it the \textit{distortion} of $S$.
\end{dfn}

\begin{dfn}[Coupling]\label{def:coupling}
    Let $X$ and $Y$ be metric spaces, and let $\mu \in P(X)$ and $\nu \in P(Y)$. 
    We say that $\pi \in P(X\times Y)$ is a \textit{coupling} between $\mu$ and $\nu$ if \[\pi(A \times Y) = \mu(A)\hspace{3mm}\text{and}\hspace{3mm}\pi(X\times B) = \nu(B)\] hold for any Borel subsets $A \subset X$ and $B \subset Y$. 
\end{dfn}

\begin{dfn}[Box distance,~\cite{Nakajima_box_dist}]\label{def:box distance}
    For two mm-spaces $X$ and $Y$, we define the \textit{box distance} $\square(X,Y)$ between $X$ and $Y$ as the infimum of $\max\{\mathrm{dis}(S), 1-\pi(S)\}$, where $\pi$ runs over all couplings between $\mu_{X}$ and $\mu_{Y}$ and $S$ runs over all Borel subsets of $X\times Y$.
\end{dfn}

\begin{rem}
    We note that the box distance was originally defined using the notion of a \textit{parameter} by Gromov~\cite{Gromov}.
    It was shown by Nakajima~\cite{Nakajima_box_dist} that Definition~\ref{def:box distance} is equivalent to the original definition.
\end{rem}

The function $\square:\X \times \X \to [0,1]$ is a metric on $\X$ (see~\cite[Theorem 4.10]{S2016book}) and $(\X, \square)$ is a complete and separable metric space (see~\cite[Theorem 4.14, Proposition 4.25]{S2016book}).
The topology on $\X$ induced by the metric $\square$ is called the \textit{box topology}.

It is sometimes useful to describe convergence in the box topology in terms of the following maps.

\begin{dfn}[$\varepsilon$-mm-isomorphism]\label{def:epsilon-mm-isom.}
    Let $X$ and $Y$ be mm-spaces, let $f:X \to Y$ be a Borel measurable map, and let $\varepsilon \ge 0$.
    The map $f$ is said to be an $\varepsilon$-\textit{mm-isomorphism} if there exists a Borel subset $\TX \subset X$ such that 
    \begin{itemize}
        \item $\mu_{X}(\TX) \ge 1 - \varepsilon$;
        \item $|d_{Y}(f(x),\, f(y))-d_{X}(x,y)| \le \varepsilon$ for any $x, y \in \TX$;
        \item $\dP(\mu_{Y},\, f_{*}\mu_{X}) \le \varepsilon$.
    \end{itemize}
    We call $\TX$ a \textit{non-exceptional domain} of $f$.
\end{dfn}

\begin{prop}[{\cite[Lemma 4.22]{S2016book}}] \label{prop:box dist. <-> epsillon-mm-isom.}
    Let $X$ and $Y$ be mm-spaces and let $\varepsilon > 0$. 
    Then we have the following (i) and (ii).
    \begin{enumerate}
        \item If there exists an $\varepsilon$-mm-isomorphism from $X$ to $Y$, then $\square(X,Y) \le 3\varepsilon$.
        \item If $\square(X,Y) < \varepsilon$, then there exists a $3\varepsilon$-mm-isomorphism from $X$ to $Y$.
    \end{enumerate}
\end{prop}
The following lemma is useful to estimate box distances.
\begin{lem}[{\cite[Proposition 4.12]{S2016book}}]\label{lem:box < 2dP}
    Suppose that $(X,d_{X})$ is a complete and separable metric space.
    For any $\mu, \nu \in P(X)$, we have \[\square((X,d_{X},\mu),\hspace{1mm}(X,d_{X},\nu)) \le 2\dP(\mu, \nu). \]
\end{lem}
Next, we recall the $\square$-\textit{precompactness criterion}.
\begin{dfn}\label{def:supporting net}
    Let $\varepsilon > 0$, and let $X$ be an mm-space and $\mathcal{N} \subset X$ a net, that is, a discrete subset of $X$.
    We say that $\mathcal{N}$ is an $\varepsilon$-\textit{supporting net} if it satisfies \[\mu_{X}(B_{\varepsilon}(\mathcal{N})) \ge 1-\varepsilon.\]
\end{dfn}
\begin{dfn}\label{def:box-precompact subset}
    A subset $\mathcal{Y} \subset \X$ is said to be $\square$-\textit{precompact} if the metric space $(\mathcal{Y}, \square)$ is totally bounded.
\end{dfn}

\begin{lem}[{\cite[Lemma 4.28 (3)]{S2016book}}]\label{lem:box-precompactness criterion}
    For a subset $\mathcal{Y} \subset \X$, the following (i) and (ii) are equivalent to each other.
    \begin{enumerate}
        \item $\mathcal{Y}$ is $\square$-precompact.
        \item For any $0 < \varepsilon < 1$, there exists a positive number $\Delta(\varepsilon)$ such that for any $Y \in \mathcal{Y}$, there exists an $\varepsilon$-supporting net $\mathcal{N} \subset Y$ that satisfies $\#\mathcal{N} \le \Delta(\varepsilon)$ and $\diam{\mathcal{N}} \le \Delta(\varepsilon)$.
    \end{enumerate}
\end{lem}
\subsection{Pyramids}
\begin{dfn}[Pyramid]\label{def:pyramid}
    Let $\mathcal{P}$ be a nonempty subset of $\X$. 
    We say that $\mathcal{P}$ is a \textit{pyramid} if it satisfies the following (i), (ii), and (iii).
    \begin{enumerate}
        \item $\mathcal{P}$ is closed in the box topology.
        \item If $Y \in \mathcal{P}$ and $X\prec Y$, then $X \in \mathcal{P}$.
        \item The subset $\mathcal{P}$ is \textit{directed} with respect to the Lipschitz order, that is, for any $X, Y \in \mathcal{P}$, there exists $Z \in \mathcal{P}$ such that $X \prec Z$ and $Y \prec Z$. 
    \end{enumerate}
\end{dfn}
The sets $\{*\}$ and $\X$ are both pyramids.
For an mm-space $X$, define \[\mathcal{P}_{X} := \{Y \in \X\mid Y \prec X\}.\]
Then $\mathcal{P}_{X}$ is a pyramid, called the \textit{pyramid associated with} $X$ (see~\cite[\S6.1]{S2016book}). 
For two mm-spaces $X$ and $Y$, $X \prec Y$ if and only if $\mathcal{P}_{X} \subset \mathcal{P}_{Y}$.
The pyramid $\X$ is maximum, that is, any pyramid is contained in $\X$.

Next, we recall the notion of the \textit{weak convergence} of pyramids.

\begin{dfn}[Weak convergence]\label{def:weak Hausdorff convergence of pyramids}
    Let $\mathcal{P}$ and $\mathcal{P}_{n}$, $n \in \N$, be pyramids. We say that $\mathcal{P}_{n}$ \textit{converges weakly} to $\mathcal{P}$ if the following (i) and (ii) hold.
    \begin{enumerate}
        \item For any $X \in \mathcal{P}$, we have \[\lim_{n \to \infty} \square(X,\, \mathcal{P}_{n}) = 0.\]
        \item For any $X \in \X \setminus \mathcal{P}$, we have \[\liminf_{n \to \infty} \square(X,\, \mathcal{P}_{n}) > 0.\]
    \end{enumerate}
\end{dfn}

The next lemma is useful to construct pyramids.

\begin{lem}[{\cite[\S1.5]{EKM2024}}]\label{lem:pyramids generated by directed set}
    For a directed subset $\mathcal{D} \subset \X$ with respect to the Lipschitz order $\prec$, \[\mathcal{P}\mathcal{D} := \overline{\bigcup_{X \in \mathcal{D}}\mathcal{P}_{X}}^{\,\square}\] is a pyramid, where upper bar with $\square$ means the closure with respect to the box topology.
\end{lem}

\begin{ex}[{$\ell^{p}$-product, for details, see~\cite[Example 1.13]{EKM2024}}]\label{ex:l^p-product of mm-spaces}
    Let $X$ and $Y$ be two mm-spaces, and let $p \in [1, \infty]$. 
    We define a function $d_{X\times_{p}Y}:(X\times Y) \times (X\times Y) \to [0, \infty)$ by \[d_{X\times_{p}Y}((x,y), (x^{\prime}, y^{\prime})) := \begin{cases} (d_{X}(x,x^{\prime})^{p} + d_{Y}(y, y^{\prime})^{p})^{1/p} &\text{if $1 \le p < \infty$}, \\ \max \{d_{X}(x, x^{\prime}), d_{Y}(y, y^{\prime})\} &\text{if $p=\infty$},\end{cases}\] for $x, x^{\prime} \in X$ and $y, y^{\prime} \in Y$. 
    The function $d_{X\times_{p} Y}$ defines a metric on $X \times Y$, and we call it \textit{$\ell^{p}$-product metric}. 
    Moreover, we define a triple \[X\times_{p} Y := (X \times Y,\, d_{X\times_{p} Y},\, \mu_{X}\otimes \mu_{Y})\] and call it the \textit{$\ell^{p}$-product} of $X$ and $Y$. 
    Note that $\mu_{X}\otimes \mu_{Y}$ is the product measure of $\mu_{X}$ and $\mu_{Y}$.
    The $\ell^{p}$-product $X \times_{p} Y$ is an mm-space. 
    For an mm-space $X$ and $p \in [1, \infty]$, we define an mm-space $X_{p}^{n}$ inductively by $X_{p}^{1} := X$, $X_{p}^{n} := (X_{p}^{n-1})\times_{p} X$ for $n \ge 2$. 
    Since the set of mm-spaces \[\mathcal{D}_{X_{p}^{\infty}} := \{X_{p}^{n} \mid n \in \N\}\] is directed, we obtain a pyramid \[X_{p}^{\infty} := \mathcal{P}\mathcal{D}_{X_{p}^{\infty}}\]by Lemma~\ref{lem:pyramids generated by directed set} and call it the \textit{infinite $\ell^{p}$-product} of $X$.
\end{ex}

\begin{ex}[$\ell^{p}$-product of pyramids,~\cite{EKM2024,KNS2024}]\label{ex:l^p-product of pyramids}
    For two pyramids $\mathcal{P}, \Q$ and for $p \in [1, \infty]$, we see that the set of mm-spaces \[\mathcal{D}_{\mathcal{P}\times_{p}\Q} := \{X\times_{p}Y \mid X \in \mathcal{P},\, Y \in \Q\}\] is directed with respect to the Lipschitz order relation. 
    By Lemma~\ref{lem:pyramids generated by directed set}, we obtain a pyramid \[\mathcal{P}\times_{p}\Q := \mathcal{P}\mathcal{D}_{\mathcal{P}\times_{p}\Q}\] and call it the \textit{$\ell^{p}$-product} of $\mathcal{P}$ and $\Q$. 
\end{ex}

\begin{lem}[{Approximation of pyramids,~\cite[Lemma 7.14]{S2016book}}]\label{lem:app. seq. of pyramid}
    For any pyramid $\mathcal{P}$, there exist mm-spaces $X_{n} \in \mathcal{P}$, $n \in \N$, such that \[X_{1} \prec X_{2} \prec \cdots \prec X_{n} \prec \cdots\hspace{3mm}\text{and}\hspace{4mm} \overline{\bigcup_{n=1}^{\infty}\mathcal{P}_{X_{n}}}^{\,\square} = \mathcal{P}.\] 
    We call such a sequence $\{X_{n}\}_{n=1}^{\infty}$ an \textit{approximation sequence} of $\mathcal{P}$, or we say that $\{X_{n}\}_{n=1}^{\infty}$ \textit{approximates} $\mathcal{P}$.
\end{lem}
\begin{rem}\label{rem:approximation sequence is a weak convergent sequence}
For a pyramid $\mathcal{P}$ and an approximation sequence $\{X_{n}\}_{n=1}^{\infty}$ of $\mathcal{P}$, it is straightforward to see that the pyramid $\mathcal{P}_{X_{n}}$ converges weakly to the pyramid $\mathcal{P}$.
\end{rem}

Next, we review the notion of \textit{measurements} to define the metric $\rho$ on $\Pi$.

\begin{dfn}[{Measurement,~\cite{OS2015_limit_formula, S2016book}}]\label{def:measurement}
    For $N \in \mathbb{N}$ and a real number $R > 0$, we define $\mathcal{M}(N,R)$ as the set of all Borel probability measures on $\R^{N}$ whose support is contained in $B_{R}^{N} := \{x \in \mathbb{R}^{N}\mid \|x\|_{\infty}\le R\}$, where $\|\cdot\|_{\infty}$ is the $\ell^{\infty}$-norm on $\R^{N}$. 
    For an mm-space $X$ and a pyramid $\mathcal{P}$, the $(N,R)$-\textit{measurement} is defined by \[\mathcal{M}(X;N,R) :=\{\,f_{*}\mu_{X}\mid f:X \to (\mathbb{R}^{N},\, \|\cdot\|_{\infty})\,\text{is 1-Lipschitz}\,\} \cap \mathcal{M}(N,R)\] and \[\M(\mathcal{P};N,R) := \{\,\mu \in \M(N,R)\mid (B_{R}^{N},\, \|\cdot\|_{\infty},\, \mu) \in \mathcal{P}\,\}.\]
\end{dfn}
\begin{rem}\label{rem:equivalent representation of measurements}
    For any $N \in \N$ and $R > 0$, and for any mm-space $X$, the following equality holds: \[\M(X;N,R) = \{\,f_{*}\mu_{X}\mid f:X \to (B_{N}^{R},\,\|\cdot\|_{\infty})\,\text{is 1-Lipschitz}\,\}.\]
\end{rem}
Note that $\M(X;N,R)$ and $\M(\mathcal{P};N,R)$ are compact subsets of $P((\R^{N},\|\cdot\|_{\infty}))$ with respect to $\dP$.
We also remark that $\M(\mathcal{P}_{X};N,R) = \M(X;N,R)$ for any mm-space $X$.

\begin{dfn}[{Metric $\rho$,~\cite[Definition 3.5]{OS2015_limit_formula}}]\label{def:metric rho}
    For two pyramids $\mathcal{P}$ and $\Q$, we define \[\rho(\mathcal{P},\Q) := \sum_{N=1}^{\infty}\frac{1}{2^{N}}\cdot \frac{1}{2N}\dPH(\M(\mathcal{P};N,N),\, \M(\Q;N,N)),\] where $\dPH$ is the Hausdorff distance induced from the Prokhorov distance.
\end{dfn}
We denote the set of all pyramids by $\Pi$.
The following theorem follows from~\cite[Theorem 3.7]{OS2015_limit_formula} and~\cite[Theorem 6.12]{S2016book}.
\begin{thm}\label{thm:metric rho}
    The function $\rho$ is a metric on the space $\Pi$ that is compatible with weak convergence. 
    Moreover, $\Pi$ is compact with respect to $\rho$.
\end{thm}

\begin{rem}
    In~\cite{S2016book, Shioya_2017}, another metric on $\Pi$ is introduced, and a result similar to Theorem~\ref{thm:metric rho} is proved.
\end{rem}

The following proposition is useful for estimating $\rho$.
It follows from~\cite[Theorem 3.7 (3)]{OS2015_limit_formula} and~\cite[Proposition 5.5 (2)]{S2016book}.
\begin{prop}\label{prop:rho < box}
    For any mm-spaces $X$ and $Y$, we have \[\rho(\mathcal{P}_{X}, \mathcal{P}_{Y}) \le \square(X,Y).\]
\end{prop}

Next, we recall the notion of 1-Lipschitz maps up to an additive error and related notions.
These will be used frequently in this paper. 
\begin{dfn}[1-Lipschitz map up to additive error]
    Let $X, Y$ be mm-spaces and $f:X \to Y$ a Borel measurable map, and let $\varepsilon \ge 0$.
    The map $f$ is said to be a \textit{1-Lipschitz map up to (an additive error)} $\varepsilon$ if there exists a Borel subset $\TX \subset X$ such that 
    \begin{itemize}
        \item $\mu_{X}(\TX) \ge 1 - \varepsilon$;
        \item $d_{Y}(f(x),\, f(y)) \le d_{X}(x,y) + \varepsilon$ \,for any $x, y \in \TX$.
    \end{itemize}
    We call such a Borel subset $\TX$ a \textit{non-exceptional domain} of $f$.
\end{dfn}

\begin{dfn}[{\cite[Definition 3.3]{KY2021}}]\label{def:Lip. order up to error}
    For two mm-spaces $X, Y$ and a real number $\varepsilon \ge 0$, we write $Y \prec_{\varepsilon} X$ if there exists a Borel measurable map $f:X \to Y$ that is 1-Lipschitz up to $\varepsilon$ and that satisfies $\dP(f_{*}\mu_{X},\, \mu_{Y}) \le \varepsilon$.
\end{dfn}

\begin{dfn}[{\cite[Definition 3.14]{KY2021}}]\label{def:domination up to error for pyramids}
    Let $\mathcal{P}$ and $\mathcal{Q}$ be two pyramids, and let $\varepsilon > 0$. 
    We write $\mathcal{P} \prec_{\varepsilon} \mathcal{Q}$ if, for any $X \in \mathcal{P}$, there exists $Y \in \mathcal{Q}$ such that $X \prec_{\varepsilon} Y$.
\end{dfn}
We will frequently use the following two results.
\begin{prop}[{\cite[Proposition 3.15]{KY2021}}]\label{prop:Pn <en Qn => P < Q}
    Let $\{\mathcal{P}_{n}\}_{n=1}^{\infty}$ and $\{\mathcal{Q}_{n}\}_{n=1}^{\infty}$ be sequences of pyramids that converge weakly to pyramids $\mathcal{P}$ and $\mathcal{Q}$, respectively. 
    Let $\{\varepsilon_{n}\}_{n=1}^{\infty}$ be a sequence of positive numbers that converges to 0. 
    If $\mathcal{P}_{n} \prec_{\varepsilon_{n}} \mathcal{Q}_{n}$ for every $n \ge 1$, then we have $\mathcal{P} \subset \mathcal{Q}$.
\end{prop}
The next lemma follows from Proposition~\ref{prop:box dist. <-> epsillon-mm-isom.}.
\begin{lem}\label{lem:PXn -> P, Z in P => Z <ek Xnk}
    Suppose that $\{X_{n}\}_{n=1}^{\infty}$ is a sequence of mm-spaces and $\mathcal{P}$ is a pyramid. 
    If $\mathcal{P}_{X_{n}}$ converges weakly to $\mathcal{P}$, then for any $Z \in \mathcal{P}$ and for any positive numbers $\varepsilon_{k} \to 0+$, there exists a subsequence $\{X_{n(k)}\}_{k=1}^{\infty}$ of $\{X_{n}\}_{n=1}^{\infty}$ such that \[Z \prec_{\varepsilon_{k}} X_{n(k)}\] for each $k \in \N$.
\end{lem}
Finally, we state the following known result, which will be used later.
\begin{prop}[cf.~{\cite[$3\frac{1}{2}_{+}$.15 (f)]{Gromov}, see also~\cite[Corollary 3.20]{KY2021}}]\label{prop:P is compact <=> P=PX}
    For a pyramid $\mathcal{P}$, the following (i) and (ii) are equivalent to each other.
    \begin{enumerate}
        \item $\mathcal{P} \subset \X$ is a compact subset with respect to the box topology.
        \item There is an mm-space $X$ such that $\mathcal{P} = \mathcal{P}_{X}$.
    \end{enumerate}
\end{prop}
\begin{proof}
    First, we prove (i) $\Rightarrow$ (ii).
    Assume (i).
    By Lemma~\ref{lem:app. seq. of pyramid}, there exists an approximation sequence $\{X_{n}\}_{n=1}^{\infty}$ of the pyramid $\mathcal{P}$.
    Since $\mathcal{P} \subset \X$ is compact with respect to the box topology, we have a convergent subsequence $\{X_{n(k)}\}_{k=1}^{\infty}$ of $\{X_{n}\}_{n=1}^{\infty}$ with respect to the box topology.
    Let $X$ be a limit space.
    By Proposition~\ref{prop:rho < box}, we see that the pyramid $\mathcal{P}_{X_{n(k)}}$ converges weakly to the pyramid $\mathcal{P}_{X}$.
    Since $\{X_{n}\}_{n=1}^{\infty}$ is an approximation of $\mathcal{P}$, the pyramid $\mathcal{P}_{X_{n(k)}}$ also converges weakly to $\mathcal{P}$, which implies that $\mathcal{P}=\mathcal{P}_{X}$.
    This completes the proof of (i) $\Rightarrow$ (ii).

    Next, we prove (i) $\Leftarrow$ (ii).
    Assume (ii).
    Take any $\varepsilon > 0$ and $\varepsilon$-supporting finite net $\mathcal{N}$ of $X$, and fix them.
    Since $X$ is separable, there exists such $\mathcal{N}$.
    Set $\Delta(\varepsilon):= \max\{\#\mathcal{N}, \diam{\mathcal{N}}\}$.
    Take any $Y \in \mathcal{P}_{X}$ and take a domination map $f:X \to Y$.
    Set $\mathcal{N}_{Y} := f(\mathcal{N})$.
    Then we have $\#\mathcal{N}_{Y} \le \Delta(\varepsilon)$ and $\diam{\mathcal{N}_{Y}} \le \Delta(\varepsilon)$.
    Furthermore, we see that \[\mu_{Y}(B_{\varepsilon}(\mathcal{N}_{Y})) = \mu_{X}(f^{-1}(B_{\varepsilon}(f(\mathcal{N})))) \ge \mu_{X}(B_{\varepsilon}(\mathcal{N})) \ge 1-\varepsilon,\]which implies that $\mathcal{N}_{Y}$ is an $\varepsilon$-supporting net of $Y$.
    By Lemma~\ref{lem:box-precompactness criterion}, we see that $\mathcal{P}_{X}$ is $\square$-precompact.
    Since $\mathcal{P}_{X}$ is closed with respect to the box topology, we obtain this proposition.
\end{proof}
\subsection{Invariants for mm-spaces and pyramids}

\, 

In this subsection, we recall some invariants for mm-spaces and pyramids.

\begin{dfn}[Separation distance]
    For an mm-space $X$ and real numbers $0 < \kappa_{0}, \kappa_{1},\dots, \kappa_{N} < 1$, $N \in \N$, with $\sum_{i=0}^{N}\kappa_{i} \le 1$, the \textit{separation distance} of $X$ is defined by
    \begin{align*}
        &\mathrm{Sep}(X;\kappa_{0},\kappa_{1},\dots, \kappa_{N}) \\ &\hspace{10mm}:= \sup\left\{\min_{i\neq j}d_{X}(A_{i}, A_{j}) \hspace{1mm} \middle\vert \begin{aligned}&A_{i} \subset X,\, i=0,1,\dots, N, \hspace{1mm}\text{are Borel subsets}\\ &\text{with}\hspace{1mm}\mu_{X}(A_{i}) \ge \kappa_{i}\end{aligned}\right\}.
    \end{align*}
    For a pyramid $\mathcal{P}$, the \textit{separation distance} of $\mathcal{P}$ is defined by \[\mathrm{Sep}(\mathcal{P};\kappa_{0},\kappa_{1},\dots,\kappa_{N}) := \lim_{\varepsilon\to0+}\sup_{X\in \mathcal{P}}\mathrm{Sep}(X;\kappa_{0}-\varepsilon,\kappa_{1}-\varepsilon,\dots,\kappa_{N}-\varepsilon).\]
\end{dfn}

Note that the separation distance is left-continuous and monotone nonincreasing in $\kappa_{i}$ for each $i$.
For any mm-space $X$ and any positive numbers $\kappa_{0},\kappa_{1},\dots,\kappa_{N}$, the following equality holds: \[\mathrm{Sep}(\mathcal{P}_{X};\kappa_{0},\kappa_{1},\dots,\kappa_{N}) = \mathrm{Sep}(X;\kappa_{0},\kappa_{1},\dots,\kappa_{N}).\]
For details, see~\cite[Section 4]{OS2015_limit_formula}.

\begin{dfn}[{Observable diameter}]\label{def:observable diameter}
    Let $X$ be an mm-space and $0 < \kappa < 1$. 
    The $\kappa$-\textit{partial diameter} of $X$ is defined by 
    \begin{align*}
        &\mathrm{diam}(X;1-\kappa) =\mathrm{diam}(\mu_{X};1-\kappa) \\ &\hspace{10mm}:= \inf\{\,\mathrm{diam}(A)\mid A\subset X\hspace{1mm}\text{is a Borel subset with}\hspace{1mm}\mu_{X}(A) \ge 1-\kappa\, \}
    \end{align*}
    and the $\kappa$-\textit{observable diameter} of $X$ is defined by \[\obsdiam{X}{\kappa} := \sup\{\,\mathrm{diam}(f_{*}\mu_{X};1-\kappa)\mid f:X \to \mathbb{R}\hspace{1mm}\text{is 1-Lipschitz}\,\}.\]
    Note that $\mathrm{diam}(\emptyset) := 0$.
    
    For a pyramid $\mathcal{P}$, the $\kappa$-\textit{observable diameter} of $\mathcal{P}$ is defined by \[\obsdiam{\mathcal{P}}{\kappa} := \sup_{X \in \mathcal{P}}\obsdiam{X}{\kappa}.\]
\end{dfn}
Note that the observable diameter is right-continuous and monotone nonincreasing in $\kappa$.
For any mm-space $X$ and $\kappa > 0$, we have \[\obsdiam{\mathcal{P}_{X}}{\kappa} = \obsdiam{X}{\kappa}.\]
For details, see~\cite[Section 3]{OS2015_limit_formula}.

\begin{rem}
    The separation distance and the observable diameter were both originally defined for mm-spaces by Gromov~\cite{Gromov}, and later extended to pyramids by Ozawa and Shioya~\cite{OS2015_limit_formula}.
    The original definition of the observable diameter of a pyramid in~\cite{OS2015_limit_formula} differs from Definition~\ref{def:observable diameter} in form, but they are equivalent to each other (see~\cite[Remark 2.14]{Y_limit_formula}).
\end{rem}

For an mm-space $X$, a pyramid $\mathcal{P}$, and a real number $t > 0$, we define \[tX := (X,\, td_{X},\, \mu_{X})\hspace{3mm}\text{and}\hspace{3mm} t\mathcal{P} := \{tX\mid X \in \mathcal{P}\}. \]

\begin{lem}[{\cite[Proposition 3.4]{OS2015_limit_formula},~\cite[Proposition 2.19]{S2016book}}]\label{lem:scale invariance of observable diameter}
    For an mm-space $X$, a pyramid $\mathcal{P}$, and real numbers $0 < \kappa < 1$ and $t > 0$, we have \[\obsdiam{tX}{\kappa} = t\obsdiam{X}{\kappa}\] and \[\obsdiam{t\mathcal{P}}{\kappa} = t\obsdiam{\mathcal{P}}{\kappa}.\]
\end{lem}

Next, we recall the notions of scale-invariant pyramids and pyramids generated by atoms.
The main references are~\cite{EKM2024} and~\cite{KNS2024}.

\begin{dfn}[{Scale-invariant pyramids,~\cite{EKM2024,KNS2024}}]\label{def:scale invariant pyramid}
    We say that a pyramid $\mathcal{P}$ is \textit{scale invariant} if $t\mathcal{P} = \mathcal{P}$ holds for any $t > 0$.
\end{dfn}

\begin{rem}
    A pyramid $\mathcal{P}$ is scale invariant if and only if there exists a number $t > 0$ with $t \neq 1$ such that $t\mathcal{P} = \mathcal{P}$ (see~\cite[Proposition 2.10]{EKM2024}).
\end{rem}

\begin{prop}[{\cite[Theorem 6.2]{EKM2024},~\cite{KNS2024}}]\label{prop:pyramids generated by atoms}
    For a sequence $A=(a_{n})_{n=1}^{N} \in \A$, set \[\mathcal{P}_{A} := \left\{X \in \X\,\middle|\, \begin{aligned} &\text{There exists a map}\hspace{1mm} f:\{1,\dots, N\} \to X\hspace{1mm} \text{such that}\\ &\mu_{X} \ge \textstyle\sum_{n=1}^{N} a_{n}\delta_{f(n)}\end{aligned} \right\}.\]
    Then $\mathcal{P}_{A}$ is a pyramid, called the \textit{pyramid generated by atoms} $A$.
\end{prop}

\begin{lem}[{\cite[Proposition 6.9]{EKM2024},~\cite[Lemma 4.4]{KNS2024}}]\label{lem:A=B <=> P_A = P_B}
    Let $A=(a_{n})_{n=1}^{N}$ and $B=(b_{n})_{n=1}^{M}$ be elements of $\mathcal{A}$.
    Then $\mathcal{P}_{A} = \mathcal{P}_{B}$ if and only if $A=B$, where $A=B$ means that $N=M$ and $a_{n}=b_{n}$ for all $n \in [N]$.
\end{lem}
The scale-invariant pyramids are completely determined.
\begin{thm}[{\cite[Theorem 1.7]{KNS2024}}]\label{thm:the characterization of scale-invariant pyramids}
    Let $\mathcal{P}$ be a pyramid.
    Then the following (i) and (ii) are equivalent to each other.
    \begin{enumerate}
        \item $\mathcal{P}$ is scale invariant.
        \item There exists $A \in \mathcal{A}$ such that $\mathcal{P}=\mathcal{P}_{A}$.
    \end{enumerate}
\end{thm}

Finally, we recall the notion of \textit{dissipation}, which is the phenomenon opposite to the concentration.

\begin{dfn}[{Dissipation, cf.~\cite[Definition 8.1]{S2016book}}]\label{def:infinitely dissipate}
    Let $\{X_{n}\}_{n=1}^{\infty}$ be a sequence of mm-spaces. 
    We say that $\{X_{n}\}_{n=1}^{\infty}$ \textit{infinitely dissipates} if, for any positive numbers $\kappa_{0}, \kappa_{1},\dots, \kappa_{N}$ with $\sum_{i=0}^{N}\kappa_{i}<1$, where $N \in \N$, it satisfies \[\lim_{n\to\infty}\Sep(X_{n};\kappa_{0},\kappa_{1},\dots,\kappa_{N}) = +\infty.\]
\end{dfn}

The following Lemma is useful to check whether a given sequence of mm-spaces infinitely dissipates.
\begin{lem}[{cf.~\cite[Lemma 8.3]{S2016book}}]\label{lem:infinitely dissipation criterion}
    Let $\{X_{n}\}_{n=1}^{\infty}$ be a sequence of mm-spaces.
    If, for each $n$, there exist Borel subsets $A_{n}^{i} \subset X_{n}$, $i=1,2,\dots,N_{n}$, $N_{n} \in \N$, that satisfies the following (i), (ii), and (iii), then $\{X_{n}\}_{n=1}^{\infty}$ infinitely dissipates.
    \begin{enumerate}
        \item $\lim_{n\to\infty}\mu_{X_{n}}(\bigcup_{i=1}^{N_{n}}A_{n}^{i}) = 1$;
        \item $\lim_{n\to\infty}\min_{i\neq j}d_{X_{n}}(A_{n}^{i}, A_{n}^{j})=+\infty$;
        \item $\lim_{n\to\infty}\max_{i=1}^{N_{n}}\mu_{X_{n}}(A_{n}^{i}) = 0$.
    \end{enumerate}
\end{lem}
The following proposition will be used frequently later.
\begin{prop}[{\cite[Proposition 8.5 (2)]{S2016book}}]\label{prop:infinitely dissipated sequence converges weakly to the maximum pyramid}
    Let $\{X_{n}\}_{n=1}^{\infty}$ be a sequence of mm-spaces.
    If it infinitely dissipates, then $\mathcal{P}_{X_{n}}$ converges weakly to $\X$ as $n \to \infty$.
\end{prop}
\section{Direct sum of mm-spaces and pyramids}
In this section, we first establish the theory of convex combinations of probability measures.
After that, we define the direct sum of mm-spaces and study its relation to extended mm-spaces. 
We then introduce the direct sum of pyramids as a generalization of that of mm-spaces.

\subsection{Convex combinations of probability measures}
\begin{dfn}\label{def:the sum of probability measures}
    Let $X$ be a metric space, and let $(a_{n})_{n=1}^{N} \in \A_{1}$ and $\mu_{n} \in P(X)$, $n \in [N]$.
    We define $\sum_{n=1}^{N}a_{n}\mu_{n}$ by \[\left(\sum_{n=1}^{N}a_{n}\mu_{n}\right)(A) := \sum_{n=1}^{N}a_{n}\mu_{n}(A)\]for any Borel subset $A \subset X$.
\end{dfn}
By definition, we see that $\sum_{n=1}^{N}a_{n}\mu_{n}$ is a Borel probability measure on $X$.
For any bijective map $\sigma:[N] \to [N]$, we easily observe that \[\sum_{n=1}^{N}a_{n}\mu_{n} = \sum_{n=1}^{N}a_{\sigma(n)}\mu_{\sigma(n)}.\]
Let $X$ and $Y$ be two metric spaces.
Let $(a_{n})_{n=1}^{N}$, $(b_{m})_{m=1}^{M} \in \A_{1}$, and let $\mu_{n} \in P(X)$, $\nu_{m} \in P(Y)$, $n \in [N]$, $m \in [M]$.
Take a bijective map $\sigma:[NM] \to [N] \times [M]$, $\sigma(n) = (\sigma_{1}(n), \sigma_{2}(n))$, $n \in [NM]$.
We note that if $N=\infty$ or $M=\infty$, then $NM:=\infty$.
We define 
\begin{equation}\label{eq:the sum of product measures}
    \sum_{n \in [N],\,m\in [M]}a_{n}b_{m}(\mu_{n}\otimes \nu_{m}) := \sum_{n=1}^{NM}a_{\sigma_{1}(n)}b_{\sigma_{2}(n)}(\mu_{\sigma_{1}(n)}\otimes \nu_{\sigma_{2}(n)}).
\end{equation}
It is straightforward to see that the right-hand side of~\eqref{eq:the sum of product measures} is independent of the choice of $\sigma$.
Therefore, the left-hand side of~\eqref{eq:the sum of product measures} is well-defined and is a Borel probability measure on $X \times Y$.

By the $\pi$-$\lambda$ theorem, we immediately obtain the following lemma.
\begin{lem}\label{lem:distributive raw for sums of measures}
    Let $X$ and $Y$ be a metric space.
    Let $(a_{n})_{n=1}^{N}$, $(b_{m})_{m=1}^{M} \in \A_{1}$, and let $\mu_{n} \in P(X)$, $\nu_{m} \in P(Y)$, $n \in [N]$, $m \in [M]$.
    Then we have \[\left(\sum_{n=1}^{N}a_{n}\mu_{n}\right)\otimes \left(\sum_{m=1}^{M}b_{m}\nu_{m}\right) = \sum_{n \in [N],\,m \in [M]}a_{n}b_{m}(\mu_{n}\otimes\nu_{m}).\]
\end{lem}
The following two lemmas are straightforward and will be used later.
\begin{lem}\label{lem:the estimation of dist. between convex combinations of prob. meas. ver_1}
    Let $M \in \N$ and $(a_{n})_{n=1}^{M} \in \A_{1}$, and let $X$ be a metric space.
    Then, for any $\mu_{n}$, $\nu_{n} \in P(X)$, $n = 1,2,\cdots, M$, we have \[d_{\mathrm{P}}\left(\sum_{n=1}^{M}a_{n}\mu_{n},\, \sum_{n=1}^{M}a_{n}\nu_{n}\right) \le \max_{n=1}^{M}\, d_{\mathrm{P}}(\mu_{n}, \nu_{n}).\]
\end{lem}
\begin{lem}\label{lem:the estimation of dist. between convex combinations of prob. meas. ver_2}
    For any Borel probability measures $\mu_{i}$, $\nu_{i}$ on a metric space $X$ for $i=1,2$, and for any real number $0 < a <1$, we have \[d_{\mathrm{P}}(a\mu_{1}+(1-a)\mu_{2},\, a\nu_{1}+(1-a)\nu_{2}) \le d_{\mathrm{P}}(\mu_{1},\, \nu_{1}) + 1-a. \] 
\end{lem}

\ 

For $A=(a_{n})_{n=1}^{N}$, $B=(b_{n})_{n=1}^{M} \in \mathcal{A}$ (or $\mathcal{A}_{1}$), we set \[\|A-B\|_{1} := \begin{cases} \sum_{n=1}^{N} |a_{n} - b_{n}| + \sum_{n=N+1}^{M} |b_{n}| &\text{if\, $N \le M$},\\
                                    \sum_{n=1}^{M} |a_{n}-b_{n}| + \sum_{n=M+1}^{N} |a_{n}| &\text{if\, $N \ge M$}.
    \end{cases} \]
    
\begin{lem}\label{lem:the estimation of dist. between convex combinations of probability measures}
    Suppose $A=(a_{n})_{n=1}^{N}$ and $B=(b_{n})_{n=1}^{N^{\prime}}$ are elements of $\mathcal{A}_{1}$, and let $\mu_{k}$, $\nu_{\ell}$, $k \in [N]$, $\ell \in [N^{\prime}]$, be Borel probability measures on a metric space $X$.
    Then for any integer $M \in [\min\{N, N^{\prime}\}]$, we have \[ d_{\mathrm{P}}\left(\sum_{n=1}^{N} a_{n}\mu_{n},\, \sum_{n=1}^{N^{\prime}} b_{n}\nu_{n}\right) \le \max_{n=1}^{M}\, d_{\mathrm{P}}(\mu_{n},\, \nu_{n}) + \|A-B\|_{1} + \sum_{n=M+1}^{N} a_{n} + \sum_{n=M+1}^{N^{\prime}} b_{n}.\]
\end{lem}

\begin{proof}
    For $n = 1,2,\cdots,M$, we set $c_{n} := \min\{a_{n}, b_{n}\}$. 
    We see that the Borel probability measures $\sum_{n=1}^{N}a_{n}\mu_{n}$ and $\sum_{n=1}^{N^{\prime}}b_{n}\nu_{n}$ are represented respectively as 
    \begin{align*}
        &\sum_{n=1}^{N}a_{n}\mu_{n} = \sum_{n=1}^{M}c_{n}\mu_{n} + \left(\sum_{n=1}^{M} (a_{n}-c_{n})\mu_{n} + \sum_{n=M+1}^{N}a_{n}\mu_{n}\right), \\
        &\sum_{n=1}^{N^{\prime}}b_{n}\nu_{n} = \sum_{n=1}^{M}c_{n}\nu_{n} + \left(\sum_{n=1}^{M}(b_{n}-c_{n})\nu_{n} + \sum_{n=M+1}^{N^{\prime}}b_{n}\nu_{n}\right). 
    \end{align*}
    By Lemma~\ref{lem:the estimation of dist. between convex combinations of prob. meas. ver_1} and Lemma~\ref{lem:the estimation of dist. between convex combinations of prob. meas. ver_2}, we have 
        \begin{align*}
            d_{\mathrm{P}}\left(\sum_{n=1}^{N}a_{n}\mu_{n},\, \sum_{n=1}^{N^{\prime}}b_{n}\nu_{n}\right) &\le d_{\mathrm{P}}\left(\sum_{n=1}^{M}\frac{c_{n}}{\sum_{m=1}^{M}c_{m}}\mu_{n},\, \sum_{n=1}^{M}\frac{c_{n}}{\sum_{m=1}^{M}c_{m}}\nu_{n}\right) + 1- \sum_{n=1}^{M} c_{n} \\
                &\le \max_{n=1}^{M}\, d_{\mathrm{P}}(\mu_{n}, \nu_{n}) + \sum_{n=1}^{N} a_{n} - \sum_{n=1}^{M} c_{n} \\
                &\le \max_{n=1}^{M}\, d_{\mathrm{P}}(\mu_{n}, \nu_{n}) + \sum_{n=1}^{M}(a_{n}-c_{n}) + \sum_{n=M+1}^{N} a_{n} \\
                &\le \max_{n=1}^{M}\, d_{\mathrm{P}}(\mu_{n}, \nu_{n}) + \|A - B\|_{1} + \sum_{n=M+1}^{N} a_{n} + \sum_{n=M+1}^{N^{\prime}} b_{n}.
        \end{align*}
    This completes the proof.
\end{proof}
By Lemma~\ref{lem:the estimation of dist. between convex combinations of probability measures}, we obtain the following corollary.
\begin{cor}\label{cor:continuity of convex combination of probability measures}
    Let $A=(a_{n})_{n=1}^{N}$, $A_{k}=(a_{nk})_{n=1}^{N}$, $k \in \N$, be elements of $\A_{1}$.
    Let $\mu_{n}$, $\mu_{nk}$, $n \in [N]$, $k \in \N$, be Borel probability measures on a metric space $X$.
    If $A_{k}$ converges to $A$ in the $\ell^{1}$-norm and, for each $n$, $\mu_{nk}$ converges weakly to $\mu_{n}$ as $k \to \infty$, then $\sum_{n=1}^{N}a_{nk}\mu_{nk}$ converges weakly to $\sum_{n=1}^{N}a_{n}\mu_{n}$ as $k \to \infty$.
\end{cor}
\subsection{Direct sum of mm-spaces}

\, 

In this paper, the term \textit{extended metric space} means that its metric is allowed to take its values in $\lbrack 0, +\infty\rbrack$.
For an extended metric space $(X, d_{X})$, the function $d_{X}$ is called an \textit{extended metric}.
\begin{dfn}[\cite{OS2015_limit_formula}]\label{def:extended mm-space}
    Let $(X, d_{X})$ be a complete and separable extended metric space, and let $\mu_{X}$ be a Borel probability measure on $(X, d_{X})$. Then the triple $(X, d_{X}, \mu_{X})$ is called an \textit{extended mm-space}. We say that two extended mm-spaces $(X,d_{X},\mu_{X})$ and $(Y,d_{Y},\mu_{Y})$ are \textit{mm-isomorphic} to each other if there exists an isometry $f:\supp{\mu_{X}} \to \supp{\mu_{Y}}$ such that $f_{*}\mu_{X} = \mu_{Y}$.
\end{dfn}

For any extended mm-space $(X, d_{X}, \mu_{X})$, we assume that $X = \supp{\mu_{X}}$ unless otherwise stated.

For $(a_{n})_{n=1}^{N} \in \mathcal{A}_{1}$ and a sequence of mm-spaces $\{X_{n}\}_{n=1}^{N}$, we define an extended metric $d$ on the disjoint union $\bigsqcup_{n=1}^{N} X_{n}$ by
  \begin{equation*}
    d(x,y) := \begin{cases}
                    d_{X_{n}}(x,y) &\text{if\, $x,y \in X_{n}$, $n \in [N]$,}\\
                    +\infty &\text{if\, $x \in X_{n},\, y \in X_{m}$\, with\, $n \neq m$,}
                \end{cases}   
  \end{equation*}
and define a Borel probability measure $\mu$ on $(\bigsqcup_{n=1}^{N}X_{n},\,d)$ by  
  \[\mu := \sum_{n=1}^{N} a_{n}(\iota_{n})_{*}\mu_{X_{n}},\]
where the map $\iota_{n}:X_{n} \ni\, x \mapsto x \in X_{n} \subset \bigsqcup_{n=1}^{N} X_{n}$, $n \in [N]$, is an isometric embedding.
We consider $X_{n}$ as a subspace of $(\bigsqcup_{n=1}^{N}X_{n}, d)$ and write $\mu_{X_{n}}$ instead of $(\iota_{n})_{*}\mu_{X_{n}}$ unless otherwise stated. 
We call the triple $(\bigsqcup_{n=1}^{N} X_{n},\, d,\, \mu)$ the \textit{direct sum of} $\{X_{n}\}_{n=1}^{N}$ \textit{with weight} $(a_{n})_{n=1}^{N}$ and write $\sum_{n=1}^{N} X_{n}^{a_{n}}$. 
If $N < \infty$, then we also denote the direct sum by $X_{1}^{a_{1}} + X_{2}^{a_{2}} + \cdots + X_{N}^{a_{N}}$. 
The direct sum of mm-spaces $\sum_{n=1}^{N}X_{n}^{a_{n}}$ is an extended mm-space. 

Since extended mm-spaces are separable, we see that any extended mm-space admits a unique representation as a direct sum of mm-spaces.
The following proposition is straightforward.

\begin{prop}\label{prop:decomposition of an ext. mm-sp and its uniqueness}
    For any extended mm-space $X$, there exist $(a_{n})_{n=1}^{N} \in \mathcal{A}_{1}$ and a sequence of mm-spaces $\{X_{n}\}_{n=1}^{N}$ such that the direct sum of mm-spaces $\sum_{n=1}^{N} X_{n}^{a_{n}}$ is mm-isomorphic to $X$. Moreover, if $\sum_{n=1}^{N} X_{n}^{a_{n}}$ and $\sum_{n=1}^{M} Y_{m}^{b_{m}}$ are mm-isomorphic to each other, then we have $N=M$, and there exists a bijective map $\varphi:[N] \to [N]$ such that the following (i) and (ii) are both satisfied.
    \begin{enumerate}
        \item $a_{n} = b_{\varphi(n)}$ for each $n \in [N]$.
        \item $X_{n}$ is mm-isomorphic to $Y_{\varphi(n)}$ for each $n \in [N]$. 
    \end{enumerate}
\end{prop}

We define the Lipschitz order relation $\prec$ between mm-spaces and extended mm-spaces in the same manner as for mm-spaces. 
Ozawa and Shioya~\cite[Section 5]{OS2015_limit_formula} prove that, for any extended mm-space $X$, \[\mathcal{P}_{X} := \{Y \in \mathcal{X} \mid Y \prec X\}\]is a pyramid.

For any $(a_{n})_{n=1}^{N} \in \mathcal{A}_{1}$ and any sequence of mm-spaces $\{X_{n}\}_{n=1}^{N}$, we observe that 
\begin{equation}\label{eq:the representation of a pyramid associated with an extended mm-space}
    \mathcal{P}_{\sum_{n=1}^{N} X_{n}^{a_{n}}} = \left\{X \in \mathcal{X}\,\middle\vert \,\begin{aligned} &\text{For}\hspace{1mm}n \in [N],\hspace{1mm}\text{there exist a 1-Lipschitz}\\&\text{map}\hspace{1mm}f_{n}:X_{n} \to X \hspace{1mm}\text{such that}\\&\mu_{X}=\textstyle\sum_{n=1}^{N} a_{n}\mu_{X_{n}} \end{aligned} \right\}.
\end{equation}

\subsection{Direct sum of pyramids and its properties}

\ 

In this subsection, we define the direct sum of pyramids and study its properties.
For example, we obtain the continuity of the direct sum operation (Theorem~\ref{thm:the continuity of the direct sum of pyramids}).
Furthermore, we prove some algebraic properties, such as the associative law (Proposition~\ref{prop:the associative law}) and the distributive law (Proposition~\ref{prop:the distributive law}).
\begin{dfn}[Direct sum of pyramids]\label{def:direct sum of pyramids in Section 3}
For $(a_{n})_{n=1}^{N} \in \mathcal{A}_{1}$ and a sequence of pyramids $\{\mathcal{P}_{n}\}_{n=1}^{N}$, we set \[\sum_{n=1}^{N} \mathcal{P}_{n}^{a_{n}} := \left\{X \in \mathcal{X} \,\middle|\, \begin{aligned} &\text{For}\hspace{1mm}n \in [N],\hspace{1mm}\text{there exist an mm-space}\hspace{1mm} X_{n} \in \mathcal{P}_{n}\\ &\text{and a 1-Lipschitz map}\hspace{1mm} f_{n}:X_{n} \to X\hspace{1mm}\text{such that}\\&\mu_{X} = \sum_{n=1}^{N} a_{n}(f_{n})_{*}\mu_{X_{n}} \end{aligned}\right\}.\]
We call $\textstyle\sum_{n=1}^{N}\mathcal{P}_{n}^{a_{n}}$ the \textit{direct sum of pyramids} $\{\mathcal{P}_{n}\}_{n=1}^{N}$ \textit{with weight} $(a_{n})_{n=1}^{N}$. 
If $N < \infty$, then we also write $\mathcal{P}_{1}^{a_{1}} + \mathcal{P}_{2}^{a_{2}}+ \cdots + \mathcal{P}_{N}^{a_{N}}$ instead of $\textstyle\sum_{n=1}^{N}\mathcal{P}_{n}^{a_{n}}$.
\end{dfn}
\begin{rem}
    The notion of a direct sum of pyramids is a natural generalization of pyramids generated by atoms observed in~\cite{EKM2024, KNS2024} (see Proposition~\ref{prop:pyramids generated by atoms}).
\end{rem}
From the equation~\eqref{eq:the representation of a pyramid associated with an extended mm-space}, we easily see that for $(a_{n})_{n=1}^{N} \in \A_{1}$ and a sequence of pyramids $\{\mathcal{P}_{n}\}_{n=1}^{N}$, 
\begin{equation}\label{eq:another representation of direct sum of pyramids}
    \sum_{n=1}^{N}\mathcal{P}_{n}^{a_{n}} = \bigcup_{(X_{n})_{n=1}^{N} \in \prod_{n=1}^{N}\mathcal{P}_{n}} \mathcal{P}_{\sum_{n=1}^{N} X_{n}^{a_{n}}}.
\end{equation}

\begin{proof}[Proof of Theorem~\ref{thm:direct sum of pyramids is a pyramid}]
    We set $\mathcal{P} := \sum_{n=1}^{N} \mathcal{P}_{n}^{a_{n}}$, which is not empty since the one-point mm-space $*$ belongs to $\mathcal{P}$. 
    We next prove that $\mathcal{P}$ is a directed subset of $\mathcal{X}$.
    Take any $X, Y \in \mathcal{P}$. 
    By the formula~\eqref{eq:another representation of direct sum of pyramids}, there exist mm-spaces $X_{n},\, Y_{n} \in \mathcal{P}_{n}$, $n \in [N]$, such that $X \prec \sum_{n=1}^{N} X_{n}^{a_{n}}$ and $Y \prec \sum_{n=1}^{N} Y_{n}^{a_{n}}$. 
    Since $\mathcal{P}_{n}$ is a pyramid, there exists $Z_{n} \in \mathcal{P}_{n}$ such that $X_{n} \prec Z_{n}$ and $Y_{n} \prec Z_{n}$.
    Then we have $X \prec \sum_{n=1}^{N} X_{n}^{a_{n}} \prec \sum_{n=1}^{N} Z_{n}^{a_{n}}$ and $Y \prec \sum_{n=1}^{N} Y_{n}^{a_{n}} \prec \sum_{n=1}^{N} Z_{n}^{a_{n}}$, which implies $X, Y \in \mathcal{P}_{\sum_{n=1}^{N}Z_{n}^{a_{n}}}$. 
    Since $\mathcal{P}_{\sum_{n=1}^{N}Z_{n}^{a_{n}}}$ is a pyramid, there exists an mm-space $Z \in \mathcal{P}_{\sum_{n=1}^{N}Z_{n}^{a_{n}}}$ such that $X \prec Z$ and $Y \in Z$. 
    By the definition of $\mathcal{P}$, we have $Z \in \mathcal{P}_{\sum_{n=1}^{N}Z_{n}^{a_{n}}} \subset \mathcal{P}$.
    This shows that $\mathcal{P}$ is directed with respect to the Lipschitz order.

    Finally, we prove that $\mathcal{P}$ is a closed subset of $\mathcal{X}$ with respect to the box topology.
    Take a sequence $\{X_{k}\}_{k=1}^{\infty}$ of $\mathcal{P}$ that converges to $X \in \mathcal{X}$ in the box topology. 
    By Proposition~\ref{prop:box dist. <-> epsillon-mm-isom.}, there exist real numbers $\varepsilon_{k} \to 0+$ and $\varepsilon_{k}$-mm-isomorphisms $p_{k}:X_{k} \to X$.
    By the definition of $\mathcal{P}$, we find mm-spaces $Y_{nk} \in \mathcal{P}_{n}$ and 1-Lipschitz maps $f_{nk}:Y_{nk} \to X_{k}$ for $n \in [N]$ and $k \in \mathbb{N}$ such that $\mu_{X_{k}} = \sum_{n=1}^{N} a_{n}(f_{nk})_{*}\mu_{Y_{nk}}$. 
    Since we observe that, for each $n$, 
    \begin{align*}
        a_{n}(p_{k}\circ f_{nk})_{*}\mu_{Y_{nk}} &\le \sum_{n=1}^{N} a_{n}(p_{k}\circ f_{nk})_{*}\mu_{Y_{nk}}= (p_{k})_{*}\left(\sum_{n=1}^{N} a_{n}(f_{nk})_{*}\mu_{Y_{nk}}\right)\\
        &= (p_{k})_{*}\mu_{X_{k}}
    \end{align*}
    and since $\{(p_{k})_{*}\mu_{X_{k}}\}_{k=1}^{\infty}$ is a weakly convergent sequence, we see that the set $\{(p_{k}\circ f_{nk})_{*}\mu_{Y_{nk}}\mid k \in \N\}$ is tight, and by Prokhorov's theorem (see~\cite[Theorem 5.1]{Bill}), this set is relatively compact with respect to $\dP$.
    By a diagonal argument, we obtain a subsequence $\{(p_{k(\ell)}\circ f_{nk(\ell)})_{*}\mu_{Y_{nk(\ell)}}\}_{\ell=1}^{\infty}$ of $\{(p_{k}\circ f_{nk})_{*}\mu_{Y_{nk}}\}_{k=1}^{\infty}$ which is a weakly convergent subsequence for each $n$. 
    We denote this limit by $\nu_{n}$, a Borel probability measure on $X$. 
    For $n \in [N]$ and $\ell \in \N$, we set \[Y_{n\ell}^{\prime} := (X_{k(\ell)},\, d_{X_{k(\ell)}},\, (f_{nk(\ell)})_{*}\mu_{Y_{nk(\ell)}}).\]
    Since $Y_{n\ell}^{\prime} \prec Y_{nk(\ell)}$, we have $Y_{n\ell}^{\prime} \in \mathcal{P}_{n}$.
    Considering $p_{k(\ell)}$ as the map from $Y_{n\ell}^{\prime}$ to $(X,\, d_{X},\, \nu_{n})$, we see that $p_{k(\ell)}$ is an $\varepsilon_{\ell}^{\prime}$-mm-isomorphism, where \[\varepsilon_{\ell}^{\prime} := \max\{\varepsilon_{k(\ell)}/a_{n},\, d_{\mathrm{P}}((p_{k(\ell)}\circ f_{nk(\ell)})_{*}\mu_{Y_{nk(\ell)}},\, \nu_{n})\} \to 0+.\] 
    By Proposition~\ref{prop:box dist. <-> epsillon-mm-isom.}, this implies that for any $n \in [N]$, $\{Y_{n\ell}^{\prime}\}_{\ell=1}^{\infty}$ converges to $(X,\, d_{X},\, \nu_{n})$ in the box topology, and we have $(X,\, d_{X},\, \nu_{n}) \in \mathcal{P}_{n}$ because $\mathcal{P}_{n}$ is a closed subset of $\mathcal{X}$ with respect to the box topology.
    Since $\{(p_{k(\ell)} \circ f_{nk(\ell)})_{*}\mu_{Y_{nk(\ell)}}\}_{\ell=1}^{\infty}$ converges weakly to $\nu_{n}$, the sequence $\{(p_{k(\ell)})_{*}\mu_{X_{k(\ell)}}=\sum_{n=1}^{N} a_{n}(p_{k(\ell)} \circ f_{nk(\ell)})_{*}\mu_{Y_{nk(\ell)}}\}_{\ell=1}^{\infty}$ converges weakly to $\sum_{n=1}^{N} a_{n}\nu_{n}$ by Corollary~\ref{cor:continuity of convex combination of probability measures}, and we obtain $\mu_{X} = \sum_{n=1}^{N} a_{n}\nu_{n}$, which implies $X \in \mathcal{P}$. 
    This completes the proof.
\end{proof}

\begin{rem}
    The direct sum of pyramids has the following connections with both extended mm-spaces and pyramids generated by atoms.
    
    For any $(a_{n})_{n=1}^{N} \in \mathcal{A}_{1}$ and any sequence of mm-spaces $\{X_{n}\}_{n=1}^{N}$, we have \[\mathcal{P}_{\sum_{n=1}^{N} X_{n}^{a_{n}}} = \sum_{n=1}^{N} \mathcal{P}_{X_{n}}^{a_{n}}.\] 
    Moreover, for any $A = (a_{n})_{n=1}^{N} \in \mathcal{A}$, we see that 
    \begin{align}\label{eq:pyramid generated by atoms is a direct sum of pyramids}
    \mathcal{P}_{A} =\begin{cases}
                        \X &\text{if $\|A\|_{1}=0$,}\\
                         \mathcal{X}^{1-\|A\|_{1}} + \left(\sum_{n=1}^{N} \{*\}^{a_{n}/\|A\|_{1}}\right)^{\|A\|_{1}} &\text{if $0 <\|A\|_{1} <1$,} \\
                         \sum_{n=1}^{N}\{*\}^{a_{n}} &\text{if $\|A\|_{1}=1$.}
                    \end{cases}
    \end{align}
\end{rem}

We now introduce a family of mm-spaces suitable for approximating the direct sum of pyramids.
For $A = (a_{n})_{n=1}^{N} \in \mathcal{A}_{1}$ and a sequence of mm-spaces $\{X_{n}\}_{n=1}^{N}$, and for $r > 0$, we set \[\mathcal{X}(\{X_{n}\}_{n=1}^{N};A;r) := \left\{X \in \mathcal{X}\, \middle\vert\, \begin{aligned}&\text{There exist isometric embeddings}\\ &\varphi_{n}:X_{n} \to X \hspace{1mm}, n \in [N],\hspace{1mm}\text{such that}\\ &\mu_{X} = \textstyle\sum_{n=1}^{N}a_{n}(\varphi_{n})_{*}\mu_{X_{n}}\hspace{1mm} \text{and}\\ &d_{X}(\varphi_{n}(X_{n}),\, \varphi_{m}(X_{m})) \ge r \hspace{1mm}\text{hold}\\ & \text{for any}\hspace{1mm} n,m\,\text{with}\hspace{1mm} n \neq m  \end{aligned}  \right\}. \]
In particular, if $N=2$, then we write $\mathcal{X}((X_{1}, X_{2}); (a_{1}, a_{2});r)$ instead of $\mathcal{X}(\{X_{n}\}_{n=1}^{2};A;r)$. 
We note that the set $\mathcal{X}(\{X_{n}\}_{n=1}^{N};A;r)$ is not empty. 
For example, take points $\bar{x}_{n} \in X_{n}$, $n \in [N]$, fix them, and define a triple $(\textstyle\sum_{n=1}^{N}(X_{n}, \bar{x}_{n})^{a_{n}})_{r}:=(\bigsqcup_{n=1}^{N}X_{n},\, d,\, \mu)$ by
\begin{align*}
    d(x,y) := \begin{cases}
                    d_{X_{n}}(x,y) &\text{if $x, y \in X_{n}$ for $n \in [N]$}, \\
                    d_{X_{n}}(x, \bar{x}_{n}) + d_{X_{m}}(y, \bar{x}_{m}) + r &\text{if $x \in X_{n}, y \in X_{m}$ with $n \neq m$,}
            \end{cases}
\end{align*}
and \[\mu := \sum_{n=1}^{N} a_{n}\mu_{X_{n}}. \]
Then, we see that \[(\textstyle\sum_{n=1}^{N}(X_{n},\bar{x}_{n})^{a_{n}})_{r} \in \mathcal{X}(\{X_{n}\}_{n=1}^{N};A;r).\]
When considering the limit $r \to \infty$, we will see later that the asymptotic behavior of $(\textstyle\sum_{n=1}^{N}(X_{n}, \bar{x}_{n})^{a_{n}})_{r}$ does not depend on the choice of base points $\{\bar{x}_{n}\}_{n=1}^{N}$ (see Corollary~\ref{cor:the sum of Xnk's approximates the sum of pyramids}). 
Thus, we sometimes ignore the base points and write $(\textstyle\sum_{n=1}^{N}X_{n}^{a_{n}})_{r}$. 
For the same reason, whenever the symbol $(\textstyle\sum_{n=1}^{N}X_{n}^{a_{n}})_{r}$ appears, it is understood that base points $\bar{x}_{n} \in X_{n}$ have been chosen.

Take any $Z \in \mathcal{X}(\{X_{n}\}_{n=1}^{N};A;r)$ and choose isometric embeddings $\varphi_{n}:X_{n} \to Z$, $n \in [N]$, that appear in the definition of $\mathcal{X}(\{X_{n}\}_{n=1}^{N};A;r)$. 
Then we see that 
\begin{itemize}
    \item $\mathrm{supp}\,(\varphi_{n})_{*}\mu_{X_{n}} = \varphi_{n}(X_{n})$ for each $n$;
    \item $\mathrm{supp}\,\mu_{Z} = \bigsqcup_{n=1}^{N}\mathrm{supp}\,(\varphi_{n})_{*}\mu_{X_{n}} = \bigsqcup_{n=1}^{N} \varphi_{n}(X_{n})$.
\end{itemize}
These properties follow from the fact that $d_{Z}(\varphi_{n}(X_{n}),\, \varphi_{m}(X_{m})) \ge r$ for $n,m$ with $n \neq m$.

\begin{lem}\label{lem:dist. between the measurement of the sum of mm-spaces with r}
    Let $A=(a_{n})_{n=1}^{N}$, $B=(b_{n})_{n=1}^{N^{\prime}}$ be elements of $\mathcal{A}_{1}$. Let $\{X_{n}\}_{n=1}^{N}$ be a sequence of mm-spaces and $\{\mathcal{P}_{n}\}_{n=1}^{N^{\prime}}$ a sequence of pyramids. 
    Then, for any integer $M \in [\min\{N, N^{\prime}\}]$, any real number $r > 0$, any positive integer $k$ with $k \le r/2$, and any $Z \in \mathcal{X}(\{X_{n}\}_{n=1}^{N};A;r)$, we have 
    \begin{align*}
        &(d_{\mathrm{P}})_{\mathrm{H}}\left(\mathcal{M}\left(Z;\, k,k\right),\, \mathcal{M}\left(\sum_{n=1}^{N^{\prime}} \mathcal{P}_{n}^{b_{n}};\, k,k\right) \right) \\ &\hspace{-2mm}\le \max_{n=1}^{M}\,(d_{\mathrm{P}})_{\mathrm{H}}(\mathcal{M}(X_{n};k,k),\, \mathcal{M}(\mathcal{P}_{n};k,k)) + \|A-B\|_{1} + \sum_{n=M+1}^{N} a_{n} +\sum_{n=M+1}^{N^{\prime}}b_{n}.
    \end{align*}
\end{lem}

\begin{proof}
     By the definition of $Z$, there exist isometric embeddings $\varphi_{n}:X_{n} \to Z$, $n \in [N]$, with the required properties. 
     We fix such maps $\varphi_{n}$. 
     Set \[\varepsilon := \max_{n=1}^{M} (d_{\mathrm{P}})_{\mathrm{H}}(\mathcal{M}(X_{n};k,k),\, \mathcal{M}(\mathcal{P}_{n};k,k)) + \|A-B\|_{1} + \sum_{n=M+1}^{N} a_{n} + \sum_{n=M+1}^{N^{\prime}}b_{n}.\] 
     Take any number $\delta_{n}$ with $\delta_{n} > (d_{\mathrm{P}})_{\mathrm{H}}(\mathcal{M}(X_{n};\,k,k),\, \mathcal{M}(\mathcal{P}_{n};\,k,k))$ for $n = 1,2,\cdots,M$. 
     First, we prove $\mathcal{M}(Z;\, k,k) \subset B_{\varepsilon}(\mathcal{M}(\sum_{n=1}^{N^{\prime}}\mathcal{P}_{n}^{b_{n}};\, k,k))$. 
     Take any $\mu \in \mathcal{M}(Z;\, k,k)$. 
     Then there exists a 1-Lipschitz map $F:Z \to (B_{k}^{k},\, \|\cdot\|_{\infty})$ such that $\mu = F_{*}\mu_{Z}$. 
     Since $\mu_{n} := (F \circ \varphi_{n})_{*}\mu_{X_{n}} \in \mathcal{M}(X_{n};\, k,k)$, there exists a Borel probability measure $\nu_{n} \in \mathcal{M}(\mathcal{P}_{n};\, k,k)$ such that $d_{\mathrm{P}}(\mu_{n},\, \nu_{n}) < \delta_{n}$. 
     For $n$ with $M < n \le N^{\prime}$, define $\nu_{n} := \delta_{0}$, the Dirac measure at $0 \in B_{k}^{k}$ and set $\nu := \sum_{n=1}^{N^{\prime}} b_{n}\nu_{n}$. 
     Then we see that $(B_{k}^{k},\, \|\cdot\|_{\infty},\, \nu) \in \textstyle\sum_{n=1}^{N^{\prime}}\mathcal{P}_{n}^{b_{n}}$, that is, $\nu \in \mathcal{M}(\sum_{n=1}^{N^{\prime}}\mathcal{P}_{n}^{b_{n}};\, k,k)$. 
     Moreover, by Lemma~\ref{lem:the estimation of dist. between convex combinations of probability measures}, we have 
    \begin{align*}
        d_{\mathrm{P}}(\mu,\, \nu) &= d_{\mathrm{P}}\left(\sum_{n=1}^{N}a_{n}\mu_{n},\, \sum_{n=1}^{N^{\prime}}b_{n}\nu_{n}\right) \\
        &\le \max_{n=1}^{M}\,d_{\mathrm{P}}(\mu_{n},\, \nu_{n}) + \|A-B\|_{1} + \sum_{n=M+1}^{N}a_{n} + \sum_{n=M+1}^{N^{\prime}}b_{n} \\
        &< \max_{n=1}^{M}\, \delta_{n} + \|A-B\|_{1} + \sum_{n=M+1}^{N}a_{n} + \sum_{n=M+1}^{N^{\prime}}b_{n} =: \varepsilon^{\prime}, 
    \end{align*}
    which implies $\mu \in U_{\varepsilon^{\prime}}(\mathcal{M}(\textstyle\sum_{n=1}^{N^{\prime}}\mathcal{P}_{n}^{b_{n}};\, k,k))$. 
    Since the number $\varepsilon^{\prime}$ does not depend on $\mu$, we have $\mathcal{M}(Z;\, k,k) \subset U_{\varepsilon^{\prime}}\left(\mathcal{M}\left(\sum_{n=1}^{N^{\prime}}\mathcal{P}_{n}^{b_{n}};\, k,k \right)\right)$. 
    Letting $\varepsilon^{\prime} \to \varepsilon+0$, we obtain $\mathcal{M}(Z;\, k,k) \subset B_{\varepsilon}\left(\mathcal{M}\left(\sum_{n=1}^{N^{\prime}}\mathcal{P}_{n}^{b_{n}};\, k,k \right)\right)$.

    Next, we prove $\mathcal{M}\left(\textstyle\sum_{n=1}^{N^{\prime}}\mathcal{P}_{n}^{b_{n}};\,k,k\right) \subset B_{\varepsilon}(\mathcal{M}(Z;\,k,k))$. 
    Take any $\mu \in \mathcal{M}(\sum_{n=1}^{N^{\prime}}\mathcal{P}_{n}^{b_{n}};\, k,k)$. 
    Since $(B_{k}^{k},\, \|\cdot\|_{\infty},\, \mu) \in \sum_{n=1}^{N^{\prime}}\mathcal{P}_{n}^{b_{n}}$, there exist mm-spaces $Z_{n} \in \mathcal{P}_{n}$ and 1-Lipschitz maps $f_{n}:Z_{n} \to (B_{k}^{k},\, \|\cdot\|_{\infty})$ for $n \in [N^{\prime}]$ such that $\mu = \sum_{n=1}^{N^{\prime}}b_{n}(f_{n})_{*}\mu_{Z_{n}}$. 
    Moreover, since $(f_{n})_{*}\mu_{Z_{n}} \in \mathcal{M}(\mathcal{P}_{n};\,k,k)$, we find $\nu_{n} \in \mathcal{M}(X_{n};\,k,k)$, $n = 1,2,\cdots,M$, such that $d_{\mathrm{P}}(\nu_{n},\, (f_{n})_{*}\mu_{Z_{n}}) < \delta_{n}$. 
    Take a 1-Lipschitz map $\Phi_{n}:X_{n} \to (B_{k}^{k},\, \|\cdot\|_{\infty})$ such that $\nu_{n} = (\Phi_{n})_{*}\mu_{X_{n}}$. 
    Define a map $\Phi:Z \to (B_{k}^{k},\, \|\cdot\|_{\infty})$ by \[\Phi(z) := \begin{cases} \Phi_{n}(\varphi_{n}^{-1}(z)) &\text{if $z \in \varphi_{n}(X_{n})$ for $1\le n \le M$}, \\ 0 &\text{if $z \in \varphi_{n}(X_{n})$ for $M < n \le N$}.\end{cases}\] 
    Since $2k \le r$, the map $\Phi$ is 1-Lipschitz.
    We set $\nu := \Phi_{*}\mu_{Z}$, then $\nu \in \mathcal{M}(Z;\,k,k)$.
    We here note that \[\nu = \sum_{n=1}^{N}a_{n}(\Phi\circ \varphi_{n})_{*}\mu_{X_{n}} = \sum_{n=1}^{M}a_{n}(\Phi_{n})_{*}\mu_{X_{n}}+\sum_{n=M+1}^{N}a_{n}\delta_{0}.\]
    By Lemma~\ref{lem:the estimation of dist. between convex combinations of probability measures}, we see that 
    \begin{align*}
        d_{\mathrm{P}}(\nu,\, \mu) &= d_{\mathrm{P}}\left(\sum_{n=1}^{N}a_{n}(\Phi \circ \varphi_{n})_{*}\mu_{X_{n}},\, \sum_{n=1}^{N^{\prime}}b_{n}(f_{n})_{*}\mu_{Z_{n}}\right) \\
        &\le \max_{n=1}^{M}\, d_{\mathrm{P}}((\Phi_{n})_{*}\mu_{X_{n}},\, (f_{n})_{*}\mu_{Z_{n}}) + \|A-B\|_{1} + \sum_{n=M+1}^{N} a_{n} + \sum_{n=M+1}^{N^{\prime}}b_{n} \\
        &< \max_{n=1}^{M}\, \delta_{n} + \|A-B\|_{1} + \sum_{n=M+1}^{N} a_{n} + \sum_{n=M+1}^{N^{\prime}}b_{n} =:\varepsilon^{\prime\prime},
    \end{align*}
    which implies that $\mu \in U_{\varepsilon^{\prime\prime}}(\mathcal{M}(Z;\,k,k))$. 
    Since $\varepsilon^{\prime\prime}$ does not depend on $\mu$, we have $\mathcal{M}(\sum_{n=1}^{N^{\prime}}\mathcal{P}_{n}^{b_{n}};\,k,k) \subset U_{\varepsilon^{\prime\prime}}(\mathcal{M}(Z;\,k,k))$. 
    Letting $\varepsilon^{\prime\prime} \to \varepsilon+0$, we obtain $\mathcal{M}\left(\sum_{n=1}^{N}\mathcal{P}_{n}^{b_{n}};\, k,k\right) \subset B_{\varepsilon}\left(\mathcal{M}(Z;\, k,k \right))$. 
    This completes the proof.
\end{proof}

\begin{lem}\label{lem:dist. between the sum of pyramids and the sum of mm-spaces with r}
        Suppose that $A=(a_{n})_{n=1}^{N}$, $B=(b_{n})_{n=1}^{N^{\prime}}$ are elements of $\mathcal{A}_{1}$. 
        Let $\{X_{n}\}_{n=1}^{N}$ be a sequence of mm-spaces and $\{\mathcal{P}_{n}\}_{n=1}^{N^{\prime}}$ a sequence of pyramids. 
        Then, for any integer $M \in [\min\{N, N^{\prime}\}]$, any real number $r > 0$, and any $Z \in \mathcal{X}(\{X_{n}\}_{n=1}^{N};A;r)$, we have \[\rho\left(\mathcal{P}_{Z},\, \sum_{n=1}^{N^{\prime}} \mathcal{P}_{n}^{b_{n}}\right) \le \sum_{n=1}^{M}\, \rho(\mathcal{P}_{X_{n}},\, \mathcal{P}_{n}) + \frac{1}{2}\|A-B\|_{1} + \frac{1}{2}\sum_{n=M+1}^{N} a_{n} +\frac{1}{2}\sum_{n=M+1}^{N^{\prime}} b_{n} + \frac{1}{2^{r/2}}.\]
\end{lem}

\begin{proof}
    Let $m_{r} \in \mathbb{N}$ be the largest integer less than or equal to $r/2$.
    By Lemma~\ref{lem:dist. between the measurement of the sum of mm-spaces with r}, we have 
    \begin{align*}
        &\rho\left(\mathcal{P}_{Z},\, \sum_{n=1}^{N^{\prime}} \mathcal{P}_{n}^{b_{n}}\right) =\sum_{k=1}^{\infty} \frac{1}{2^{k}\cdot 2k}(d_{\mathrm{P}})_{\mathrm{H}}\left(\mathcal{M}(Z\,;k,k),\, \mathcal{M}\left(\sum_{n=1}^{N^{\prime}}\mathcal{P}_{n}^{b_{n}}\,;k,k \right)\right)\\
        &\begin{aligned} &\le \sum_{k=1}^{m_{r}} \frac{1}{2^{k}\cdot 2k}\left\{\max_{n=1}^{M}\, (d_{\mathrm{P}})_{\mathrm{H}}(\mathcal{M}(X_{n}\,;k,k),\, \mathcal{M}(\mathcal{P}_{n}\,;k,k)) \right. \\&\hspace{20mm}+ \|A- B\|_{1} + \sum_{n=M+1}^{N} a_{n} + \sum_{n=M+1}^{N^{\prime}} b_{n}\bigg\}+\sum_{k=m_{r}+1}^{\infty}\frac{1}{2^{k+1}} \end{aligned} \\
        &\begin{aligned} &\le \sum_{k=1}^{m_{r}} \sum_{n=1}^{M}\frac{1}{2^{k}\cdot 2k}(d_{\mathrm{P}})_{\mathrm{H}}(\mathcal{M}(X_{n}\,;k,k),\, \mathcal{M}(\mathcal{P}_{n}\,;k,k)) \\ &\hspace{15mm} + \sum_{k=1}^{m_{r}} \frac{1}{2^{k}\cdot 2k}\left\{\|A- B\|_{1} + \sum_{n=M+1}^{N} a_{n} + \sum_{n=M+1}^{N^{\prime}} b_{n} \right\} + \frac{1}{2^{m_{r}+1}}\end{aligned}\\
        &\le \sum_{n=1}^{M}\, \rho(\mathcal{P}_{X_{n}},\, \mathcal{P}_{n}) + \frac{1}{2}\|A-B\|_{1} + \frac{1}{2}\sum_{n=M+1}^{N} a_{n} + \frac{1}{2}\sum_{n=M+1}^{N^{\prime}} b_{n} +\frac{1}{2^{r/2}}.
    \end{align*}
    This completes the proof.
\end{proof}
By Lemma~\ref{lem:dist. between the sum of pyramids and the sum of mm-spaces with r}, we obtain the following two corollaries.
\begin{cor}\label{cor:the sum of Xnk's approximates the sum of pyramids}
    Let $A=(a_{n})_{n=1}^{N}$, $A_{k}=(a_{nk})_{n=1}^{N}$, $k \in \mathbb{N}$, be elements of $\mathcal{A}_{1}$, and let $\{\mathcal{P}_{n}\}_{n=1}^{N}$ be a sequence of pyramids and $\{X_{nk}\}_{k=1}^{\infty}$ a sequence of mm-spaces for $n \in [N]$. 
    If $A_{k}$ converges to $A$ in the $\ell^{1}$-norm and $\mathcal{P}_{X_{nk}}$ converges weakly to $\mathcal{P}_{n}$ for each $n$ as $k \to \infty$, then, for any sequence of real numbers $\{r(k)\}_{k=1}^{\infty}$ divergent to infinity and any $Z_{k} \in \mathcal{X}(\{X_{nk}\}_{n=1}^{N};A_{k};r(k))$, $k \in \mathbb{N}$, the sequence $\{\mathcal{P}_{Z_{k}}\}_{k=1}^{\infty}$ converges weakly to the pyramid $\sum_{n=1}^{N} \mathcal{P}_{n}^{a_{n}}$.
\end{cor}

\begin{proof}
    Take any $\varepsilon > 0$ and fix it. 
    There exists an integer $M \in [N]$ such that $\sum_{n=M+1}^{N} a_{n} < \varepsilon/4$. 
    Then we obtain a large integer $k_{0} \ge 1$ in such a way that for any $k \ge k_{0}$, we have $2^{-r(k)/2} < \varepsilon/4$,  $\|A-A_{k}\|_{1} \le \varepsilon/4$, and $\rho(\mathcal{P}_{X_{nk}},\, \mathcal{P}_{n}) < \varepsilon/4M$ for any $n = 1,2,\cdots, M$. 
    By Lemma~\ref{lem:dist. between the sum of pyramids and the sum of mm-spaces with r}, we see that for any $k \ge k_{0}$, 
    \begin{align*}
        &\rho\left(\mathcal{P}_{Z_{k}},\, \sum_{n=1}^{N}\mathcal{P}_{n}^{a_{n}}\right)\\ &\le \sum_{n=1}^{M}\rho(\mathcal{P}_{X_{nk}},\, \mathcal{P}_{n}) + \frac{1}{2}\|A_{k} -A\|_{1} + \frac{1}{2}\sum_{n=M+1}^{N}a_{n} + \frac{1}{2}\sum_{n=M+1}^{N}a_{nk}+ \frac{1}{2^{r(k)/2}} \\
        &\le \sum_{n=1}^{M} \rho(\mathcal{P}_{X_{nk}},\, \mathcal{P}_{n}) + \|A_{k}-A\|_{1} + \sum_{n=M+1}^{N}a_{n} + \frac{1}{2^{r(k)/2}} \\
        &\le \sum_{n=1}^{M}\frac{\varepsilon}{4M} + \frac{\varepsilon}{4} + \frac{\varepsilon}{4}+ \frac{\varepsilon}{4} = \varepsilon,   
    \end{align*}
    which completes the proof.
\end{proof}

\begin{cor}\label{cor:dist. between direct sums of pyramids}
    Let $A=(a_{n})_{n=1}^{N}$, $B=(b_{n})_{n=1}^{N^{\prime}}$ be elements of $\mathcal{A}_{1}$, and let $\{\mathcal{P}_{n}\}_{n=1}^{N}$, $\{\mathcal{Q}_{n}\}_{n=1}^{N^{\prime}}$ be sequences of pyramids. Then, for any integer $M \in [\min\{N,N^{\prime}\}]$, we have \[\rho\left(\sum_{n=1}^{N} \mathcal{P}_{n}^{a_{n}},\, \sum_{n=1}^{N^{\prime}} \mathcal{Q}_{n}^{b_{n}}\right) \le \sum_{n=1}^{M} \rho(\mathcal{P}_{n},\, \mathcal{Q}_{n}) + \frac{1}{2}\|A-B\|_{1} + \frac{1}{2}\sum_{n=M+1}^{N} a_{n} + \frac{1}{2}\sum_{n=M+1}^{N^{\prime}} b_{n}.\]    
\end{cor}
\begin{proof}
    For each $n \in [N^{\prime}]$, by Lemma~\ref{lem:app. seq. of pyramid} there exists an approximation sequence $\{X_{nk}\}_{k=1}^{\infty}$ of the pyramid $\mathcal{Q}_{n}$.
    Take $Z_{k} \in \mathcal{X}(\{X_{nk}\}_{n=1}^{N^{\prime}};B;k)$ for each $k \in \mathbb{N}$. 
    By Lemma~\ref{lem:dist. between the sum of pyramids and the sum of mm-spaces with r}, we see that 
    \begin{equation}\label{eq:cor:dist. between direct sums of pyramids}
        \begin{aligned}
            &\rho\left(\sum_{n=1}^{N}\mathcal{P}_{n}^{a_{n}},\, \mathcal{P}_{Z_{k}}\right) \\
            &\le \sum_{n=1}^{M}\,\rho(\mathcal{P}_{n},\, \mathcal{P}_{X_{nk}}) + \frac{1}{2}\|A-B\|_{1}+\frac{1}{2}\sum_{n=M+1}^{N}a_{n}+\frac{1}{2}\sum_{n=M+1}^{N^{\prime}}b_{n} + \frac{1}{2^{k/2}}.
        \end{aligned}
    \end{equation}
    Since $\{P_{X_{nk}}\}_{k=1}^{\infty}$ converges weakly to $\mathcal{Q}_{n}$, the sequence $\{\mathcal{P}_{Z_{k}}\}_{k=1}^{\infty}$ converges weakly to the pyramid $\sum_{n=1}^{N^{\prime}}\mathcal{Q}_{n}^{b_{n}}$ by Corollary~\ref{cor:the sum of Xnk's approximates the sum of pyramids}. 
    Letting $k \to \infty$ in the inequality~\eqref{eq:cor:dist. between direct sums of pyramids}, we obtain this corollary.
\end{proof}
\begin{rem}
    It is easier to prove Corollary~\ref{cor:dist. between direct sums of pyramids} directly from Lemma~\ref{lem:the estimation of dist. between convex combinations of probability measures}. 
    The approach of its proof is the same as those of Lemma~\ref{lem:dist. between the measurement of the sum of mm-spaces with r} and Lemma~\ref{lem:dist. between the sum of pyramids and the sum of mm-spaces with r}.
\end{rem}

\begin{rem}\label{rem:cf. Lemma 6.28}
    The distance between two pyramids generated by atoms is estimated in~\cite[Lemma 6.28]{EKM2024}, that is, for any $A$, $B \in \mathcal{A}$, we have $\rho(\mathcal{P}_{A},\,\mathcal{P}_{B}) \le \|A-B\|_{1}$.
    Although the metric $\rho$ on $\Pi$ used in~\cite{EKM2024} differs from that in our paper, the same bound still follows from Corollary~\ref{cor:dist. between direct sums of pyramids}.
\end{rem}
From Corollary~\ref{cor:dist. between direct sums of pyramids}, we obtain this theorem immediately.
\begin{thm}\label{thm:the continuity of the direct sum of pyramids}
    Let $A=(a_{n})_{n=1}^{N}$, $A_{k}=(a_{nk})_{n=1}^{N}$, $k \in \mathbb{N}$, be elements of $\mathcal{A}_{1}$, and let $\{\mathcal{P}_{n}\}_{n=1}^{N}$, $\{\mathcal{P}_{nk}\}_{n=1}^{N}$, $k \in \mathbb{N}$, be sequences of pyramids. 
    If $\{A_{k}\}_{k=1}^{\infty}$ converges to $A$ in the $\ell^{1}$-norm, and $\{\mathcal{P}_{nk}\}_{k=1}^{N}$ converges weakly to the pyramid $\mathcal{P}_{n}$ for each $n \in [N]$, then the sequence of pyramids $\{\sum_{n=1}^{N} \mathcal{P}_{nk}^{a_{nk}}\}_{k=1}^{\infty}$ converges weakly to the direct sum $\sum_{n=1}^{N} \mathcal{P}_{n}^{a_{n}}.$
\end{thm}

\begin{rem}
    It is straightforward to verify that, for $A_{k}$, $A \in \mathcal{A}_{1}$, $k \in \N$, the following (i) and (ii) are equivalent to each other:
        \begin{enumerate}
            \item $\lim_{k\to\infty}\|A_{k} - A\|_{1}=0$;
            \item $\lim_{k\to\infty}\|A_{k} - A\|_{\infty} =0$.
        \end{enumerate}
    Here, $\|\cdot\|_{\infty}$ denotes the $\ell^{\infty}$-\textit{norm}, which is defined for $A=(a_{n})_{n=1}^{N}$, $B=(b_{n})_{n=1}^{M} \in \mathcal{A}_{1}$ by
    \begin{equation*}
        \|A-B\|_{\infty} := \begin{cases}
                                \max\{\max_{n=1}^{N}|a_{n}-b_{n}|,\; \max_{n=N+1}^{M}b_{n}\} &\text{if $N \le M$;}\\
                                \max\{\max_{n=1}^{M}|a_{n}-b_{n}|,\;\max_{n=M+1}^{N}a_{n}\} &\text{if $N \ge M$.}
                            \end{cases}
    \end{equation*}
    From this fact, it follows that the conclusion of Theorem~\ref{thm:the continuity of the direct sum of pyramids} also holds when $A_{k}$ converges to A in the $\ell^{\infty}$-norm.
\end{rem}

\begin{rem}
    For pyramids generated by atoms, a result similar to Theorem~\ref{thm:the continuity of the direct sum of pyramids} can be found in~\cite[Theorem 0.7]{EKM2024} and~\cite[Lemma 4.5]{KNS2024}.
    It states that for $A_{n}$, $A \in \mathcal{A}$, $n \in \N$, if $A_{n}$ converges to $A$ in the $\ell^{\infty}$-norm as $n \to \infty$, then $\mathcal{P}_{A_{n}}$ converges weakly to $\mathcal{P}_{A}$.
    We note that this result also follows from Corollary~\ref{cor:dist. between direct sums of pyramids}.
    We omit the alternative proof, as it is not directly relevant to the main results and is somewhat involved and technical.
\end{rem}
The next two propositions show the algebraic properties of the direct sum operation of pyramids.

\begin{dfn}\label{def:the definition of partitions}
    For a non-empty set $S$ and a family $\{S_{\lambda}\}_{\lambda \in \Lambda}$ of non-empty subsets of $S$, we say that $\{S_{\lambda}\}_{\lambda \in \Lambda}$ is a \textit{partition} of $S$ if it satisfies the following (i) and (ii).
    \begin{enumerate}
        \item $S_{\lambda} \cap S_{\mu} = \emptyset$ for any $\lambda, \mu \in \Lambda$ with $\lambda \neq \mu$.
        \item $S = \bigcup_{\lambda \in \Lambda} S_{\lambda}$.
    \end{enumerate}
\end{dfn}

\begin{prop}[Associative law]\label{prop:the associative law}
    Let $(a_{n})_{n=1}^{N} \in \mathcal{A}_{1}$, and let $\{\mathcal{P}_{n}\}_{n=1}^{N}$ be a sequence of pyramids. 
    For any partition $\{N_{k}\}_{k=1}^{M},\, M \in \overline{\mathbb{N}},$ of the set $[N]$, we have \[\sum_{n=1}^{N} \mathcal{P}_{n}^{a_{n}} = \sum_{k=1}^{M} \left(\sum_{n \in N_{k}} \mathcal{P}_{n}^{a_{n}/\alpha_{k}}\right)^{\alpha_{k}},\] where $\alpha_{k} := \sum_{n \in N_{k}} a_{n} > 0$ for $k \in [M]$.
\end{prop}

\begin{proof}
    First, we prove the inclusion $\subset$. 
    Take any $X \in \sum_{n=1}^{N} \mathcal{P}_{n}^{a_{n}}$. 
    There exist mm-spaces $X_{n} \in \mathcal{P}_{n}$ and 1-Lipschitz maps $f_{n}:X_{n} \to X$, $n \in [N]$, such that $\mu_{X} = \sum_{n=1}^{N} a_{n}(f_{n})_{*}\mu_{X_{n}}$. 
    For $k \in [M]$, set \[Y_{k} := \left(X,\, d_{X},\, \sum_{n \in N_{k}} \frac{a_{n}}{\alpha_{k}}(f_{n})_{*}\mu_{X_{n}}\right).\]
    We see that $Y_{k}$ is an mm-space and $Y_{k} \in \sum_{n \in N_{k}} \mathcal{P}_{n}^{a_{n}/\alpha_{k}}$. 
    For $k \in [M]$, define a map $g_{k}:Y_{k} \to X$ by $g_{k}(x) := x$.
    We observe that \[\sum_{k=1}^{M}\alpha_{k}(g_{k})_{*}\mu_{Y_{k}} = \sum_{k=1}^{M} \alpha_{k} \sum_{n \in N_{k}} \frac{a_{n}}{\alpha_{k}}(f_{n})_{*}\mu_{X_{n}} = \sum_{n=1}^{N} a_{n}(f_{n})_{*}\mu_{X_{n}} = \mu_{X}. \] 
    Since $g_{k}$ is 1-Lipschitz, this equality implies $X \in \sum_{k=1}^{M}(\sum_{n \in N_{k}} \mathcal{P}_{n}^{a_{n}/\alpha_{k}})^{\alpha_{k}}$, and we obtain the inclusion $\subset$.

    Next, we prove $\supset$. 
    Take any $X \in \sum_{k=1}^{M}(\sum_{n \in N_{k}} \mathcal{P}_{n}^{a_{n}/\alpha_{k}})^{\alpha_{k}}$. 
    For $k \in [M]$, there exist mm-spaces $Y_{k} \in \sum_{n \in N_{k}} \mathcal{P}_{n}^{a_{n}/\alpha_{k}}$ and 1-Lipschitz maps $q_{k}:Y_{k} \to X$ such that $\mu_{X} = \sum_{k=1}^{M} \alpha_{k}(q_{k})_{*}\mu_{Y_{k}}$. 
    Moreover, for $k \in [M]$ and $n \in N_{k}$, there exist mm-spaces $X_{n} \in \mathcal{P}_{n}$ and 1-Lipschitz maps $f_{nk}:X_{n} \to Y_{k}$ such that $\mu_{Y_{k}} =\sum_{n \in N_{k}} a_{n}/\alpha_{k}\,(f_{nk})_{*}\mu_{X_{n}}$. 
    Define maps $h_{nk}:X_{n} \to X$ by $ h_{nk} := q_{k}\circ f_{nk}$. 
    We see that $h_{nk}$ is 1-Lipschitz and \[\sum_{n=1}^{N} a_{n}(h_{nk})_{*}\mu_{X_{n}} = \sum_{k=1}^{M}\alpha_{k}\sum_{n \in N_{k}} \frac{a_{n}}{\alpha_{k}}(h_{nk})_{*}\mu_{X_{n}} = \sum_{k=1}^{M}\alpha_{k}(q_{k})_{*}\mu_{Y_{k}} = \mu_{X},\]
    which implies $X \in \sum_{n=1}^{N} \mathcal{P}_{n}^{a_{n}}$. 
    This completes the proof.
\end{proof}

\begin{lem}\label{lem:l^p product of extended mm-sp. is equal to the l^p product of pyramids}
    Suppose $p \in [1,\infty]$. 
    For any $(a_{n})_{n=1}^{N}$, $(b_{m})_{m=1}^{M} \in \mathcal{A}_{1}$, and for any sequences $\{X_{n}\}_{n=1}^{N}$ and $\{Y_{m}\}_{m=1}^{M}$ of mm-spaces, we have \[\mathcal{P}_{(\sum_{n=1}^{N} X_{n}^{a_{n}})\times_{p}(\sum_{m=1}^{M} Y_{m}^{b_{m}})} = \mathcal{P}_{\sum_{n=1}^{N} X_{n}^{a_{n}}} \times_{p} \mathcal{P}_{\sum_{m=1}^{M} Y_{m}^{b_{m}}}.\]
\end{lem}

\begin{proof}
    It is clear that the inclusion $\supset$ holds.
    Thus we prove the opposite inclusion $\subset$. 
    We set $\widetilde{X} := \sum_{n=1}^{N}X_{n}^{a_{n}}$, $\widetilde{Y} := \sum_{m=1}^{M}Y_{m}^{b_{m}}$, and $\widetilde{Z} := \sum_{n \in [N],\, m \in [M]}(X_{n}\times_{p} Y_{m})^{a_{n}b_{m}}$.
    Take any $H \in \mathcal{P}_{\widetilde{X} \times_{p} \widetilde{Y}}$, and for $k \in \mathbb{N}$, take any $C_{k} \in \mathcal{X}(\{X_{n}\}_{n=1}^{N};(a_{n})_{n=1}^{N};k)$ and any $D_{k} \in \mathcal{X}(\{Y_{m}\}_{m=1}^{M};(b_{m})_{m=1}^{M};k)$.
    Note that $C_{k} \in \mathcal{P}_{\widetilde{X}}$ and $D_{k} \in \mathcal{P}_{\widetilde{Y}}$.
    Take any bijective map $\varphi:[NM] \to [N] \times [M]$, $\varphi(n) =(\varphi_{1}(n),\,\varphi_{2}(m))$, and fix it.
    We define a sequence $E=(e_{n})_{n=1}^{NM}$ by $e_{n} := a_{\varphi_{1}(n)}b_{\varphi_{2}(n)}$ for $n \in [NM]$.
    Note that if $N=\infty$ or $M=\infty$, then we set $NM := \infty$.
    Then $E$ is a rearrangement of the terms $a_{n}b_{m}$ and belongs to $\mathcal{A}_{1}$. 
    We see that $C_{k} \times_{p} D_{k} \in \mathcal{X}(\{W_{n}\}_{n=1}^{NM};E;k)$, where $W_{n} := X_{\varphi_{1}(n)}\times_{p} Y_{\varphi_{2}(n)}$, $n \in [NM]$.
    Therefore, by Corollary~\ref{cor:the sum of Xnk's approximates the sum of pyramids}, $\mathcal{P}_{C_{k}\times_{p}D_{k}}$ converges weakly to the pyramid $\mathcal{P}_{\widetilde{Z}}$.
    Moreover, since extended mm-spaces $\widetilde{X}\times_{p}\widetilde{Y}$ and $\widetilde{Z}$ are mm-isomorphic to each other, we have $\mathcal{P}_{\widetilde{Z}} = \mathcal{P}_{\widetilde{X}\times_{p}\widetilde{Y}}$, and the pyramid $\mathcal{P}_{C_{k}\times_{p}D_{k}}$ converges weakly to $\mathcal{P}_{\widetilde{X}\times_{p}\widetilde{Y}}$.
    This implies that, by Lemma~\ref{lem:PXn -> P, Z in P => Z <ek Xnk}, there exists a subsequence $\{C_{k(\ell)} \times_{p} D_{k(\ell)}\}_{\ell=1}^{\infty}$ of $\{C_{k}\times_{p} D_{k}\}_{k=1}^{\infty}$ such that \[H \prec_{1/\ell} C_{k(\ell)} \times_{p} D_{k(\ell)}\] for any $\ell \in \N$, that is, $\mathcal{P}_{H} \prec_{1/\ell} \mathcal{P}_{\widetilde{X}} \times_{p} \mathcal{P}_{\widetilde{Y}}$.
    By Proposition~\ref{prop:Pn <en Qn => P < Q}, we have $\mathcal{P}_{H} \subset \mathcal{P}_{\widetilde{X}} \times_{p} \mathcal{P}_{\widetilde{Y}}$, and hence $H \in \mathcal{P}_{\widetilde{X}} \times_{p} \mathcal{P}_{\widetilde{Y}}$. 
    This completes the proof.
\end{proof}

\begin{prop}[Distributive law]\label{prop:the distributive law}
    For any $(a_{n})_{n=1}^{N}, (b_{n})_{n=1}^{M} \in \mathcal{A}_{1}$, any sequences of pyramids $\{\mathcal{P}_{n}\}_{n=1}^{N}$ and $\{\mathcal{Q}_{n}\}_{n=1}^{M}$, and any $p \in [1,\infty]$, we have \[\left(\sum_{n=1}^{N} \mathcal{P}_{n}^{a_{n}}\right) \times_{p} \left(\sum_{m=1}^{M} \mathcal{Q}_{m}^{b_{m}}\right) = \sum_{n \in [N],\, m \in [M]} (\mathcal{P}_{n} \times_{p}\mathcal{Q}_{m})^{a_{n}b_{m}}.\]
\end{prop}

\begin{proof}
    First, we prove the inclusion $\subset$. 
    By the definition of $\ell^{p}$-product $(\sum_{n=1}^{N}\mathcal{P}_{n}^{a_{n}})\times_{p}(\sum_{m=1}^{M}\mathcal{Q}_{m}^{b_{m}})$, it suffices to show that $X \times_{p} Y \in \sum_{n,m}(\mathcal{P}_{n}\otimes_{p}\mathcal{Q}_{m})^{a_{n}b_{m}}$ for any $X \in \sum_{n=1}^{N}\mathcal{P}_{n}^{a_{n}}$ and any $Y \in \sum_{m=1}^{M}\mathcal{Q}_{m}^{b_{m}}$. 
    For $n \in [N]$ and $m \in [M]$, there exist mm-spaces $X_{n} \in \mathcal{P}_{n},\, Y_{m} \in \mathcal{Q}_{m}$, and 1-Lipschitz maps $f_{n}:X_{n} \to X,\, g_{m}:Y_{m} \to Y$ such that $\mu_{X} = \sum_{n=1}^{N}a_{n}(f_{n})_{*}\mu_{X_{n}}$ and $\mu_{Y} = \sum_{m=1}^{M}b_{m}(g_{m})_{*}\mu_{Y_{m}}$. 
    We define a map $f_{n}\times g_{m}:X_{n}\times_{p} Y_{m} \to X\times_{p} Y$ by $(f_{n}\times g_{m})(x,y) := (f_{n}(x),\, g_{m}(y))$ for $x \in X_{n}$ and $y \in Y_{m}$.
    By Lemma~\ref{lem:distributive raw for sums of measures}, we obtain \[\sum_{n \in [N],\, m \in [M]} a_{n}b_{m}(f_{n}\times g_{m})_{*}(\mu_{X_{n}}\otimes\mu_{Y_{m}}) = \mu_{X}\otimes\mu_{Y}.\]
    Since $X_{n}\times_{p} Y_{m} \in \mathcal{P}_{n}\times_{p}\mathcal{Q}_{m}$ and the maps $f_{n}\times g_{m}$ are 1-Lipschitz, we have $X\times_{p} Y \in \sum_{n,m}(\mathcal{P}_{n}\times_{p}\mathcal{Q}_{m})^{a_{n}b_{m}}$. 
    This completes the proof of the inclusion $\subset$.

    Next, we prove the opposite inclusion $\supset$. 
    For $n \in [N]$ and $m \in [M]$, by Lemma~\ref{lem:app. seq. of pyramid}, there exist approximation sequences $\{X_{nk}\}_{k=1}^{\infty}$ and $\{Y_{mk}\}_{k=1}^{\infty}$ of the pyramids $\mathcal{P}_{n}$ and $\mathcal{Q}_{m}$, respectively. 
    We set $\widetilde{X}_{k} := \sum_{n=1}^{N} X_{nk}^{a_{n}}$, $\widetilde{Y}_{k} := \sum_{m=1}^{M}Y_{mk}^{b_{m}}$, and $\widetilde{Z}_{k} := \sum_{n \in [N],\, m \in [M]} (X_{nk}\times_{p}Y_{mk})^{a_{n}b_{m}}$.
    By~\cite[Lemma 1.21]{EKM2024}, the sequence $\{\mathcal{P}_{X_{nk}\times_{p}Y_{mk}}\}_{k=1}^{\infty}$ converges weakly to the pyramid $\mathcal{P}_{n} \times_{p}\mathcal{Q}_{m}$.
    Hence, by Theorem~\ref{thm:the continuity of the direct sum of pyramids}, the sequence $\{\mathcal{P}_{\widetilde{Z}_{k}}\}_{k=1}^{\infty}$ converges weakly to the pyramid $\sum_{n, m}\,(\mathcal{P}_{n}\times_{p}\mathcal{Q}_{m})^{a_{n}b_{m}}$.
    Since the extended mm-space $\widetilde{Z}_{k}$ is mm-isomorphic to the extended mm-space $\widetilde{X}_{k}\times_{p}\widetilde{Y}_{k}$, Lemma~\ref{lem:l^p product of extended mm-sp. is equal to the l^p product of pyramids} implies that
    \begin{align*}
        \mathcal{P}_{\widetilde{Z}_{k}} &= \mathcal{P}_{\widetilde{X}_{k}\times_{p}\widetilde{Y}_{k}} = \mathcal{P}_{\widetilde{X}_{k}} \times_{p} \mathcal{P}_{\widetilde{Y}_{k}} \subset \left(\sum_{n=1}^{N}\mathcal{P}_{n}^{a_{n}}\right)\times_{p}\left(\sum_{m=1}^{M}\mathcal{Q}_{m}^{b_{m}}\right). 
    \end{align*}    
    Letting $k \to \infty$, we obtain the inclusion $\supset$.
    This completes the proof.
\end{proof}

\begin{rem}
    \cite[Theorem 6.43]{EKM2024} provides a formula for the $\ell^{p}$-product of two pyramids generated by atoms.
    Since pyramids generated by atoms can be represented as direct sums of pyramids (see~\eqref{eq:pyramid generated by atoms is a direct sum of pyramids}), Proposition~\ref{prop:the distributive law} may be regarded as a generalization of~\cite[Theorem 6.43]{EKM2024}.
\end{rem}

\section{Examples of direct sum of pyramids}
In this section, we study examples of sequences of mm-spaces that converge weakly to a direct sum of pyramids.
We first establish the following useful lemmas to calculate examples.
Note that for an mm-space $X$, a closed subset $A \subset X$ with $\mu_{X}(A) > 0$ defines an mm-space \[\left(X,d_{X},\frac{1}{\mu_{X}(A)}\mu_{X}\Big|_{A}\right).\]
We denote it by the same symbol $A$.
\begin{lem}\label{lem:box distance with X and A in X}
    Let $X$ be an mm-space and $A \subset X$ a closed subset with $\mu_{X}(A)>0$.
    Then we have \[\square(A, X) \le 4(1-\mu_{X}(A)).\]
\end{lem}
\begin{proof}
    By Lemma~\ref{lem:dP < dTV} and Lemma~\ref{lem:box < 2dP}, we have \[\square(A,X) \le 2\dP(\mu_{A}, \mu_{X}) \le 2\dTV(\mu_{A},\mu_{X}).\]
    We estimate the total variation distance.
    For any Borel subset $C \subset X$, we observe that
    \begin{align*}
        \lvert \mu_{A}(C) - \mu_{X}(C) \rvert &= \left\lvert \frac{\mu_{X}(C\cap A)}{\mu_{X}(A)} - \mu_{X}(C \cap A) - \mu_{X}(C \setminus A) \right\rvert \\
            &\le (1-\mu_{X}(A))\frac{\mu_{X}(C \cap A)}{\mu_{X}(A)} + \mu_{X}(X\setminus A) \le 2(1-\mu_{X}(A)),
    \end{align*}
    which implies that \[\dTV(\mu_{A},\mu_{X}) \le 2(1-\mu_{X}(A)).\]
    This completes the proof.
\end{proof}
\begin{lem}\label{lem:the structure of mm-spaces that converge weakly to a direct sum of pyramids}
    Let $\mathcal{P}$ and $\mathcal{Q}$ be two pyramids and $\{X_{n}\}_{n=1}^{\infty}$ a sequence of mm-spaces, and let $0 < \alpha < 1$.
    If there exist closed subsets $A_{n}$, $B_{n}$ of $X_{n}$, $n \in \N$, such that 
    \begin{itemize}
        \item $\mu_{X_{n}}(A_{n}) > 0$ and $\mu_{X_{n}}(B_{n}) > 0$;
        \item $\displaystyle\lim_{n\to\infty}\mu_{X_{n}}(A_{n}) = \alpha$ and $\displaystyle\lim_{n\to\infty}\mu_{X_{n}}(B_{n}) = 1-\alpha$;
        \item $\displaystyle\lim_{n\to\infty}d_{X_{n}}(A_{n},B_{n}) = +\infty$;
        \item $\mathcal{P}_{A_{n}} \longrightarrow \mathcal{P}$ and $\mathcal{P}_{B_{n}} \longrightarrow \mathcal{Q}$ weakly as $n \to \infty$.
    \end{itemize}
    Then $\mathcal{P}_{X_{n}}$ converges weakly to the direct sum $\mathcal{P}^{\alpha} + \mathcal{Q}^{1-\alpha}$.
\end{lem}
\begin{proof}
    We set $R_{n} := d_{X_{n}}(A_{n}, B_{n})$.
    Without loss of generality, we assume that $R_{n} > 0$ for any $n \in \N$.
    For $n \in \N$, we define \[\widetilde{X}_{n} := \left(X_{n},d_{X_{n}}, \mu_{\widetilde{X}_{n}}:=\frac{\alpha}{\mu_{X_{n}}(A_{n})}\mu_{X_{n}}\Big|_{A_{n}} + \frac{1-\alpha}{\mu_{X_{n}}(B_{n})}\mu_{X_{n}}\Big|_{B_{n}}\right),\]which are mm-spaces.
    Since $\widetilde{X}_{n} \in \X((A_{n},B_{n});(\alpha, 1-\alpha);R_{n})$, the pyramid $\mathcal{P}_{\widetilde{X}_{n}}$ converges weakly to the direct sum $\mathcal{P}^{\alpha}+\mathcal{Q}^{1-\alpha}$ by Corollary~\ref{cor:the sum of Xnk's approximates the sum of pyramids}.
    By Lemma~\ref{lem:dP < dTV}, Lemma~\ref{lem:box < 2dP}, and Proposition~\ref{prop:rho < box}, we have
    \begin{align*}
    \rho(\mathcal{P}_{X_{n}},\mathcal{P}^{\alpha}+\mathcal{Q}^{1-\alpha}) &\le \square(X_{n}, \widetilde{X}_{n}) + \rho(\mathcal{P}_{\widetilde{X}_{n}},\mathcal{P}^{\alpha}+\mathcal{Q}^{1-\alpha})\\
            &\le 2\dP(\mu_{X_{n}}, \mu_{\widetilde{X}_{n}}) + \rho(\mathcal{P}_{\widetilde{X}_{n}},\mathcal{P}^{\alpha}+\mathcal{Q}^{1-\alpha}) \\
            &\le 2d_{\mathrm{TV}}(\mu_{X_{n}}, \mu_{\widetilde{X}_{n}}) + \rho(\mathcal{P}_{\widetilde{X}_{n}},\mathcal{P}^{\alpha}+\mathcal{Q}^{1-\alpha}).
    \end{align*}
    Therefore, it suffices to prove that \[\lim_{n\to\infty}d_{\mathrm{TV}}(\mu_{X_{n}}, \mu_{\widetilde{X}_{n}})=0.\]
    For any Borel subset $C \subset X_{n}$, we observe that
    \begin{align*}
        \lvert \mu_{X_{n}}(C) - \mu_{\widetilde{X}_{n}}(C)\rvert &\le \left\lvert \mu_{X_{n}}(C \cap A_{n}) - \frac{\alpha}{\mu_{X_{n}}(A_{n})}\mu_{X_{n}}(C\cap A_{n}) \right\rvert \\
        &\hspace{5mm}+\left\lvert \mu_{X_{n}}(C \cap B_{n}) - \frac{1-\alpha}{\mu_{X_{n}}(B_{n})}\mu_{X_{n}}(C\cap B_{n}) \right\rvert + 1- \mu_{X_{n}}(A_{n}\cup B_{n})\\
        &\le \lvert \mu_{X_{n}}(A_{n}) - \alpha \rvert + \lvert \mu_{X_{n}}(B_{n})-(1-\alpha)\rvert + 1- \mu_{X_{n}}(A_{n}\cup B_{n}),
    \end{align*}
    which implies that 
    \begin{align*}
        d_{\mathrm{TV}}(\mu_{X_{n}}, \mu_{\widetilde{X}_{n}}) &\le \lvert \mu_{X_{n}}(A_{n}) - \alpha \rvert + \lvert \mu_{X_{n}}(B_{n})-(1-\alpha)\rvert + 1- \mu_{X_{n}}(A_{n}\cup B_{n}) \longrightarrow 0
    \end{align*}
    as $n \to \infty$.
    This completes the proof.
\end{proof}

We here recall the definition of wedge sum of two pointed mm-spaces.
Let $(X,x_{0})$ and $(Y,y_{0})$ be two pointed mm-spaces and let $0<\alpha<1$.
We define a function $d:(X\sqcup Y)\times (X\sqcup Y) \to [0,\infty)$ by
\begin{equation*}
    d(z,z^{\prime}) := \begin{cases}
                            d_{X}(z,z^{\prime}) &\text{if $z,z^{\prime} \in X$,} \\
                            d_{Y}(z,z^{\prime}) &\text{if $z,z^{\prime} \in Y$,} \\
                            d_{X}(z,x_{0}) +d_{Y}(z^{\prime},y_{0}) &\text{if $z \in X,\hspace{1mm}z^{\prime} \in Y$,} \\
                            d_{X}(z^{\prime},x_{0}) + d_{Y}(z,y_{0}) &\text{if $z^{\prime} \in X,\hspace{1mm}z \in Y$,}
                        \end{cases}
\end{equation*}
which is a pseudo-metric on $X\sqcup Y$.
We take a quotient of $X\sqcup Y$ by the equivalence relation $d=0$, that is, identifying two points $x_{0} \in X$ and $y_{0} \in Y$, and denote the quotient map by $\pi:X\sqcup Y \to (X\sqcup Y)/(d=0)$.
Then, the pseudo-metric $d$ on $X\sqcup Y$ induces a metric on $(X\sqcup Y)/(d=0)$, and we write the same symbol $d$ for it.
We denote the metric space $((X\sqcup Y)/(d=0), d)$ by $(X,x_{0})\vee (Y,y_{0})$.
Note that this metric space is complete and separable.
Moreover, we define a triple \[(X,x_{0})^{\alpha}\vee (Y,y_{0})^{1-\alpha} := ((X\sqcup Y)/(d=0)),d,\pi_{*}(\alpha\mu_{X}+(1-\alpha)\mu_{Y})),\] which is an mm-space.
We call this mm-space the \textit{wedge sum} of $(X,x_{0})$ and $(Y,y_{0})$.

\begin{prop}\label{prop:wedge sum converges weakly to direct sum}
    Let $\{X_{n}^{1}\}_{n=1}^{\infty}$ and $\{X_{n}^{2}\}_{n=1}^{\infty}$ be sequences of mm-spaces, and let $\mathcal{P}^{1}$ and $\mathcal{P}^{2}$ be pyramids.
    Suppose that $\mathcal{P}_{X_{n}^{i}}$ converges weakly to $\mathcal{P}^{i}$ for each $i=1,2$ as $n \to \infty$, and that there exist two sequences of points $\{x_{n}^{i} \in X_{n}\}_{n=1}^{\infty}$, $i=1,2$, that satisfy \[\lim_{n\to\infty}\mu_{X_{n}^{i}}(U_{r}(x_{n}^{i}))=0,\]for any $r >0$.
    Then, for any $0 < \alpha < 1$, the sequence $\{\mathcal{P}_{(X_{n}^{1},x_{n}^{1})^{\alpha}\vee(X_{n}^{2},x_{n}^{2})^{1-\alpha}}\}_{n=1}^{\infty}$ converges weakly to the direct sum $(\mathcal{P}^{1})^{\alpha} + (\mathcal{P}^{2})^{1-\alpha}$.
\end{prop}

\begin{proof}
    Set $\alpha^{1} := \alpha$, $\alpha^{2}:=1-\alpha$, and $Z_{n} := (X_{n}^{1}, x_{n}^{1})^{\alpha} \vee (X_{n}^{2}, x_{n}^{2})^{1-\alpha}$ for $n \in \mathbb{N}$.
    By the assumption, we find two sequences $\{R_{n}^{i}\}_{n=1}^{\infty}$, $i=1,2$, of positive numbers tending to infinity such that \[\lim_{n\to\infty}\mu_{X_{n}^{i}}(U_{R_{n}^{i}}(x_{n}^{i})) = 0.\]
    For $i=1,2$, we set $A_{n}^{i} := X_{n}^{i}\setminus U_{R_{n}^{i}}(x_{n}^{i})$, then we have \[\lim_{n\to\infty}\mu_{Z_{n}}(A_{n}^{i}) = \alpha^{i}\]and \[d_{Z_{n}}(A_{n}^{1}, A_{n}^{2}) \ge R_{n}^{1}+R_{n}^{2} \longrightarrow +\infty\]as $n \to \infty$.
    Moreover, by Proposition~\ref{prop:rho < box} and Lemma~\ref{lem:box distance with X and A in X}, we have 
    \begin{align*}
        \rho(\mathcal{P}_{A_{n}^{i}}, \mathcal{P}^{i}) &\le \square(A_{n}^{i}, X_{n}^{i}) + \rho(\mathcal{P}_{X_{n}^{i}}, \mathcal{P}^{i}) \\&\le 4\mu_{X_{n}^{i}}(U_{R_{n}^{i}}(x_{n}^{i})) + \rho(\mathcal{P}_{X_{n}^{i}}, \mathcal{P}^{i}) \longrightarrow 0
    \end{align*}
    as $n \to \infty$, which implies that $\{\mathcal{P}_{A_{n}^{i}}\}_{n=1}^{\infty}$ converges weakly to $\mathcal{P}^{i}$.
    By Lemma~\ref{lem:the structure of mm-spaces that converge weakly to a direct sum of pyramids}, we obtain this proposition.
\end{proof}

As an application of Proposition~\ref{prop:wedge sum converges weakly to direct sum}, we next prove the following proposition.
\begin{prop}\label{prop:example of a sequence of mm-spaces that weak converge to a direct sum of pyramids}
    Let $X$ and $Y$ be nontrivial mm-spaces and let $0 < \alpha < 1$ and $p, q \in [1, \infty)$.
    For any two sequences of points $\{\bar{x}^{n} \in X^{n}\}_{n=1}^{\infty}$ and $\{\bar{y}^{n} \in Y^{n}\}_{n=1}^{\infty}$, the pyramid $\mathcal{P}_{(X_{p}^{n},\bar{x}^{n})^{\alpha}\vee(Y_{q}^{n},\bar{y}^{n})^{1-\alpha}}$ converges weakly to $(X_{p}^{\infty})^{\alpha}+(Y_{q}^{\infty})^{1-\alpha}$. 
\end{prop}

For example, we apply this theorem for the $n$-torus $T^{n} := S^{1}(1)\times S^{1}(1)\times \cdots \times S^{1}(1)$, which is equipped with the Riemannian product metric and the normalized Riemannian volume measure.
Note that $T^{n}$ is mm-isomorphic to $T_{2}^{n}$, the $\ell^{2}$-product of $n$ copies of $S^{1}(1)$.
We obtain the following corollary.

\begin{cor}\label{cor:example n-torus}
    For any $0 < \alpha < 1$ and for any two sequences of points $\{x^{n} \in T^{n}\}_{n=1}^{\infty}$ and $\{y^{n} \in T^{n}\}_{n=1}^{\infty}$, the wedge sum $(T^{n}, x^{n})^{\alpha} \vee (T^{n}, y^{n})^{1-\alpha}$ converges weakly to the direct sum $(T_{2}^{\infty})^{\alpha} + (T_{2}^{\infty})^{1-\alpha}$.
\end{cor}
Proposition~\ref{prop:example of a sequence of mm-spaces that weak converge to a direct sum of pyramids} directly follows from Proposition~\ref{prop:wedge sum converges weakly to direct sum} and the following lemma.
\begin{lem}\label{lem:the measure of r-balls of infinite product space are zero}
    Let $X$ be a nontrivial mm-space and let $p \in [1, \infty)$.
    Then, for any $r > 0$, we have \[\lim_{n\to\infty}\sup_{x\in X^{n}}\mu_{X}^{\otimes n}\left(U_{r}^{X_{p}^{n}}(x)\right) = 0.\]
\end{lem}
To prove this, we need the next lemma.
\begin{lem}\label{lem:lem:the measure of r-balls of infinite product space are zero 1}
    Let $X$ be a nontrivial mm-space.
    Then, for any number $r$ with $0 < r < \diam{X}/2$, we have \[\sup_{x \in X} \mu_{X}(U_{r}(x)) < 1.\]
\end{lem}

\begin{proof}
    We prove this by contradiction.
    We assume that there exists a number $r$ with $0 < r < \diam{X}/2$ such that \[\sup_{x \in X} \mu_{X}(U_{r}(x))=1.\]        
    Then, there exists a sequence $\{x_{n}\}_{n=1}^{\infty}$ of points of $X$ such that $\mu_{X}(U_{r}(x_{n}))$ converges to $1$ as $n \to \infty$.
    We take two points $x, x^{\prime} \in X$ such that $d_{X}(x,x^{\prime}) > 2r$ and take any number $\delta$ with \[0 < \delta < \frac{d_{X}(x,x^{\prime})-2r}{2}.\]
    Set $\eta := \mu_{X}(U_{\delta}(x))$ and $\xi := \mu_{X}(U_{\delta}(x^{\prime}))$.
    We note that $\eta > 0$ and $\xi > 0$ since $x,x^{\prime} \in \supp{\mu_{X}}$.
    We take a large number $n \in \N$ and we have \[\mu_{X}(U_{r}(x_{n})) > 1- \min\{\eta, \xi\},\]which implies that \[U_{r}(x_{n}) \cap U_{\delta}(x) \neq \emptyset \hspace{3mm}\text{and}\hspace{3mm}U_{r}(x_{n})\cap U_{\delta}(x^{\prime}) \neq \emptyset.\]
    Take two points $y \in U_{r}(x_{n}) \cap U_{\delta}(x)$ and $y^{\prime} \in U_{r}(x_{n}) \cap U_{\delta}(x^{\prime})$.
    Then we have
    \begin{align*}
        2r &\ge d_{X}(y,y^{\prime}) \ge d_{X}(x,x^{\prime}) - d_{X}(x,y) - d_{X}(x^{\prime}, y^{\prime}) \\
        &\ge d_{X}(x,x^{\prime}) - 2\delta \\
        &> 2r,
    \end{align*}
    which is a contradiction.
    This completes the proof.
\end{proof}

\begin{proof}[Proof of Lemma~\ref{lem:the measure of r-balls of infinite product space are zero}]
    In this proof, we denote $U_{r}^{X_{p}^{n}}(x)$ by $U_{r}(x)$ for simplicity.
    We take any $r > 0$ and fix it.
    Since $\lim_{n\to\infty}\diam{X_{p}^{n}} = \infty$, there exists a number $n_{0} \in \N$ such that $\diam{X_{p}^{n}} > 2r$ for any $n \ge n_{0}$.

    For each $n \ge n_{0}$, there exists a number $k \in \N$ such that $kn_{0} \le n < (k+1)n_{0}$.
    Then, for any point $x=(x_{1},x_{2},\dots,x_{n}) \in X^{n}$ 
    \begin{align*}
        U_{r}(x) &\subset U_{r}((x_{1},\dots,x_{n_{0}}))\times U_{r}((x_{n_{0}+1},\dots,x_{2n_{0}}))\times\cdots \\
        &\hspace{15mm}\cdots\times U_{r}((x_{(k-1)n_{0}+1},\dots,x_{kn_{0}}))\times X^{n-kn_{0}},
    \end{align*}
    which implies that \[\mu_{X}^{\otimes n}(U_{r}(x)) \le \prod_{i=1}^{k}\mu_{X}^{\otimes n_{0}}(U_{r}((x_{(i-1)n_{0}+1},\dots,x_{in_{0}}))) \le \left(\sup_{x^{\prime} \in X^{n}}\mu_{X}^{\otimes n_{0}}(U_{r}(x^{\prime}))\right)^{k}.\]
    Since the right hand side of this inequality does not depend on the point $x$, Claim~\ref{lem:lem:the measure of r-balls of infinite product space are zero 1} implies that
    \[\sup_{x\in X^{n}}\mu_{X}^{\otimes n}(U_{r}(x)) \le \left(\sup_{x^{\prime} \in X^{n}}\mu_{X}^{\otimes n_{0}}(U_{r}(x^{\prime}))\right)^{k} \longrightarrow 0\] as $n \to \infty$.
    This completes the proof.
\end{proof}

We remark that, by the same argument in the proof of Lemma~\ref{lem:the measure of r-balls of infinite product space are zero}, we obtain the next corollary.

\begin{cor}\label{cor:the measure of r-balls of infinite product space for sup-norm are zero}
    Let $X$ be a nontrivial mm-space.
    Then, for any number $r$ with $0 < r < \diam{X}/2$, we have \[\lim_{n\to\infty}\sup_{x \in X^{n}}\mu_{X}^{\otimes n} \left(U_{r}^{X_{\infty}^{n}}(x)\right) = 0.\]
\end{cor}

\section{Decomposition of pyramids}

In this section, we prove Theorem~\ref{thm:decomposition of pyramids} and Theorem~\ref{thm:the uniquness of decomposition of pyramid}.

\subsection{Decomposition}

\begin{dfn}\label{def:finiteness of the observable diameter of pyramids}
    Let $\mathcal{P}$ be a pyramid. 
    We say that the observable diameter of $\mathcal{P}$ is \textit{finite} if \[ \mathrm{ObsDiam}(\mathcal{P}\,;\,-\kappa) < \infty\] for every $0 < \kappa < 1$. 
    Otherwise, it is said to be \textit{infinite}.
\end{dfn}

\begin{rem}
    This notion coincides with condition (7.4) in~\cite{S2016book}, which concerns the uniform boundedness of the $\kappa$-observable diameter over a family of mm-spaces for each fixed $0 < \kappa < 1$, when restricted to pyramids.
\end{rem}

For convenience, we introduce the following notation: For $A=(a_{n})_{n=1}^{N} \in \mathcal{A}$ and a sequence of pyramids $\{\mathcal{P}_{n}\}_{n=1}^{N}$, we write 
\begin{align*}
    \X^{1-\|A\|_{1}} + \sum_{n=1}^{N}\mathcal{P}_{n}^{a_{n}} 
        :=\begin{cases}
            \X &\text{if $\|A\|_{1}=0$,}\\
            \sum_{n=0}^{N}\mathcal{P}_{n}^{a_{n}} &\text{if $0 < \|A\|_{1} < 1$}, \\
            \sum_{n=1}^{N}\mathcal{P}_{n}^{a_{n}} &\text{if $\|A\|_{1} = 1$},
        \end{cases}
\end{align*}
where $a_{0} := 1-\|A\|_{1}$ and $\mathcal{P}_{0} := \X$.

To prove Theorem~\ref{thm:decomposition of pyramids}, we establish some lemmas.

\begin{lem}\label{lem:the limit of pyramids tP when t to 0+}
    Let $\mathcal{P}$ be a pyramid. 
    Then $t\mathcal{P}$ converges weakly as $t \to 0+$, and the limit is $\mathcal{P}_{A}$ for some $A \in \mathcal{A}$.
\end{lem}

\begin{proof}
    Take any sequence of positive numbers $\{t_{n}\}_{n=1}^{\infty}$ converging to 0.
    There exists a monotone decreasing subsequence $\{t_{n(k)}\}_{k=1}^{\infty} \subset \{t_{n}\}_{n=1}^{\infty}$.
    We have $t_{n(k+1)}\mathcal{P} \subset t_{n(k)}\mathcal{P}$ for each $k$.
    By~\cite[Proposition 1.12 (2)]{EKM2024}, it follows that $t_{n(k)}\mathcal{P}$ converges weakly to the pyramid $\bigcap_{k=1}^{\infty}t_{n(k)}\mathcal{P}$.
    We here observe that \[\bigcap_{t > 0} t\mathcal{P} = \bigcap_{k=1}^{\infty} t_{n(k)}\mathcal{P},\] which implies that $\{t_{n}\mathcal{P}\}_{n=1}^{\infty}$ has the  subsequence $\{t_{n(k)}\mathcal{P}\}_{k=1}^{\infty}$ converging weakly to the pyramid $\bigcap_{t > 0}t\mathcal{P}$.
    Therefore, $t\mathcal{P}$ converges weakly to the pyramid $\bigcap_{t > 0}t\mathcal{P}$ as $t \to 0+$.
    Moreover, by definition, the pyramid $\bigcap_{t > 0}t\mathcal{P}$ is scale invariant, and hence, by Theorem~\ref{thm:the characterization of scale-invariant pyramids}, there exists $A \in \A$ such that $\mathcal{P}_{A} = \bigcap_{t>0} t\mathcal{P}$.
    This completes the proof.
\end{proof}

\begin{lem}\label{lem:scale transformation of direct sums of pyramids}
    Let $(a_{n})_{n=1}^{N} \in \mathcal{A}_{1}$, and let $\{\mathcal{P}_{n}\}_{n=1}^{N}$ be a sequence of pyramids. For any $t > 0$ , we have \[t\left(\sum_{n=1}^{N} \mathcal{P}_{n}^{a_{n}}\right) = \sum_{n=1}^{N} (t\mathcal{P}_{n})^{a_{n}}.\]
\end{lem}

\begin{proof}
    For any mm-spaces $X_{n}$, $n \in [N]$, it is clear that $t(\sum_{n=1}^{N}X_{n}^{a_{n}}) = \sum_{n=1}^{N}(tX_{n})^{a_{n}}$ and $t\mathcal{P}_{\sum_{n=1}^{N}X_{n}^{a_{n}}} = \mathcal{P}_{\sum_{n=1}^{N}(tX_{n})^{a_{n}}}$.
    This implies that
    \begin{align*}
        t\left(\sum_{n=1}^{N} \mathcal{P}_{n}^{a_{n}}\right) &=t\left(\bigcup_{(X_{n})_{n=1}^{N} \in \prod_{n=1}^{N}\mathcal{P}_{n}}\mathcal{P}_{\sum_{n=1}^{N}X_{n}^{a_{n}}}\right)\\
        &= \bigcup_{(X_{n})_{n=1}^{N} \in \prod_{n=1}^{N}\mathcal{P}_{n}}t\mathcal{P}_{\sum_{n=1}^{N}X_{n}^{a_{n}}}= \bigcup_{(X_{n})_{n=1}^{N} \in \prod_{n=1}^{N}\mathcal{P}_{n}}\mathcal{P}_{\sum_{n=1}^{N}(tX_{n})^{a_{n}}}\\
        &= \bigcup_{(Y_{n})_{n=1}^{N} \in \prod_{n=1}^{N}t\mathcal{P}_{n}}\mathcal{P}_{\sum_{n=1}^{N}Y_{n}^{a_{n}}} = \sum_{n=1}^{N}(t\mathcal{P}_{n})^{a_{n}}.
    \end{align*}
    We obtain this lemma.
\end{proof}

\begin{lem}\label{lem:the characterization of 1pt mm-sp.}
    For an mm-space $X$, the following (i) and (ii) are equivalent to each other.
    \begin{enumerate}
        \item $X$ is mm-isomorphic to the one-point mm-space $*$.
        \item  For any $\kappa$ with $0 < \kappa < 1$, we have $\mathrm{ObsDiam}(X;\, -\kappa) = 0$.
    \end{enumerate}
\end{lem}

\begin{proof}
    It is clear that (i) $\Rightarrow$ (ii). 
    We prove (i) $\Leftarrow$ (ii). 
    Suppose that $X$ is not mm-isomorphic to the one-point mm-space $*$. 
    Then, there exist Borel subsets $A_{1}$ and $A_{2}$ of $X$ such that $\mu_{X}(A_{1}) > 0$, $\mu_{X}(A_{2}) > 0$, and $d_{X}(A_{1},\, A_{2}) > 0$. 
    Take two positive numbers $\kappa$ and $\kappa^{\prime}$ with $0 < \kappa^{\prime} < \kappa < \min_{i=1,2}\,\mu_{X}(A_{i})$. 
    By~\cite[Proposition 2.26 (2)]{S2016book}, we have \[\mathrm{ObsDiam}(X;\,-\kappa^{\prime}) \ge \mathrm{Sep}(X;\kappa,\, \kappa) \ge d_{X}(A_{1},\, A_{2}) > 0.\]
    This implies that (ii) does not hold, which completes the proof.
\end{proof}

\begin{lem}\label{lem:the characterization of pyramid whose observable diameter is finite}
    For a pyramid $\mathcal{P}$, the following (i) and (ii) are equivalent to each other.
    \begin{enumerate}
        \item The observable diameter of $\mathcal{P}$ is finite.
        \item $t\mathcal{P} \longrightarrow \{*\} \hspace{5mm} \text{weakly as} \hspace{5mm}t \to 0+$.
    \end{enumerate}
\end{lem}

\begin{proof}
    First, we prove (i) $\Rightarrow$ (ii). 
    Assume that the observable diameter of $\mathcal{P}$ is finite.
    Set $\mathcal{P}^{\prime} := \bigcap_{t > 0} t\mathcal{P}$.
    Since $\mathcal{P}^{\prime} \subset t\mathcal{P}$ for any $t > 0$ and $\mathcal{P}$ has finite observable diameter, Lemma~\ref{lem:scale invariance of observable diameter} shows that for any $0 < \kappa < 1$, \[\mathrm{ObsDiam}(\mathcal{P}^{\prime};\,-\kappa) \le \mathrm{ObsDiam}(t\mathcal{P};\,-\kappa) = t\mathrm{ObsDiam}(\mathcal{P};\,-\kappa) \longrightarrow 0\] as $t \to 0+$. 
    This implies that $\mathrm{ObsDiam}(\mathcal{P}^{\prime};\, -\kappa) = 0$ for any $0 < \kappa < 1$, that is, $\mathrm{ObsDiam}(X;\, -\kappa) = 0$ for any $X \in \mathcal{P}^{\prime}$ and any $0 < \kappa < 1$. 
    By Lemma~\ref{lem:the characterization of 1pt mm-sp.}, we obtain $\mathcal{P}^{\prime} = \{*\}$. 
    This completes the proof of (i)$\Rightarrow$(ii).

    Next, we prove (i) $\Leftarrow$ (ii). Assume that (i) does not hold. Then there exists a real number $0 < \kappa_{0} < 1$ such that $\mathrm{ObsDiam}(\mathcal{P};\,-\kappa_{0}) = \infty$. 
    Since the observable diameter of $\mathcal{P}$ is monotone nonincreasing in $\kappa$, we have \[\mathrm{ObsDiam}(\mathcal{P};\,-\kappa) = \infty\] for any $0 < \kappa < \kappa_{0}$. 
    For such a $\kappa$, we have 
    \begin{align*}
        &\lim_{\varepsilon \to 0+}\liminf_{n \to \infty} \mathrm{ObsDiam}\left(\frac{1}{n}\mathcal{P};\,-(\kappa+\varepsilon)\right)\\ &= \lim_{\varepsilon \to 0+} \liminf_{n \to \infty} \frac{1}{n}\mathrm{ObsDiam}(\mathcal{P};\,-(\kappa+\varepsilon)) = \infty.
    \end{align*}
    By the limit formula for the observable diameter (see~\cite[Theorem 1.1]{Y_limit_formula}), this implies that the pyramid $t\mathcal{P}$ does not converge weakly to the pyramid $\{*\}$ as $t \to 0+$, in other words, condition (ii) does not hold. 
    This completes the proof of (i) $\Leftarrow$ (ii).
\end{proof}
For $A=(a_{n})_{n=1}^{N} \in \mathcal{A}$ and a real number $c > 0$, we define \[cA:=(ca_{n})_{n=1}^{N},\]which also belongs to $\mathcal{A}$.
\begin{dfn}\label{def:the sum of sequences with weight}
    Let $(a_{n})_{n=1}^{N} \in \A$ and $B_{n} = (b_{nk})_{k=1}^{N_{n}} \in \mathcal{A}$, $n \in [N]$. 
    We define $\sum_{n=1}^{N} a_{n}B_{n}$ as the nonincreasing rearrangement of the terms $a_{n}b_{nk}$ for all $n \in [N]$ and $k \in [N_{n}]$. 
\end{dfn}

 We note that $\sum_{n=1}^{N} a_{n}B_{n} \in \mathcal{A}$.
 
\begin{lem}\label{lem:the condition when A coincides with the direct sum of seq. of real numbers}
    For $A=(a_{n})_{n=1}^{N} \in \mathcal{A}\setminus \{(0)\}$ and $B_{n}=(b_{nk})_{k=1}^{N_{n}} \in \mathcal{A}$, $n \in [N]$, the following (i) and (ii) are equivalent to each other.
    \begin{enumerate}
        \item $A = \sum_{n=1}^{N} a_{n}B_{n}$.
        \item $B_{n} = (1)$ for any $n \in [N]$.
    \end{enumerate}
\end{lem}
 
\begin{proof}
    It is clear that (i) $\Leftarrow$ (ii). 
    We prove (i) $\Rightarrow$ (ii). 
    Assume that the condition (ii) does not hold. 
    Then there exists a number $n_{0} \in [N]$ such that $B_{n_{0}} \neq (1)$. 
    We suppose that the number $n_{0}$ is the smallest one which satisfies $B_{n_{0}} \neq (1)$. 
    Set $m := \#\{n \in \mathbb{N} \mid a_{n} = a_{n_{0}}\} < \infty$ and $(b_{n})_{n=1}^{N^{\prime}} := \sum_{n=1}^{N} a_{n}B_{n} \in \mathcal{A}$. 
    Then we have \[\#\{n \in \mathbb{N} \mid b_{n} = a_{n_{0}}\} = \#\{n \in \mathbb{N} \mid a_{n}=a_{n_{0}}\hspace{1mm}\text{and}\hspace{1mm}B_{n}=(1)\} \le m-1 < m. \]
    This implies that $A \neq \sum_{n=1}^{N} a_{n}B_{n}$, in other words, the condition (i) does not hold.
    This completes the proof of (i) $\Rightarrow$ (ii).
\end{proof} 

\begin{lem}\label{lem:the direct sum of pyramids equals the direct sum of seq. of real number}
    For any $A=(a_{n})_{n=1}^{N} \in \mathcal{A}_{1}\cap\mathcal{A}$ and any $B_{n}=(b_{nk})_{k=1}^{N_{n}} \in \mathcal{A}$, $n \in [N]$, we have \[\sum_{n=1}^{N} \mathcal{P}_{B_{n}}^{a_{n}} = \mathcal{P}_{\sum_{n=1}^{N} a_{n}B_{n}}.\]
\end{lem}

\begin{proof}
    First, we prove the inclusion $\subset$. 
    Take any $X \in \sum_{n=1}^{N}\mathcal{P}_{B_{n}}^{a_{n}}$. 
    For $n \in [N]$, there exist mm-spaces $X_{n} \in \mathcal{P}_{B_{n}}$ and 1-Lipschitz maps $f_{n}:X_{n} \to X$ such that $\mu_{X} = \sum_{n=1}^{N}a_{n}(f_{n})_{*}\mu_{X_{n}}$. 
    Moreover, since $X_{n} \in \mathcal{P}_{B_{n}}$, there exist points $x_{nk} \in X_{n}$, $k \in [N_{n}]$, such that $\sum_{k=1}^{N_{n}}b_{nk}\delta_{x_{nk}} \le \mu_{X_{n}}$. 
    For $n \in [N]$ and $k \in [N_{n}]$, set $y_{nk} := f_{n}(x_{nk})$.
    We take mutually distinct points $y_{m} \in X$ for $m \in [M]$, $M \in \overline{\mathbb{N}}$, so that $\{y_{m}\mid m \in [M]\} = \{y_{nk} \mid n \in [N], k \in [N_{n}]\}$.
    For $m \in [M]$, set \[N_{1}(m) := \{n \in [N] \mid \text{There exists }k \in [N_{n}] \hspace{1mm}\text{such that }\hspace{1mm} y_{nk} = y_{m}\}\]and \[N(m) := \{(n,k) \mid y_{nk} = y_{m},\hspace{1mm}n \in [N], k \in [N_{n}]\}.\]
    Then we have 
    \begin{align*}
        \mu_{X}(\{y_{m}\}) &=\sum_{n=1}^{N}a_{n}(f_{n})_{*}\mu_{X_{n}}(\{y_{m}\}) \ge \sum_{n \in N_{1}(m)}a_{n}(f_{n})_{*}\mu_{X_{n}}(\{y_{m}\}) \\
            &\ge \sum_{n \in N_{1}(m)}a_{n}\sum_{\ell=1}^{N_{n}}b_{n\ell}\delta_{x_{n\ell}}(f_{n}^{-1}(\{y_{m}\})) = \sum_{n \in N_{1}(m)}a_{n}\sum_{k;(n,k) \in N(m)}b_{nk} \\
            &= \sum_{(n,k) \in N(m)}a_{n}b_{nk},
    \end{align*}
    and we obtain \[\mu_{X} \ge \sum_{m=1}^{M}\sum_{(n,k) \in N(m)}a_{n}b_{nk}\delta_{y_{m}} = \sum_{n \in [N],\, k \in [N_{n}]}a_{n}b_{nk}\delta_{y_{nk}}.\] 
    This implies that $X \in \mathcal{P}_{\sum_{n=1}^{N} a_{n}B_{n}}$, which completes the proof of the inclusion $\subset$.

    Next, we prove the opposite inclusion $\supset$. 
    Take any $X \in \mathcal{P}_{\sum_{n=1}^{N} a_{n}B_{n}}$. 
    There exist points $x_{nk} \in X$, $n \in [N]$, $k \in [N_{n}]$, such that \[\sum_{n \in [N],\, k \in [N_{n}]} a_{n}b_{nk}\delta_{x_{nk}} \le \mu_{X}.\] 
    Set 
    \begin{equation*}
        \nu := \mu_{X} - \sum_{k \in [N_{n}],\, n \in [N]} a_{n}b_{nk}\delta_{x_{nk}} \hspace{3mm}\text{and}\hspace{3mm}\nu_{n} := \begin{cases}
                        \frac{1-\|B_{n}\|_{1}}{\nu(X)}\nu &\text{if $\nu(X) > 0$,} \\
                        0 &\text{if $\nu(X) = 0$,}
                    \end{cases}
    \end{equation*}
    for $n \in [N]$.
    We observe that $\sum_{n=1}^{N} a_{n}\nu_{n} = \nu$. 
    In fact, it is clear if $\nu(X) = 0$. 
    If $\nu(X) > 0$, then we have \[\sum_{n=1}^{N}a_{n}\nu_{n} = \sum_{n=1}^{N} \frac{a_{n}-a_{n}\|B_{n}\|_{1}}{\nu(X)}\nu = \frac{1-\sum_{n=1}^{N}a_{n}\|B_{n}\|_{1}}{\nu(X)}\nu = \nu.\]
    For $n \in [N]$, define \[\mu_{n} := \sum_{k=1}^{N_{n}} b_{nk}\delta_{x_{nk}} + \nu_{n}.\]
    It is clear that $\mu_{n}$ is a Borel probability measure on $X$, and the triple $X_{n}:=(X,\, d_{X},\, \mu_{n})$ is an mm-space that belongs to $\mathcal{P}_{B_{n}}$. 
    We define 1-Lipschitz maps $f_{n}:X_{n} \to X$ by $f_{n}(x) := x$ for $n \in [N]$, then we have 
    \begin{align*}
        \sum_{n=1}^{N}a_{n}(f_{n})_{*}\mu_{X_{n}} &= \sum_{n=1}^{N}a_{n}\mu_{n} = \sum_{n=1}^{N}\sum_{k=1}^{N_{n}}a_{n}b_{nk}\delta_{x_{nk}} + \sum_{n=1}^{N}a_{n}\nu_{n} \\ &= \sum_{n=1}^{N}\sum_{k=1}^{N_{n}}a_{n}b_{nk}\delta_{x_{nk}} +\nu = \mu_{X},
    \end{align*}
    which implies $X \in \sum_{n=1}^{N}\mathcal{P}_{B_{n}}^{a_{n}}$. 
    This completes the proof of the inclusion $\supset$.
\end{proof}
\begin{rem}
A similar construction of the measure $\mu_{n}$ used in the proof of Lemma~\ref{lem:the direct sum of pyramids equals the direct sum of seq. of real number} appears in the proof of~\cite[Lemma 6.28]{EKM2024}.
\end{rem}
For $n \in \mathbb{N}$, we define an mm-space $\mathbb{D}_{n} = (\{s_{n}^{1},\dots, s_{n}^{n}\},\, d_{\mathbb{D}_{n}},\, \mu_{\mathbb{D}_{n}})$ by \[d_{\mathbb{D}_{n}}(s_{n}^{i}, s_{n}^{j}) := \begin{cases}
                                                n &\text{if $i \neq j$} \\
                                                0 &\text{if $i = j$} \end{cases}\hspace{5mm}\text{and}\hspace{5mm} \mu_{\mathbb{D}_{n}} :=\frac{1}{n}\sum_{i=1}^{n} \delta_{s_{n}^{i}}. \]
By Lemma~\ref{lem:infinitely dissipation criterion}, we see that the sequence of mm-spaces $\{\mathbb{D}_{n}\}_{n=1}^{\infty}$ infinitely dissipates. 
Furthermore, for any $A=(a_{n})_{n=1}^{N} \in \mathcal{A}$ and any sequence of mm-spaces $\{X_{k}\}_{k=1}^{N}$, we set \[\mathbb{D}_{n}(\{X_{k}\}_{k=1}^{N};\,A) := \begin{cases}
                                                    (\sum_{k=0}^{N}X_{k}^{a_{k}})_{n} &\text{if $\|A\|_{1} < 1$,} \\
                                                    (\sum_{k=1}^{N}X_{k}^{a_{k}})_{n} &\text{if $\|A\|_{1} = 1$,} \\
                                            \end{cases}\] 
where $a_{0} := 1-\|A\|_{1}$ and $X_{0} := \mathbb{D}_{n}$.

\begin{proof}[Proof of Theorem~\ref{thm:decomposition of pyramids}]
    By Lemma~\ref{lem:the limit of pyramids tP when t to 0+}, the pyramid $t\mathcal{P}$ converges weakly to the pyramid $\mathcal{P}_{A}$ for some $A=(a_{n})_{n=1}^{N} \in \mathcal{A}$ as $t \to 0+$. 
    If $A=(0)$, then $\mathcal{P} = \mathcal{X}$, and hence this theorem holds.
    We next consider the case $A \neq (0)$.
    For $n \in \mathbb{N}$, we set $W_{n} := \mathbb{D}_{n}(\{pt_{k}\}_{k=1}^{N};\,A)$, where $pt_{k}$ is the one-point mm-space. 
    Since $ W_{n} \in \mathcal{P}_{A} \subset \mathcal{P}$, it follows from Lemma~\ref{lem:app. seq. of pyramid} that there exists an approximation sequence $\{Y_{n}\}_{n=1}^{\infty}$ of the pyramid $\mathcal{P}$ such that $W_{n} \prec Y_{n}$ for each $n$. 
    Let $g_{n}:Y_{n} \to W_{n}$ be a domination map. 
    We note that \[\mu_{Y_{n}}(g_{n}^{-1}(\{pt_{k}\})) = \mu_{W_{n}}(\{pt_{k}\}) = a_{k} > 0\] for each $k$. 
    Define triples \[X_{nk} := \left(Y_{n},\, d_{Y_{n}},\, a_{k}^{-1}\mu_{Y_{n}}|_{g_{n}^{-1}(\{pt_{k}\})}\right).\]
    Since $g_{n}^{-1}(\{pt_{k}\})$ are closed subsets of $Y_{n}$, $X_{nk}$ are mm-spaces. 
    If $\|A\|_{1} < 1$, then we define mm-spaces \[X_{n0} := \left(Y_{n},\, d_{Y_{n}},\, (1-\|A\|_{1})^{-1}\mu_{Y_{n}}|_{g_{n}^{-1}(\mathbb{D}_{n})}\right).\]
    Since $\mathbb{D}_{n} \prec X_{n0}$ for every $n$, it follows from Lemma~\ref{lem:infinitely dissipation criterion} that $\{X_{n0}\}_{n=1}^{\infty}$ infinitely dissipates. 
    By Theorem~\ref{thm:metric rho} and by a diagonal argument, we obtain a subsequence $\{X_{n(\ell)k}\}_{\ell=1}^{\infty}$ of $\{X_{nk}\}_{n=1}^{\infty}$ such that, for each $k \in [N]$, $\{\mathcal{P}_{X_{n(\ell)k}}\}_{\ell=1}^{\infty}$ is a weakly convergent sequence, and we denote its limit by $\mathcal{P}_{k}$. 
    Since for any $k \neq \ell$ we have \[\mathrm{dist}(g_{n}^{-1}(\{pt_{k}\}),\, g_{n}^{-1}(\{pt_{\ell}\})) \ge n,\] we observe that if $\|A\|_{1}=1$, then $Y_{n(\ell)} \in \mathcal{X}(\{X_{n(\ell)k}\}_{k=1}^{N};(a_{k})_{k=1}^{N};n(\ell))$ and if $\|A\|_{1} < 1$, then we also have $Y_{n(\ell)} \in \mathcal{X}(\{X_{n(\ell)k}\}_{k=0}^{N};(a_{k})_{k=0}^{N};n(\ell))$, where $a_{0} := 1-\|A\|_{1}$.
    By Proposition~\ref{prop:infinitely dissipated sequence converges weakly to the maximum pyramid} and Corollary~\ref{cor:the sum of Xnk's approximates the sum of pyramids}, $\mathcal{P}_{Y_{n(\ell)}}$ converges weakly to the pyramid $\mathcal{X}^{1-\|A\|_{1}} + \sum_{n=1}^{N}\mathcal{P}_{n}^{a_{n}}$.
    We recall that $\{Y_{n(\ell)}\}_{\ell=1}^{\infty}$ is also an approximation sequence of $\mathcal{P}$, and we obtain \[\mathcal{P} = \mathcal{X}^{1-\|A\|_{1}} + \sum_{n=1}^{N}\mathcal{P}_{n}^{a_{n}}.\]

    Next, we prove that the observable diameter of $\mathcal{P}_{n}$ is finite for every $n$.
    By Lemma~\ref{lem:the limit of pyramids tP when t to 0+}, for each $n \in [N]$ there exists $B_{n} \in \A$ such that $t\mathcal{P}_{n}$ converges weakly to $\mathcal{P}_{B_{n}}$ as $t \to 0+$. 
    By Theorem~\ref{thm:the continuity of the direct sum of pyramids} and Lemma~\ref{lem:scale transformation of direct sums of pyramids}, the pyramid $t\mathcal{P} = \mathcal{X}^{1-\|A\|_{1}} + \sum_{n=1}^{N} (t\mathcal{P}_{n})^{a_{n}}$ converges weakly to the pyramid $\mathcal{X}^{1-\|A\|_{1}} + \sum_{n=1}^{N} \mathcal{P}_{B_{n}}^{a_{n}}$, which is equal to $\mathcal{P}_{A}$. 
    We set $B=(b_{n})_{n=1}^{N^{\prime}} := \sum_{n=1}^{N}a_{n}B_{n}$.
    By Proposition~\ref{prop:the associative law} and Lemma~\ref{lem:the direct sum of pyramids equals the direct sum of seq. of real number}, we have 
    \begin{align*}
        \mathcal{X}^{1-\|A\|_{1}} + \sum_{n=1}^{N}\mathcal{P}_{B_{n}}^{a_{n}} &= \mathcal{X}^{1-\|A\|_{1}} + \left(\sum_{n=1}^{N}\mathcal{P}_{B_{n}}^{\frac{a_{n}}{\|A\|_{1}}}\right)^{\|A\|_{1}} = \mathcal{X}^{1-\|A\|_{1}} + \mathcal{P}_{\|A\|_{1}^{-1}B}^{\|A\|_{1}} \\
        &= \mathcal{X}^{1-\|A\|_{1}} + \left\lbrace\mathcal{X}^{1-\frac{\|B\|_{1}}{\|A\|_{1}}} + \left(\sum_{n=1}^{N^{\prime}}\{*\}^{\frac{b_{n}}{\|B\|_{1}}}\right)^{\frac{\|B\|_{1}}{\|A\|_{1}}}\right\rbrace^{\|A\|_{1}} \\
        &= \left(\mathcal{X}^{\frac{1-\|A\|_{1}}{1-\|B\|_{1}}} + \mathcal{X}^{\frac{\|A\|_{1}-\|B\|_{1}}{1-\|B\|_{1}}}\right)^{1-\|B\|_{1}} + \left(\sum_{n=1}^{N^{\prime}} \{*\}^{\frac{b_{n}}{\|B\|_{1}}}\right)^{\|B\|_{1}}\\ 
        &= \mathcal{X}^{1-\|B\|_{1}}+ \left(\sum_{n=1}^{N^{\prime}} \{*\}^{b_{n}/\|B\|_{1}}\right)^{\|B\|_{1}}\\
        &= \mathcal{P}_{B},
    \end{align*} 
    which implies $\mathcal{P}_{A} = \mathcal{P}_{B}$.
    By Lemma~\ref{lem:A=B <=> P_A = P_B}, we have $A=B=\sum_{n=1}^{N}a_{n}B_{n}$.
    By Lemma~\ref{lem:the condition when A coincides with the direct sum of seq. of real numbers}, we obtain $B_{n} = (1)$ for every $n \in [N]$. 
    This shows that the pyramid $t\mathcal{P}_{n}$ converges weakly to the pyramid $\{*\}$, and hence, by Lemma~\ref{lem:the characterization of pyramid whose observable diameter is finite}, the observable diameter of $\mathcal{P}_{n}$ is finite. 
    This completes the proof.
\end{proof}

\subsection{Uniqueness}

In this subsection, we establish some lemmas and prove the uniqueness of the decomposition of pyramids.

\begin{lem}\label{lem:lem of lem:the direct sum of Pn include the direct sum of Xn, then Xn in some Pm 1}
    Let $(a_{n})_{n=1}^{N} \in \mathcal{A}_{1}$ and $\{X_{n}\}_{n=1}^{N}$ a sequence of mm-spaces.
    Let $\mathcal{P}$ be a pyramid.
    For any $n_{0} \in [N]$ and any sequence of positive numbers $\{R_{k}\}_{k=1}^{\infty}$ diverging to infinity, suppose that there exist $Y_{k} \in \mathcal{P}$ and 1-Lipschitz maps $f_{k}:Y_{k} \to (\sum_{n=1}^{N}X_{n}^{a_{n}})_{R_{k}}$, $k \in \N$, such that \[a_{n_{0}}(f_{k})_{*}\mu_{Y_{k}} \le \sum_{n=1}^{N}a_{n}\mu_{X_{n}} \hspace{3mm}\text{and}\hspace{3mm}\lim_{k\to\infty}\mu_{Y_{k}}(f_{k}^{-1}(X_{n_{0}})) = 1.\]
    Then we have $X_{n_{0}} \in \mathcal{P}$.
\end{lem}

\begin{proof}
    Set $Z_{k}:=(\sum_{n=1}^{N}X_{n}^{a_{n}})_{R_{k}}$ and $\widetilde{Y}_{k} := f_{k}^{-1}(X_{n_{0}})$, $k \in \N$.
    We note that $\mu_{Y_{k}}(Y_{k}\setminus\widetilde{Y}_{k})$ converges to 0 by the assumption of this lemma.
    We take any point $x_{0} \in X_{n_{0}}$ and fix it.
    For $k \in \N$, we define a map $\widetilde{f}_{k}:Y_{k} \to X_{n_{0}}$ by
    \begin{equation*}
        \widetilde{f}_{k}(y) := \begin{cases}
                                    f_{k}(y) &\text{if $y \in \widetilde{Y}_{k}$},\\
                                    x_{0} &\text{if $y \in Y_{k}\setminus\widetilde{Y}_{k}$},
        \end{cases}
    \end{equation*}
    which is a 1-Lipschitz map up to $\mu_{Y_{k}}(Y_{k}\setminus\widetilde{Y}_{k})$.

    For any Borel subset $B \subset X_{n_{0}}$, we see that
    \begin{align*}
        (\widetilde{f}_{k})_{*}\mu_{Y_{k}}(B) &= \mu_{Y_{k}}(f_{k}^{-1}(B)) + \delta_{x_{0}}(B)\mu_{Y_{k}}(Y_{k}\setminus\widetilde{Y}_{k}) \\
            &\le \frac{1}{a_{n_{0}}}\left(\sum_{n=1}^{N}a_{n}\mu_{X_{n}}\right)(B) + \mu_{Y_{k}}(Y_{k}\setminus \widetilde{Y}_{k}) = \mu_{X_{n_{0}}}(B) + \mu_{Y_{k}}(Y_{k}\setminus\widetilde{Y}_{k})
    \end{align*}
    and
    \begin{align*}
        \mu_{X_{n_{0}}}(B) &= \frac{1}{a_{n_{0}}}\left\lbrace a_{n_{0}}(f_{k})_{*}\mu_{Y_{k}}(B) + \left(\sum_{n=1}^{N}a_{n}\mu_{X_{n}} - a_{n_{0}}(f_{k})_{*}\mu_{Y_{k}}\right)(B) \right\rbrace \\
            &\le (\widetilde{f}_{k})_{*}\mu_{Y_{k}}(B) + \frac{1}{a_{n_{0}}}\left(\sum_{n=1}^{N}a_{n}\mu_{X_{n}}-a_{n_{0}}(f_{k})_{*}\mu_{Y_{k}}\right)(X_{n_{0}}) \\
            &=(\widetilde{f}_{k})_{*}\mu_{Y_{k}}(B) + 1-(f_{k})_{*}\mu_{Y_{k}}(X_{n_{0}}) = (\widetilde{f}_{k})_{*}\mu_{Y_{k}}(B) + \mu_{Y_{k}}(Y_{k}\setminus\widetilde{Y}_{k}).
    \end{align*}
    These inequalities and Lemma~\ref{lem:dP < dTV} imply that \[\dP(\mu_{X_{n_{0}}},\, (\widetilde{f}_{k})_{*}\mu_{Y_{k}}) \le d_{\mathrm{TV}}(\mu_{X_{n_{0}}},\, (\widetilde{f}_{k})_{*}\mu_{Y_{k}}) \le \mu_{Y_{k}}(Y_{k}\setminus\widetilde{Y}_{k}),\]
    and we obtain \[X_{n_{0}} \prec_{\mu_{Y_{k}}(Y_{k}\setminus\widetilde{Y}_{k})} Y_{k},\]that is, \[X_{n_{0}} \prec_{\mu_{Y_{k}}(Y_{k}\setminus\widetilde{Y}_{k})}\mathcal{P}\]for $k \in \N$.
    By Proposition~\ref{prop:Pn <en Qn => P < Q}, we have $X_{n_{0}} \in \mathcal{P}$.
    This completes the proof.
\end{proof}

\begin{lem}\label{lem:the direct sum of Pn include the direct sum of Xn, then Xn in some Pm}
    Let $(a_{n})_{n=1}^{N} \in \mathcal{A}_{1}\cap \mathcal{A}$, and let $\{X_{n}\}_{n=1}^{N}$ be a sequence of mm-spaces and $\{\mathcal{P}_{n}\}_{n=1}^{N}$ a sequence of pyramids of finite observable diameter. 
    If \[\mathcal{P}_{\sum_{n=1}^{N} X_{n}^{a_{n}}} \subset \sum_{n=1}^{N} \mathcal{P}_{n}^{a_{n}},\]
    then there exists a bijective map $f:[N] \to [N]$ such that $a_{n} = a_{f(n)}$ and $X_{n} \in \mathcal{P}_{f(n)}$ for each $n \in [N]$.
\end{lem}

\begin{proof}
    Let $\{n_{k}\}_{k=0}^{M}$ be a monotone increasing sequence of integers, where $M \in \overline{\N}$, such that $n_{0}=0$ and \[a_{1}=a_{2}= \cdots = a_{n_{1}} > a_{n_{1}+1}=a_{n_{1}+2} = \cdots = a_{n_{2}} > \cdots.\]
    More precisely, it satisfies the following properties:
    \begin{itemize}
        \item for every $n \in [N]$, there exists $k \in [M]$ such that $n_{k-1} < n \le n_{k}$;
        \item for each $k \in [M]$ and $n$ with $n_{k-1} < n \le n_{k}$, we have $a_{n}=a_{n_{k}}$;
        \item if $N < +\infty$, then $M < +\infty$ and $n_{M} = N$.
    \end{itemize}
    We will prove later that there exist permutations $\sigma_{k} \in \mathfrak{S}_{n_{k}-n_{k-1}}$ for $k \in [M]$ such that $X_{n} \in \mathcal{P}_{n_{k-1}+\sigma_{k}(n-n_{k-1})}$ for any $n$ with $n_{k-1}+1 \le n \le n_{k}$.
    Note that we denote by $\mathfrak{S}_{n}$ the set of all permutations on $\{1,2,\dots, n\}$.
    Then, we define a map $f:[N] \to [N]$ by $f(n):=n_{k-1} + \sigma_{k}(n-n_{k-1})$ for $n_{k-1}+1 \le n \le n_{k}$ and complete the proof.
    
    For $k \in \mathbb{N}$, we define mm-spaces $Z_{k} := (\sum_{n=1}^{N}X_{n}^{a_{n}})_{k}$.
    Since $Z_{k} \in \mathcal{P} := \sum_{n=1}^{N}\mathcal{P}_{n}^{a_{n}}$, there exist mm-spaces $Y_{nk} \in \mathcal{P}_{n}$ and 1-Lipschtz maps $f_{nk}:Y_{nk} \to Z_{k}$ for $n \in [N]$ and $k \in \mathbb{N}$ such that 
    \begin{equation}\label{eq:lem:the direct sum of Pn include the direct sum of Xn, then Xn in some Pm_1}
        \sum_{n=1}^{N} a_{n}(f_{nk})_{*}\mu_{Y_{nk}} = \mu_{Z_{k}} = \sum_{n=1}^{N} a_{n}\mu_{X_{n}}.
    \end{equation}

    First, we prove the following claim.
    \begin{clm}\label{clm:clm:lem:the direct sum of Pn include the direct sum of Xn, then Xn in some Pm 1}
        For any $n \in [N]$ and any subsequence $\{k(\ell)\}_{\ell=1}^{\infty} \subset \{k\}_{k=1}^{\infty}$, there exist a number $m_{n} \in [N]$ and a subsequence $\{k^{\prime}(\ell)\}_{\ell=1}^{\infty}$ of $\{k(\ell)\}_{\ell=1}^{\infty}$ such that
        \begin{enumerate}
            \item $a_{m_{n}} \ge a_{n}$;
            \item $\lim_{\ell\to\infty}\mu_{Y_{nk^{\prime}(\ell)}}(f_{nk^{\prime}(\ell)}^{-1}(X_{m_{n}})) = 1$;
            \item $\limsup_{\ell\to\infty}a_{s}\mu_{Y_{sk^{\prime}(\ell)}}(f_{sk^{\prime}(\ell)}^{-1}(X_{m_{n}})) \le a_{m_{n}}-a_{n}$ for any $s \in [N]$ with $s \neq n$.
        \end{enumerate}
    \end{clm}

    \begin{proof}
        We take any subsequence $\{k(\ell)\}_{\ell=1}^{\infty}$ of $\{k\}_{k=1}^{\infty}$ and fix it.
        First, we prove that there exist a number $m_{n} \in [N]$ and a subsequence $\{\kpl\}_{\ell=1}^{\infty}$ such that 
        \begin{equation}\label{eq:clm:clm:lem:the direct sum of Pn include the direct sum of Xn, then Xn in some Pm 1_1}
            \liminf_{\ell\to\infty}\mu_{Y_{n\kpl}}(f_{n\kpl}^{-1}(X_{m_{n}})) > 0
        \end{equation}
        by contradiction.
        We assume the contrary. 
        Then we see that for any number $m \in [N]$ and any subsequence $\{k(\ell(s))\}_{s=1}^{\infty}$ of $\{\kl\}_{\ell=1}^{\infty}$, we have \[\liminf_{s \to \infty}\mu_{Y_{nk(\ell(s))}}(f_{nk(\ell(s))}^{-1}(X_{m})) = 0,\]which implies that, for any $m \in [N]$, we have \[\lim_{\ell\to\infty}\mu_{Y_{n\kl}}(f_{n\kl}^{-1}(X_{m})) = 0.\]
        Since we observe that
        \begin{itemize}
            \item $Y_{n\kl} = \bigcup_{m=1}^{N} f_{n\kl}^{-1}(X_{m})$; 
            \item $\lim_{\ell\to\infty} \mu_{Y_{n\kl}}(f_{n\kl}^{-1}(X_{m})) = 0$ for any $m \in [N]$;
            \item $d_{Y_{n\kl}}(f_{n\kl}^{-1}(X_{m}),\, f_{n\kl}^{-1}(X_{m^{\prime}})) \ge \kl$ for any $m, m^{\prime} \in [N]$ with $m \neq m^{\prime}$,
        \end{itemize}
        by Lemma~\ref{lem:infinitely dissipation criterion} the sequence $\{Y_{n\kl}\}_{\ell=1}^{\infty}$ infinitely dissipates. 
        Since $Y_{n\kl} \in \mathcal{P}_{n}$ for every $\ell \in \N$, it follows from Proposition~\ref{prop:infinitely dissipated sequence converges weakly to the maximum pyramid} that $\mathcal{X} \subset \mathcal{P}_{n}$, which contradicts the finiteness of the observable diameter of $\mathcal{P}_{n}$. 
        This completes the proof of the assertion.
        Take such $m_{n}$ and $\{\kpl\}_{\ell=1}^{\infty}$ and fix them.

        We next prove (ii) of this claim by contradiction.
        Assume that 
        \begin{equation}\label{eq:clm:clm:lem:the direct sum of Pn include the direct sum of Xn, then Xn in some Pm 1_2}
            \beta := \liminf_{\ell\to\infty}\mu_{Y_{n\kpl}}(f_{n\kpl}^{-1}(X_{m_{n}})) < 1.
        \end{equation}
        By~\eqref{eq:clm:clm:lem:the direct sum of Pn include the direct sum of Xn, then Xn in some Pm 1_1}, we have $0 < \beta < 1$.
        Take any $\beta_{0}$ with $\beta < \beta_{0} <1$ and fix it.
        By the assumption~\eqref{eq:clm:clm:lem:the direct sum of Pn include the direct sum of Xn, then Xn in some Pm 1_2}, we find a subsequence $\{k^{\prime}(\ell(s))\}_{s=1}^{\infty}$ of $\{\kpl\}_{\ell=1}^{\infty}$ such that \[0 < \beta/2 \le \mu_{Y_{nk^{\prime}(\ell(s))}}(f_{nk^{\prime}(\ell(s))}^{-1}(X_{m_{n}})) <\beta_{0}\] for any $s \in \N$.
        For this subsequence, we also have \[0 < 1-\beta_{0} \le \mu_{Y_{nk^{\prime}(\ell(s))}}(f_{nk^{\prime}(\ell(s))}^{-1}(Z_{k^{\prime}(\ell(s))}\setminus X_{m_{n}})).\]
        For any $\kappa$ with $0 < \kappa < \min\{\beta/2,\hspace{1mm}1-\beta_{0}\}$, by~\cite[Proposition 2.26 (2)]{S2016book}, we observe that
        \begin{align*}
            \obsdiam{\mathcal{P}_{n}}{\kappa} &\ge \obsdiam{Y_{nk^{\prime}(\ell(s))}}{\kappa} \ge \mathrm{Sep}(Y_{nk^{\prime}(\ell(s))};\beta/2,\hspace{1mm}1-\beta_{0})\\
                &\ge d_{Y_{nk^{\prime}(\ell(s))}}(f_{nk^{\prime}(\ell(s))}^{-1}(X_{m_{n}}),\hspace{1mm}f_{nk^{\prime}(\ell(s))}^{-1}(Z_{k^{\prime}(\ell(s))}\setminus X_{m_{n}}))\\
                &\ge k^{\prime}(\ell(s)) \hspace{1mm}\longrightarrow \hspace{1mm}+\infty
        \end{align*}
        as $s \to \infty$.
        Note that the last inequality is followed by the fact that the maps $f_{nk^{\prime}(\ell(s))}$ are 1-Lipschitz.
        This contradicts the finiteness of the observable diameter of $\mathcal{P}_{n}$ and we obtain (ii) of this claim.
        
        By the equation~\eqref{eq:lem:the direct sum of Pn include the direct sum of Xn, then Xn in some Pm_1}, we see that 
        \begin{align*}
            a_{m_{n}} &= \sum_{s=1}^{N}a_{s}\mu_{X_{s}}(X_{m_{n}}) = \sum_{s=1}^{N}a_{s}(f_{s\kpl})_{*}\mu_{Y_{s\kpl}}(X_{m_{n}}) \\
                &\ge a_{n}(f_{n\kpl})_{*}\mu_{Y_{n\kpl}}(X_{m_{n}}) \longrightarrow a_{n}
        \end{align*}
        as $\ell \to \infty$.
        This implies that (i) $a_{m_{n}} \ge a_{n}$. 
        Furthermore, for any $s \in [N]$ with $s \neq n$, we have \[a_{s}(f_{s\kpl})_{*}\mu_{Y_{s\kpl}}(X_{m_{n}}) \le a_{m_{n}} - a_{n}(f_{n\kpl})_{*}\mu_{Y_{n\kpl}}(X_{m_{n}}) \longrightarrow a_{m_{n}} - a_{n}\] as $\ell \to \infty$, which implies (iii) of this claim.
        We finish the proof.
    \end{proof}
    We now construct the above mentioned permutation $\sigma_{1} \in \mathfrak{S}_{n_{1}}$.
    By Claim~\ref{clm:clm:lem:the direct sum of Pn include the direct sum of Xn, then Xn in some Pm 1}, there exist a number $m_{1} \in [N]$ and a subsequence $\{k_{\ell}^{(1)}\}_{\ell=1}^{\infty}$ of $\{k\}_{k=1}^{\infty}$ such that 
    \begin{enumerate}[label=(\textbf{1.\arabic*}), leftmargin=*]
        \item $a_{m_{1}} \ge a_{1}$;
        \item $\lim_{\ell\to\infty}\mu_{Y_{1k_{\ell}^{(1)}}}(f_{1k_{\ell}^{(1)}}^{-1}(X_{m_{1}})) = 1$;
        \item $\limsup_{\ell\to\infty}\mu_{Y_{sk_{\ell}^{(1)}}}(f_{sk_{\ell}^{(1)}}^{-1}(X_{m_{1}})) \le a_{m_{1}} - a_{1}$ for any $s \in [N]$ with $s \neq 1$.
    \end{enumerate}
    From (\textbf{1.1}) and the monotonicity of $(a_{n})_{n=1}^{N}$, we have $a_{1}=a_{m_{1}}$, $1 \le m_{1} \le n_{1}$, and 
    \begin{equation*}
        (\textbf{1.3}^{\mathbf{\prime}})\hspace{1mm}\textstyle\lim_{\ell\to\infty}\mu_{Y_{sk_{\ell}^{(1)}}}(f_{sk_{\ell}^{(1)}}^{-1}(X_{m_{1}})) = 0\hspace{2mm} \text{for any}\hspace{1mm} s \in [N]\hspace{1mm} \text{with}\hspace{1mm} s \neq 1.
    \end{equation*}
    By Lemma~\ref{lem:lem of lem:the direct sum of Pn include the direct sum of Xn, then Xn in some Pm 1} and $(\mathbf{1.2})$, we have $X_{m_{1}} \in \mathcal{P}_{1}$.
    If $n_{1}=1$, then $m_{1}=1$, and we set $\sigma_{1} = \mathrm{id}$, the identity of $\mathfrak{S}_{1}$.
    If not, then by Claim~\ref{clm:clm:lem:the direct sum of Pn include the direct sum of Xn, then Xn in some Pm 1}, there exist a number $m_{2} \in [N]$ and a subsequence $\{k_{\ell}^{(2)}\}_{\ell=1}^{\infty}$ of $\{k_{\ell}^{(1)}\}_{\ell=1}^{\infty}$ such that 
    \begin{enumerate}[label=(\textbf{2.\arabic*}), leftmargin=*]
        \item $a_{m_{2}} \ge a_{2} = a_{1}$;
        \item $\lim_{\ell\to\infty}\mu_{Y_{2k_{\ell}^{(2)}}}(f_{2k_{\ell}^{(2)}}^{-1}(X_{m_{2}})) = 1$;
        \item $\limsup_{\ell\to\infty}\mu_{Y_{sk_{\ell}^{(2)}}}(f_{sk_{\ell}^{(2)}}^{-1}(X_{m_{2}})) \le a_{m_{2}} - a_{2}$ for any $s \in [N]$ with $s \neq 2$.
    \end{enumerate}
    From (\textbf{2.1}) and the monotonicity of $(a_{n})_{n=1}^{N}$, we have $a_{m_{2}} = a_{1}$, $1 \le m_{2} \le n_{1}$, and \[(\textbf{2.3}^{\mathbf{\prime}})\hspace{1mm}\textstyle\lim_{\ell\to\infty}\mu_{Y_{sk_{\ell}^{(2)}}}(f_{sk_{\ell}^{(2)}}^{-1}(X_{m_{2}})) = 0\hspace{2mm} \text{for any}\hspace{1mm} s \in [N]\hspace{1mm} \text{with}\hspace{1mm} s \neq 2.\]
    From (\textbf{1.2}) and ($\mathbf{2.3}^{\mathbf{\prime}}$), we have $m_{1} \neq m_{2}$.
    By Lemma~\ref{lem:lem of lem:the direct sum of Pn include the direct sum of Xn, then Xn in some Pm 1} and $(\mathbf{2.2})$, we have $X_{m_{2}} \in \mathcal{P}_{2}$.
    If $n_{1}=2$, then we define the permutation $\sigma_{1}$ by $\sigma_{1}(m_{i}) := i$, $i=1,2$.
    If not, by repeating this process, we obtain numbers $m_{1}, m_{2},\dots, m_{n_{1}} \in [N]$ and subsequences $\{k_{\ell}^{(n_{1})}\}_{\ell=1}^{\infty} \subset \cdots \{k_{\ell}^{(2)}\}_{\ell=1}^{\infty} \subset \{k_{\ell}^{(1)}\}_{\ell=1}^{\infty}$ such that
    \begin{enumerate}[label=$(\mathbf{n_{1}.\arabic*})$, leftmargin=*]
        \item $\{m_{1},m_{2},\dots, m_{n_{1}}\}=\{1,2,\dots, n_{1}\}$;
        \item $\lim_{\ell\to\infty}\mu_{Y_{nk_{\ell}^{(n_{1})}}}(f_{nk_{\ell}^{(n_{1})}}^{-1}(X_{m_{n}})) = 1$ for $n=1,2,\dots,n_{1}$;
        \item $\lim_{\ell\to\infty}\mu_{Y_{sk_{\ell}^{(n_{1})}}}(f_{sk_{\ell}^{(n_{1})}}^{-1}(X_{m_{n}})) = 0$ for any $s \in [N]$ with $s \neq n$, $n=1,2,\dots, n_{1}$.
    \end{enumerate}
    We define a permutation $\sigma_{1}$ by \[\sigma_{1}(m_{n}) := n\hspace{2mm}\text{for}\hspace{1mm}n=1,2,\dots,n_{1}.\]
    By Lemma~\ref{lem:lem of lem:the direct sum of Pn include the direct sum of Xn, then Xn in some Pm 1} and $(\mathbf{n_{1}.2})$, we have $X_{m_{n}} \in \mathcal{P}_{n}$, that is, $X_{n} \in \mathcal{P}_{\sigma_{1}(n)}$, $n=1,2,\dots, n_{1}$.

    Next, we construct the above mentioned permutation $\sigma_{2} \in \mathfrak{S}_{n_{2}-n_{1}}$.
    By Claim~\ref{clm:clm:lem:the direct sum of Pn include the direct sum of Xn, then Xn in some Pm 1}, there exist a number $m_{n_{1}+1} \in [N]$ and a subsequence $\{k_{\ell}^{(n_{1}+1)}\}_{\ell=1}^{\infty}$ of $\{k_{\ell}^{(n_{1})}\}_{\ell=1}^{\infty}$ such that 
    \begin{enumerate}[label=$(\mathbf{n_{1}+1.\arabic*})$, leftmargin=*]
        \item $a_{m_{n_{1}+1}} \ge a_{n_{1}+1}$;
        \item $\lim_{\ell\to\infty}\mu_{Y_{n_{1}+1,\,k_{\ell}^{(n_{1}+1)}}}(f_{n_{1}+1,\,k_{\ell}^{(n_{1}+1)}}^{-1}(X_{m_{n_{1}+1}})) = 1$;
        \item $\limsup_{\ell\to\infty}\mu_{Y_{sk_{\ell}^{(n_{1}+1)}}}(f_{sk_{\ell}^{(n_{1}+1)}}^{-1}(X_{m_{n_{1}+1}})) \le a_{m_{n_{1}+1}} - a_{n_{1}+1}$ for any $s \in [N]$ with $s \neq n_{1}+1$.
    \end{enumerate}
    From $(\mathbf{n_{1}.3})$ and $(\mathbf{n_{1}+1.2})$, we have $m_{n_{1}+1} \in [N]\setminus\{1,2,\dots,n_{1}\}$, that is, $m_{n_{1}+1} \ge n_{1}+1$.
    Furthermore, from $(\mathbf{n_{1}+1.1})$ and the monotonicity of $(a_{n})_{n=1}^{N}$, we have $a_{m_{n_{1}+1}} = a_{n_{1}+1}$, $n_{1}+1 \le m_{n_{1}+1} \le n_{2}$, and 
    \begin{align*}
        (\mathbf{n_{1}+1.3}^{\mathbf{\prime}})\hspace{1mm} &\lim_{\ell\to\infty}\mu_{Y_{sk_{\ell}^{(n_{1}+1)}}}(f_{sk_{\ell}^{(n_{1}+1)}}^{-1}(X_{m_{n_{1}+1}})) = 0\hspace{2mm} \text{for any}\hspace{1mm} s \in [N]\hspace{1mm}\text{with}\\& s \neq n_{1}+1.
    \end{align*}
    By Lemma~\ref{lem:lem of lem:the direct sum of Pn include the direct sum of Xn, then Xn in some Pm 1} and $(\mathbf{n_{1}+1.2})$, we have $X_{m_{n_{1}+1}} \in \mathcal{P}_{n_{1}+1}$.
    If $n_{2}=n_{1}+1$, then $m_{n_{1}+1}=n_{1}+1$, and we set $\sigma_{2} = \mathrm{id}$, the identity of $\mathfrak{S}_{n_{2}-n_{1}}$.
    If not, then by Claim~\ref{clm:clm:lem:the direct sum of Pn include the direct sum of Xn, then Xn in some Pm 1}, there exist a number $m_{n_{1}+2} \in [N]$ and a subsequence $\{k_{\ell}^{(n_{1}+2)}\}_{\ell=1}^{\infty}$ of $\{k_{\ell}^{(n_{1}+1)}\}_{\ell=1}^{\infty}$ such that 
    \begin{enumerate}[label=$(\mathbf{n_{1}+2.\arabic*})$, leftmargin=*]
        \item $a_{m_{n_{1}+2}} \ge a_{n_{1}+2} = a_{n_{1}+1}$;
        \item $\lim_{\ell\to\infty}\mu_{Y_{n_{1}+2,\,k_{\ell}^{(n_{1}+2)}}}(f_{n_{1}+2,\,k_{\ell}^{(n_{1}+2)}}^{-1}(X_{m_{n_{1}+2}})) = 1$;
        \item $\limsup_{\ell\to\infty}\mu_{Y_{sk_{\ell}^{(n_{1}+2)}}}(f_{sk_{\ell}^{(n_{1}+2)}}^{-1}(X_{m_{n_{1}+2}})) \le a_{m_{n_{1}+2}} - a_{n_{1}+2}$ for any $s \in [N]$ with $s \neq n_{1}+2$.
    \end{enumerate}
    From $(\mathbf{n_{1}.3})$ and $(\mathbf{n_{1}+2.2})$, we have $m_{n_{1}+2} \ge n_{1}+1$.
    Furthermore, from $(\mathbf{n_{1}+2.1})$ and the monotonicity of $(a_{n})_{n=1}^{N}$, we have $a_{m_{n_{1}+2}} = a_{n_{1}+1}$, $n_{1}+1 \le m_{n_{1}+2} \le n_{2}$, and 
    \begin{align*}
        (\mathbf{n_{1}+2.3}^{\mathbf{\prime}})\hspace{1mm} &\lim_{\ell\to\infty}\mu_{Y_{sk_{\ell}^{(n_{1}+2)}}}(f_{sk_{\ell}^{(n_{1}+2)}}^{-1}(X_{m_{n_{1}+2}})) = 0 \hspace{2mm}\text{for any}\hspace{1mm} s \in [N] \hspace{1mm} \text{with}\\ &s \neq n_{1}+2.
    \end{align*}
    From $(\mathbf{n_{1}+1.2})$ and $(\mathbf{n_{1}+2.3^{\prime}})$, we have $m_{n_{1}+1} \neq m_{n_{1}+2}$.
    By Lemma~\ref{lem:lem of lem:the direct sum of Pn include the direct sum of Xn, then Xn in some Pm 1} and $(\mathbf{n_{1}+2.2})$, we have $X_{m_{n_{1}+2}} \in \mathcal{P}_{n_{1}+2}$.
    If $n_{2}=n_{1}+2$, then we define the permutation $\sigma_{2}$ by $\sigma_{2}(m_{n_{1}+i}-n_{1}) := i$, $i=1,2$.
    If not, by repeating this process, we obtain numbers $m_{n_{1}+1}, m_{n_{1}+2},\dots, m_{n_{2}} \in [N]$ and subsequences $\{k_{\ell}^{(n_{2})}\}_{\ell=1}^{\infty} \subset \cdots \subset\{k_{\ell}^{(n_{1}+2)}\}_{\ell=1}^{\infty} \subset \{k_{\ell}^{(n_{1}+1)}\}_{\ell=1}^{\infty}$ such that
    \begin{enumerate}[label=$(\mathbf{n_{2}.\arabic*})$, leftmargin=*]
        \item $\{m_{n_{1}+1},m_{n_{1}+2},\dots, m_{n_{2}}\}=\{n_{1}+1,n_{1}+2,\dots, n_{2}\}$;
        \item $\lim_{\ell\to\infty}\mu_{Y_{nk_{\ell}^{(n_{2})}}}(f_{nk_{\ell}^{(n_{2})}}^{-1}(X_{m_{n}})) = 1$ for $n=n_{1}+1,n_{1}+2,\dots,n_{2}$;
        \item $\lim_{\ell\to\infty}\mu_{Y_{sk_{\ell}^{(n_{2})}}}(f_{sk_{\ell}^{(n_{2})}}^{-1}(X_{m_{n}})) = 0$ for any $s \in [N]$ with $s \neq n$, $n=n_{1}+1,n_{1}+2,\dots, n_{2}$.
    \end{enumerate}
    We define a permutation $\sigma_{2}$ by \[\sigma_{2}(m_{n}-n_{1}) := n-n_{1}\hspace{2mm}\text{for}\hspace{1mm}n=n_{1}+1,n_{1}+2,\dots,n_{2}.\]
    By Lemma~\ref{lem:lem of lem:the direct sum of Pn include the direct sum of Xn, then Xn in some Pm 1} and $(\mathbf{n_{2}.2})$, we have $X_{m_{n}} \in \mathcal{P}_{n}$, that is, $X_{n} \in \mathcal{P}_{n_{1}+\sigma_{2}(n-n_{1})}$, $n=n_{1}+1,n_{1}+2,\dots, n_{2}$.

    By repeating this process, we obtain permutations $\sigma_{k} \in \mathfrak{S}_{n_{k}-n_{k-1}}$, $k \in [M]$, such that $X_{n} \in \mathcal{P}_{n_{k-1}+\sigma_{k}(n-n_{k-1})}$ for $n=n_{k-1}+1, n_{k-1}+2,\dots,n_{k}$.
    This completes the proof.
\end{proof}

\begin{cor}\label{cor:P1+...+PN=Q1+...+QN, then Pn=Qm(n)}
    Let $N \in \N$ and let $\mathcal{P}_{n}$, $\mathcal{Q}_{n}$, $n = 1,2,\dots, N$, be pyramids of finite observable diameter. 
    If \[\mathcal{P}_{1}^{\frac{1}{N}} + \cdots \mathcal{P}_{N}^{\frac{1}{N}} = \mathcal{Q}_{1}^{\frac{1}{N}} +  \cdots \mathcal{Q}_{N}^{\frac{1}{N}},\]
    then there exists a permutation $\sigma \in \mathfrak{S}_{N}$ such that $\mathcal{P}_{n} = \mathcal{Q}_{\sigma(n)}$ for each $n = 1,2,\dots, N$.
\end{cor}

\begin{proof}
    By Lemma~\ref{lem:app. seq. of pyramid}, for $n = 1,2,\dots,N$ there exists an approximation sequence $\{X_{nk}\}_{k=1}^{\infty}$ of the pyramid $\mathcal{P}_{n}$. 
    Since $\mathcal{P}_{\sum_{n=1}^{N}X_{nk}^{1/N}} \subset \sum_{n=1}^{N}\mathcal{Q}_{n}^{1/N}$ for any $k$, there exists a permutation $\sigma_{k} \in \mathfrak{S}_{N}$ such that $X_{nk} \in \mathcal{Q}_{\sigma_{k}(n)}$ for each $n$ by Lemma~\ref{lem:the direct sum of Pn include the direct sum of Xn, then Xn in some Pm}. 
    By the finiteness of the number $N$, there exist a subsequence $\{X_{1k_{\ell}^{(1)}}\}_{\ell=1}^{\infty}$ of $\{X_{1k}\}_{k=1}^{\infty}$ and a number $n_{1}$ with $1\le n_{1} \le N$ such that $X_{1k_{\ell}^{(1)}} \in \mathcal{Q}_{n_{1}}$ for any $\ell$. 
    By the same argument, there exist a subsequence $\{X_{2k_{\ell}^{(2)}}\}_{\ell=1}^{\infty}$ of $\{X_{2k_{\ell}^{(1)}}\}_{\ell=1}^{\infty}$ and a number $n_{2}$ with $1 \le n_{2} \le N$ and $n_{2}\neq n_{1}$ such that $X_{2k_{\ell}^{(2)}} \in \mathcal{Q}_{n_{2}}$ for any $\ell$. 
    We repeat this process, which ends in finite steps, and we obtain a subsequence $\{X_{nk(\ell)}\}_{\ell=1}^{\infty}$ of $\{X_{nk}\}_{k=1}^{\infty}$  and a permutation $\sigma \in \mathfrak{S}_{N}$ such that $X_{nk(\ell)} \in \mathcal{Q}_{\sigma(n)}$ for any $n$ and any $\ell \in \mathbb{N}$. 
    Letting $\ell \to \infty$, we obtain $\mathcal{P}_{n} \subset \mathcal{Q}_{\sigma(n)}$ for each $n$.
    By the same argument, we find a permutation $\tau \in \mathfrak{S}_{N}$ such that $\mathcal{Q}_{n} \subset \mathcal{P}_{\tau(n)}$ for each $n$. 
    Set $m$ as the order of the permutation $\rho := \tau \circ \sigma$. 
    Then, for each $n = 1,2,\dots,N$, we have
    \begin{align*}
        &\mathcal{P}_{n} \subset \mathcal{Q}_{\sigma(n)} \subset \mathcal{P}_{\tau(\sigma(n))} = \mathcal{P}_{\rho(n)} \subset \mathcal{Q}_{\sigma(\rho(n))} \subset \mathcal{P}_{\tau(\sigma(\rho(n)))} = \mathcal{P}_{\rho^{2}(n)} \subset \cdots \\
        &\hspace{10mm} \cdots \subset \mathcal{Q}_{\sigma(\rho^{m-1}(n))} \subset \mathcal{P}_{\tau(\sigma(\rho^{m-1}(n)))} = \mathcal{P}_{\rho^{m}(n)} = \mathcal{P}_{n}, 
    \end{align*}
    which implies that \[\mathcal{P}_{n} = \mathcal{P}_{\rho^{\ell}(n)} \hspace{3mm}\text{and}\hspace{3mm} \mathcal{P}_{n} = \mathcal{Q}_{\sigma(n)}\] for any $n$ and any $\ell = 1,2,\dots, m-1$.
    This completes the proof.
\end{proof}

\begin{dfn}\label{def:partially dissipative pyramids}
    Let $\mathcal{P}$ be a pyramid. 
    We say that the pyramid $\mathcal{P}$ is \textit{partially infinitely dissipated} if and only if there exists $A \in \A$ with $\|A\|_{1} < 1$ such that \[t\mathcal{P} \longrightarrow \mathcal{P}_{A} \hspace{5mm} \text{weakly as}\hspace{2mm} t \to 0+.\]
\end{dfn}

\begin{rem}
    We note that a pyramid $\mathcal{P}$ is partially infinitely dissipated if and only if there exists $A \in \A$ with $\|A\|_{1} < 1$ such that the sequence $\{\frac{1}{n}\mathcal{P}\}_{n=1}^{\infty}$ \textit{infinitely dissipates with atoms} $A$ (for details, see~\cite{EKM2024}).
\end{rem}

\begin{lem}\label{lem:not partially dissipative pyramids}
    Let $\mathcal{P}$ be a pyramid. 
    If there exist a real number $0 < \alpha < 1$ and a pyramid $\mathcal{Q}$ such that \[\mathcal{X}^{1-\alpha} + \mathcal{Q}^{\alpha} \subset \mathcal{P},\] then the pyramid $\mathcal{P}$ is partially infinitely dissipated.    
\end{lem}

\begin{proof}
    By Lemma~\ref{lem:the limit of pyramids tP when t to 0+}, there exist $A, B \in \mathcal{A}$ such that $t\mathcal{P}$ and $t\mathcal{Q}$ converge weakly to $\mathcal{P}_{A}$ and $\mathcal{P}_{B}$ as $t \to 0+$, respectively.
    By Lemma~\ref{lem:scale transformation of direct sums of pyramids}, for any $t > 0$, we have $\mathcal{X}^{1-\alpha} + (t\mathcal{Q})^{\alpha} \subset t\mathcal{P}$. 
    Letting $t \to 0+$, we see that $\mathcal{P}_{\alpha B} = \mathcal{X}^{1-\alpha} + \mathcal{P}_{B}^{\alpha} \subset \mathcal{P}_{A}$, which implies that $\|A\|_{1} \le \alpha\|B\|_{1} < 1$. This completes the proof.
\end{proof}

\begin{lem}\label{lem:X+P=X+Q, the P=Q}
    Let $\mathcal{P}$ and $\mathcal{Q}$ be two pyramids that are not partially infinitely dissipated. 
    If there exists a real number $\alpha$ with $0 < \alpha < 1$ such that\[\mathcal{X}^{1-\alpha} + \mathcal{P}^{\alpha} = \mathcal{X}^{1-\alpha} + \mathcal{Q}^{\alpha},\]
    then we have $\mathcal{P} = \mathcal{Q}$.
\end{lem}

\begin{proof}
   We first prove that $\mathcal{P} \subset \mathcal{Q}$. 
   Take any $X \in \mathcal{P}$ and set $Z_{n} := (\mathbb{D}_{n}^{1-\alpha}+X^{\alpha})_{n}$ for $n \in \N$.
   Since $Z_{n} \in \mathcal{X}^{1-\alpha}+\mathcal{Q}^{\alpha}$, there exist mm-spaces $Y_{n} \in \mathcal{Q},\, W_{n} \in \mathcal{X}$, and 1-Lipschitz maps $f_{n}:Y_{n} \to Z_{n},\, g_{n}:W_{n}\to Z_{n}$ for $n \in \N$ such that 
   \begin{equation}\label{eq:lem:X+P=X+Q, the P=Q_1}
        (1-\alpha)\mu_{\mathbb{D}_{n}}+\alpha\mu_{X} = \mu_{Z_{n}} = (1-\alpha)(g_{n})_{*}\mu_{W_{n}} + \alpha(f_{n})_{*}\mu_{Y_{n}}.
    \end{equation} 
   \begin{clm}\label{clm:lem:X+P=X+Q, the P=Q}
        We have 
       \[\lim_{n\to\infty} \mu_{Y_{n}}(f_{n}^{-1}(X)) = 1.\]
   \end{clm}

   \begin{proof}
       We first prove that \[\lim_{n\to\infty} \mu_{Y_{n}}(f_{n}^{-1}(\mathbb{D}_{n})) = 0\]by contradiction.
       Suppose that $\limsup_{n\to\infty}\mu_{Y_{n}}(f_{n}^{-1}(\mathbb{D}_{n})) > 0$. 
       Then there exist a subsequence $\{n(k)\}_{k=1}^{\infty}$ of $\{n\}_{n=1}^{\infty}$ and a real number $0 < \beta \le 1$ such that for any $k \in \N$, \[\lim_{k\to\infty}\mu_{Y_{n(k)}}(f_{n(k)}^{-1}(\mathbb{D}_{n(k)})) = \beta \hspace{3mm}\text{and} \hspace{3mm} \mu_{Y_{n(k)}}(f_{n(k)}^{-1}(\mathbb{D}_{n(k)})) > 0.\] 
       Define mm-spaces \[A_{k} := \left(Y_{n(k)},\, d_{Y_{n(k)}},\, \mu_{A_{k}}:=\frac{1}{\mu_{Y_{n(k)}}(f_{n(k)}^{-1}(\mathbb{D}_{n(k)}))}\mu_{Y_{n(k)}}\bigg|_{f_{n(k)}^{-1}(\mathbb{D}_{n(k)})} \right).\] 
       By the definition of $\mathbb{D}_{n(k)}$ and the equation~\eqref{eq:lem:X+P=X+Q, the P=Q_1}, we see that 
       \begin{itemize}
           \item $\supp{\mu_{A_{k}}} \subset \bigsqcup_{i=1}^{n(k)} f_{n(k)}^{-1}(\{s_{n(k)}^{i}\}) = f_{n(k)}^{-1}(\mathbb{D}_{n(k)})$;
           \item $d_{Y_{n(k)}}(f_{n(k)}^{-1}(\{s_{n(k)}^{i}\}),\, f_{n(k)}^{-1}(\{s_{n(k)}^{j}\})) \ge n(k)$ for any $i \neq j$;
           \item $\mu_{A_{k}}(f_{n(k)}^{-1}(\{s_{n(k)}^{i}\})) = \displaystyle\frac{\mu_{Y_{n(k)}}(f_{n(k)}^{-1}(\{s_{n(k)}^{i}\}))}{\mu_{Y_{n(k)}}(f_{n(k)}^{-1}(\mathbb{D}_{n(k)}))}\le \displaystyle\frac{1-\alpha}{\alpha}\cdot\frac{\mu_{\mathbb{D}_{n(k)}}(\{s_{n(k)}^{i}\}) }{\mu_{Y_{n(k)}}(f_{n(k)}^{-1}(\mathbb{D}_{n(k)}))} \longrightarrow 0$ as $k \to \infty$.
       \end{itemize}
       By Lemma~\ref{lem:infinitely dissipation criterion}, these imply that the sequence $\{A_{k}\}_{k=1}^{\infty}$ infinitely dissipates. 
       If $\beta=1$, then the sequence $\{Y_{n(k)}\}_{k=1}^{\infty}$ also infinitely dissipates, and by Proposition~\ref{prop:infinitely dissipated sequence converges weakly to the maximum pyramid}, $\mathcal{P}_{Y_{n(k)}}$ converges weakly to $\X$, which contradicts the hypothesis that $\mathcal{Q}$ is not partially infinitely dissipated.
       Suppose $0<\beta < 1$. 
       Then, for large $k$, we have $\mu_{Y_{n(k)}}(f_{n(k)}^{-1}(X)) >0$.
       For such large $k$, we define \[B_{k}:=\left(Y_{n(k)}, d_{Y_{n(k)}}, \frac{1}{\mu_{Y_{n(k)}}(f_{n(k)}^{-1}(X))}\mu_{Y_{n(k)}}\Bigg|_{f_{n(k)}^{-1}(X)} \right),\] which are mm-spaces.
       By Theorem~\ref{thm:metric rho}, there exists a subsequence $\{B_{k(\ell)}\}_{\ell=1}^{\infty}$ of $\{B_{k}\}_{k=1}^{\infty}$ such that $\mathcal{P}_{B_{k(\ell)}}$ converges weakly to some pyramid $\mathcal{R}$.
       Since \[Y_{n(k(\ell))} \in \mathcal{X}((A_{k(\ell)}, B_{k(\ell)});(\alpha_{k(\ell)},1-\alpha_{k(\ell)});n(k(\ell))),\] where $\alpha_{k(\ell)} := \mu_{Y_{n(k(\ell))}}(f_{n(k(\ell))}^{-1}(\mathbb{D}_{n(k(\ell))}))$, by Corollary~\ref{cor:the sum of Xnk's approximates the sum of pyramids}, we see that $\mathcal{P}_{Y_{n(k(\ell))}}$ converges weakly to the pyramid $\mathcal{X}^{\beta}+\mathcal{R}^{1-\beta}$.
       Since $Y_{n(k(\ell))} \in \mathcal{Q}$, we have $\mathcal{X}^{\beta}+\mathcal{R}^{1-\beta} \subset \mathcal{Q}$. 
       By Lemma~\ref{lem:not partially dissipative pyramids}, this implies that the pyramid $\mathcal{Q}$ is partially infinitely dissipated, which is a contradiction, and we obtain \[\lim_{n\to\infty}\mu_{Y_{n}}(f_{n}^{-1}(\mathbb{D}_{n})) = 0\] and \[\lim_{n\to\infty}\mu_{Y_{n}}(f_{n}^{-1}(X)) = 1- \lim_{n\to\infty}\mu_{Y_{n}}(f_{n}^{-1}(\mathbb{D}_{n})) = 1.\] 
       We finish the proof.
   \end{proof}
   By Lemma~\ref{lem:lem of lem:the direct sum of Pn include the direct sum of Xn, then Xn in some Pm 1} and Claim~\ref{clm:lem:X+P=X+Q, the P=Q}, we obtain $X \in \mathcal{Q}$, which implies $\mathcal{P} \subset \mathcal{Q}$. The proof of $\mathcal{P} \supset \mathcal{Q}$ is similar to that of $\mathcal{P} \subset \mathcal{Q}$. This completes the proof.
\end{proof}

\begin{proof}[Proof of Theorem~\ref{thm:the uniquness of decomposition of pyramid}]
    For any $t > 0$, we have \[\mathcal{X}^{1-\|A\|_{1}}+\sum_{n=1}^{N}(t\mathcal{P}_{n})^{a_{n}} = \mathcal{X}^{1-\|B\|_{1}}+\sum_{m=1}^{M}(t\mathcal{Q}_{m})^{b_{m}}\] by Lemma~\ref{lem:scale transformation of direct sums of pyramids}. 
    Since both $\mathcal{P}_{n}$ and $\mathcal{Q}_{n}$ have finite observable diameters, both $t\mathcal{P}_{n}$ and $t\mathcal{Q}_{m}$ converge weakly to the pyramid $\{*\}$ as $t \to 0+$ by Lemma~\ref{lem:the characterization of pyramid whose observable diameter is finite}, and we obtain \[\mathcal{X}^{1-\|A\|_{1}}+\sum_{n=1}^{N}\{*\}^{a_{n}} = \mathcal{X}^{1-\|B\|_{1}}+\sum_{m=1}^{M}\{*\}^{b_{m}},\]and hence, $\mathcal{P}_{A} = \mathcal{P}_{B}$.
    By Lemma~\ref{lem:A=B <=> P_A = P_B}, we have $A=B$, that is, $N=M$ and $a_{n}=b_{n}$ for every $n \in [N]$. 
    Since both $\sum_{n=1}^{N}\mathcal{P}_{n}^{a_{n}/\|A\|_{1}}$ and $\sum_{n=1}^{N}\mathcal{Q}_{n}^{a_{n}/\|A\|_{1}}$ are not partially infinitely dissipated, Lemma~\ref{lem:X+P=X+Q, the P=Q} shows that\[\sum_{n=1}^{N}\mathcal{P}_{n}^{a_{n}/\|A\|_{1}} = \sum_{n=1}^{N}\mathcal{Q}_{n}^{a_{n}/\|A\|_{1}}.\] 
    For simplicity, we assume that $\|A\|_{1} = 1$. 

    For $n \in [N]$, there exists an approximation sequence $\{X_{nk}\}_{k=1}^{\infty}$ of the pyramid $\mathcal{P}_{n}$ by Lemma~\ref{lem:app. seq. of pyramid}. 
    Then we have $\mathcal{P}_{\sum_{n=1}^{N}X_{nk}^{a_{n}}} \subset \sum_{n=1}^{N}\mathcal{Q}_{n}^{a_{n}}$. 
    By Lemma~\ref{lem:the direct sum of Pn include the direct sum of Xn, then Xn in some Pm}, there exist bijective maps $f_{k}:[N] \to [N]$ such that $X_{nk} \in \mathcal{Q}_{f_{k}(n)}$ and $a_{n} = a_{f_{k}(n)}$ for $k \in \mathbb{N}$ and $n \in [N]$. 
    Take a sequence $\{n_{\ell}\}_{\ell=0}^{M}$, where $M \in \overline{\mathbb{N}}$, as in the proof of Lemma~\ref{lem:the direct sum of Pn include the direct sum of Xn, then Xn in some Pm}, that is, \[a_{1}=a_{2}=\cdots = a_{n_{1}}> a_{n_{1}+1}=a_{n_{1}+2} = \cdots a_{n_{2}}> \cdots.\]
    Then, for any $\ell \in \{0\}\cup[M]$ and any $k \in \mathbb{N}$, we have \[\mathcal{P}_{\sum_{j=1}^{n_{\ell+1}-n_{\ell}}X_{n_{\ell}+j,\, k}^{1/(n_{\ell+1}-n_{\ell})}} \subset \sum_{j=1}^{n_{\ell+1}-n_{\ell}}\mathcal{Q}_{f_{k}(n_{\ell}+j)}^{1/(n_{\ell+1}-n_{\ell})} = \sum_{j=1}^{n_{\ell+1}-n_{\ell}}\mathcal{Q}_{n_{\ell}+j}^{1/(n_{\ell+1}-n_{\ell})}.\] 
    Letting $k \to \infty$, by Theorem~\ref{thm:the continuity of the direct sum of pyramids}, we obtain \[\sum_{j=1}^{n_{\ell+1}-n_{\ell}}\mathcal{P}_{n_{\ell}+j}^{1/(n_{\ell+1}-n_{\ell})} \subset \sum_{j=1}^{n_{\ell+1}-n_{\ell}}\mathcal{Q}_{n_{\ell}+j}^{1/(n_{\ell+1}-n_{\ell})}.\] 
    By the same argument, we see that the opposite inclusion holds, which implies that for any $\ell \in \{0\}\cup[M]$, we have \[\sum_{j=1}^{n_{\ell+1}-n_{\ell}}\mathcal{P}_{n_{\ell}+j}^{1/(n_{\ell+1}-n_{\ell})} = \sum_{j=1}^{n_{\ell+1}-n_{\ell}}\mathcal{Q}_{n_{\ell}+j}^{1/(n_{\ell+1}-n_{\ell})}.\] 
    By Corollary~\ref{cor:P1+...+PN=Q1+...+QN, then Pn=Qm(n)}, there exists a permutation $\sigma_{\ell} \in \mathfrak{S}_{n_{\ell+1}-n_{\ell}}$ such that $\mathcal{P}_{n_{\ell}+j} = \mathcal{Q}_{n_{\ell}+\sigma_{\ell}(j)}$ for $1 \le j \le n_{\ell+1}-n_{\ell}$ and for $\ell \in \{0\}\cup[M]$. 
    Define a map $f:[N] \to [N]$ by $f(n_{\ell}+j) := n_{\ell}+\sigma_{\ell}(j)$. 
    Then the map $f$ is bijective and satisfies $\mathcal{P}_{n} = \mathcal{Q}_{f(n)}$ and $a_{n} = a_{f(n)} = b_{f(n)}$ for every $n \in [N]$.
    We finish the proof of this theorem.
\end{proof}

\section{Covering numbers of pyramids}
In this section, we introduce the \textit{covering number} of a pyramid and prove Theorem~\ref{thm:when a pyrmid is an ext. mm sp.}.

\begin{dfn}[\cite{VL_mm_entropy}]\label{def:covering number of mm-spaces}
    For $r > 0$, $0 < \kappa < 1$, and for an mm-space $X$, we define \[\cov{X}{r}{\kappa} := \min \hspace{1mm}\{\#\mathcal{N}\mid \mathcal{N} \subset X,\hspace{1mm}\mu_{X}(B_{r}(\mathcal{N})) \ge 1-\kappa\},\] and call it the $(r,\kappa)$-\textit{covering number} of $X$. 
\end{dfn}
\begin{rem}
    Let $X$ be an mm-space. 
    Since the measure $\mu_{X}$ is inner regular, we have $\cov{X}{r}{\kappa} < +\infty$ for any $r > 0$ and any $0 < \kappa <1$.
\end{rem}
The covering number is an invariant for mm-spaces.
The following lemma follows immediately from the definition.
\begin{lem}\label{lem:monotonicity of Cov in X}
    Let $X$ and $Y$ be two mm-spaces with $X \prec Y$, and let $r > 0$, $0 < \kappa < 1$.
    Then we have \[\cov{X}{r}{\kappa} \le \cov{Y}{r}{\kappa}.\]
\end{lem}
Lemma~\ref{lem:monotonicity of Cov in X} states that the covering number is monotone increasing with respect to the Lipschitz order relation.
By this fact, we obtain the following definition of the covering number of a pyramid (for details, see~\cite{EKM2024}).
\begin{dfn}\label{def:covering number of pyramids}
    For $r > 0$, $0 < \kappa < 1$, and for a pyramid $\mathcal{P}$, we define \[\cov{\mathcal{P}}{r}{\kappa} := \sup_{X \in \mathcal{P}}\cov{X}{r}{\kappa}\] and call it the $(r,\kappa)$-\textit{covering number} of $\mathcal{P}$.
\end{dfn}
By Lemma~\ref{lem:monotonicity of Cov in X}, for any mm-space $X$, we have \[\cov{\mathcal{P}_{X}}{r}{\kappa} = \cov{X}{r}{\kappa}.\]
The following proposition is straightforward.
\begin{prop}\label{prop:P < Q => Cov(P) < Cov(Q)}
    For two pyramids $\mathcal{P}$ and $\mathcal{Q}$ with $\mathcal{P} \subset \mathcal{Q}$, and for any $r > 0$ and $0 < \kappa < 1$, we have \[\cov{\mathcal{P}}{r}{\kappa} \le \cov{\mathcal{Q}}{r}{\kappa}.\]
\end{prop}
The following lemma is required to prove Theorem~\ref{thm:when a pyrmid is an ext. mm sp.}.
\begin{lem}\label{lem:the covering number of ext mm-space is finite}
    For any extended mm-space $X$, and for any $r > 0$ and $0 < \kappa < 1$, we have \[\cov{\mathcal{P}_{X}}{r}{\kappa} < +\infty.\]
\end{lem}
\begin{proof}
    By Proposition~\ref{prop:decomposition of an ext. mm-sp and its uniqueness}, there exist $A=(a_{n})_{n=1}^{N} \in \mathcal{A}_{1}$ and a sequence of mm-spaces $\{X_{n}\}_{n=1}^{N}$ such that $X$ is mm-isomorphic to $\sum_{n=1}^{N}X_{n}^{a_{n}}$.
    Take a number $N^{\prime} \in \N$ such that $\sum_{n=1}^{N^{\prime}}a_{n} > 1 - \kappa$, and take a number $\kappa^{\prime}$ such that \[0 < \kappa^{\prime} \le (\kappa-\textstyle\sum_{n=N^{\prime}+1}^{N}a_{n})/\sum_{n=1}^{N^{\prime}}a_{n}.\]
    For $n=1,2,\dots, N^{\prime}$, let $\mathcal{N}_{n}$ be a finite subset of $X_{n}$ such that $\mu_{X_{n}}(B_{r}(\mathcal{N}_{n})) \ge 1- \kappa^{\prime}$.
    Then we have 
    \begin{align*}
        \mu_{X}(B_{r}(\textstyle\sum_{n=1}^{N^{\prime}}\mathcal{N}_{n})) &=\sum_{n=1}^{N^{\prime}}a_{n}\mu_{X_{n}}(B_{r}(\mathcal{N}_{n})) \ge \sum_{n=1}^{N^{\prime}}a_{n}(1-\kappa^{\prime}) \\&=1-\sum_{n=N^{\prime}+1}^{N}a_{n} - \sum_{n=1}^{N^{\prime}}a_{n}\kappa^{\prime} \ge 1-\kappa,
    \end{align*}
    which implies that \[\cov{\mathcal{P}_{X}}{r}{\kappa} \le \sum_{n=1}^{N^{\prime}}\#\mathcal{N}_{n} < +\infty.\]
    This completes the proof.
\end{proof}
As a corollary, we immediately obtain the following claim.
\begin{cor}\label{cor:the condition where a pyramid is not an mm-space}
    Let $\mathcal{P}$ be a pyramid.
    If there exist numbers $r > 0$ and $0 < \kappa < 1$ such that $\cov{\mathcal{P}}{r}{\kappa} = + \infty$, then $\mathcal{P}$ is not associated with any mm-space.
\end{cor}
As an application of this corollary, we obtain the following proposition.
\begin{prop}\label{prop:measure of balls is almost zero, then P is not an mm-space}
    Let $\mathcal{P}$ be a pyramid. 
    If there exists a positive number $r$ such that \[\inf_{X \in \mathcal{P}} \sup_{x \in X} \mu_{X}(B_{r}(x)) = 0,\] then $\mathcal{P}$ is not associated with any mm-space.
\end{prop}
\begin{proof}
    Take a number $r > 0$ such that 
    \begin{equation}\label{eq:prop:measure of balls is almost zero, then P is not an mm-space 1}
        \inf_{X \in \mathcal{P}} \sup_{x \in X} \mu_{X}(B_{r}(x)) = 0
    \end{equation}
    and fix it.
    By~\eqref{eq:prop:measure of balls is almost zero, then P is not an mm-space 1}, we find mm-spaces $X_{n} \in \mathcal{P}$, $n \in \N$, such that \[\sup_{x \in X_{n}}\mu_{X_{n}}(B_{r}(x)) < \frac{1}{n}\]for $n \in \N$.
    Let $\mathcal{N}_{n}:=\{x_{n}^{1},x_{n}^{2},\dots,x_{n}^{N_{n}}\} \subset X_{n}$ be a minimizer of $\cov{X_{n}}{r}{1/2}$.
    Then we observe that
    \[\frac{1}{2} \le \mu_{X_{n}}(B_{r}(\mathcal{N}_{n})) \le \sum_{i=1}^{N_{n}}\mu_{X_{n}}(B_{r}(x_{n}^{i})) \le \frac{N_{n}}{n},\]which implies that $N_{n} \ge n/2$.
    Therefore, we obtain \[\cov{\mathcal{P}}{r}{1/2} \ge \cov{X_{n}}{r}{1/2}=N_{n} \ge n/2 \longrightarrow +\infty\]as $n \to \infty$.
    By Corollary~\ref{cor:the condition where a pyramid is not an mm-space}, we see that the pyramid $\mathcal{P}$ is not associated with any mm-space.
    This completes the proof.
\end{proof}

The following statement is known and follows immediately from~\cite[Proposition 7.37]{S2016book}.
\begin{cor}\label{cor:infinite product is not mm-sp.}
    Let $X$ be a nontrivial mm-space and let $1 \le p \le \infty$. 
    Then the pyramid $X_{p}^{\infty}$ is not associated with any mm-space.  
\end{cor}
We nevertheless give an alternative proof as a corollary of Proposition~\ref{prop:measure of balls is almost zero, then P is not an mm-space}, which may be more direct.
\begin{proof}[Proof of Corollary~\ref{cor:infinite product is not mm-sp.}]
    Take any positive number $r$ with $r < \diam{X}/2$ and fix it.
    By Lemma~\ref{lem:the measure of r-balls of infinite product space are zero} and Corollary~\ref{cor:the measure of r-balls of infinite product space for sup-norm are zero}, we have \[\inf_{X\in X_{p}^{\infty}}\sup_{x\in X}\mu_{X}(B_{r}^{X}(x)) \le \lim_{n\to\infty}\sup_{x \in X^{n}}\mu_{X}^{\otimes n}(B_{r}^{X_{p}^{n}}(x))=0.\]
    By Proposition~\ref{prop:measure of balls is almost zero, then P is not an mm-space}, we obtain this corollary.
\end{proof}
We need the following lemma to prove Theorem~\ref{thm:when a pyrmid is an ext. mm sp.}.
By Proposition~\ref{prop:P is compact <=> P=PX}, this lemma can be viewed as a variant of Lemma~\ref{lem:box-precompactness criterion}.
\begin{lem}\label{lem:the condition when a pyramid of finite observable diameter is an mm-space}
    Let $\mathcal{P}$ be a pyramid of finite observable diameter.
    Then, the following (i) and (ii) are equivalent to each other.
    \begin{enumerate}
        \item For any $r > 0$ and any $0 < \kappa < 1$, we have $\cov{\mathcal{P}}{r}{\kappa} < +\infty$.
        \item There exists an mm-space $X$ such that $\mathcal{P} = \mathcal{P}_{X}$.
    \end{enumerate}
\end{lem}
\begin{proof}
    (i) $\Leftarrow$ (ii) is followed by Lemma~\ref{lem:the covering number of ext mm-space is finite}.
    We prove (i) $\Rightarrow$ (ii).
    Assume (i). 
    We prove that the subset $\mathcal{P} \subset \mathcal{X}$ is compact with respect to the box topology. 
    Take any $0< \varepsilon < 1$ and fix it.
    By Lemma~\ref{lem:app. seq. of pyramid}, there exists an approximation sequence $\{X_{n}\}_{n=1}^{\infty}$ of the pyramid $\mathcal{P}$.
    We define $C_{\varepsilon} := \cov{\mathcal{P}}{\varepsilon/4}{\varepsilon/4} < \infty$. 
    Let $\mathcal{N}_{n} = \{x_{n}^{1}, \dots, x_{n}^{N_{n}}\} \subset X_{n}$ be a minimizer of $\cov{X_{n}}{\varepsilon/4}{\varepsilon/4}$. 
    Since $N_{n} \le C_{\varepsilon}$ and \[\sum_{i=1}^{N_{n}}\mu_{X_{n}}(B_{\varepsilon/4}(x_{n}^{i})) \ge \mu_{X_{n}}(\textstyle\bigcup_{i=1}^{N_{n}} B_{\varepsilon/4}(x_{n}^{i})) = \mu_{X_{n}}(B_{\varepsilon/4}(\mathcal{N}_{n})) \ge 1-\varepsilon/4\] for $n \ge 1$, there exists a natural number $i_{n}$ with $1 \le i_{n} \le N_{n}$ such that $\mu_{X_{n}}(B_{\varepsilon/4}(x_{n}^{i_{n}})) \ge (1-\varepsilon/4)/C_{\varepsilon} > 0$. 
    Renumbering the sequence $\{x_{n}^{i}\}_{i=1}^{N_{n}}$, we assume that 
    \begin{equation}\label{eq:section 6 meaure of ball centered at x_n^1}
        \mu_{X_{n}}(B_{\varepsilon/4}(x_{n}^{1})) \ge (1-\varepsilon/4)/C_{\varepsilon}.
    \end{equation}
    for any $n \in \N$.
    \begin{clm}\label{clm:Nn is almost uniformly bounded}
        There exist a positive number $R$ such that for any $n \in \N$, we have \[\mu_{X_{n}}(B_{\varepsilon/4}(\mathcal{N}_{n}^{R})) \ge 1 - \varepsilon/2,\] where $\mathcal{N}_{n}^{R} := \{x_{n}^{i} \in \mathcal{N}_{n} \mid d_{X_{n}}(x_{n}^{1},\, x_{n}^{i}) \le R\} = \mathcal{N}_{n}\cap B_{R}(x_{n}^{1})$.
    \end{clm}

    \begin{proof}
        Set $R :=\mathrm{Sep}(\mathcal{P};(1-\varepsilon/4)/C_{\varepsilon},\varepsilon/4) + 1$.
        By~\cite[Proposition 4.10]{OS2015_limit_formula}, for any $\kappa$ with $0 < \kappa < \min\{(1-\varepsilon/4)/C_{\varepsilon},\hspace{1mm}\varepsilon/4\}$, we have \[\mathrm{Sep}(\mathcal{P};(1-\varepsilon/4)/C_{\varepsilon},\varepsilon/4) \le \obsdiam{\mathcal{P}}{\kappa}.\]
        Since the observable diameter of $\mathcal{P}$ is assumed to be finite, we have $R < \infty$.
        By~\eqref{eq:section 6 meaure of ball centered at x_n^1}, for any $n \in \N$, we have \[\mu_{X_{n}}(B_{\varepsilon/4}(x_{n}^{1})) \ge (1-\varepsilon/4)/C_{\varepsilon}.\]
        By a simple observation, we see that 
        \begin{align*}
            d_{X_{n}}(B_{\varepsilon/4}(x_{n}^{1}),\hspace{1mm}B_{\varepsilon/4}(\mathcal{N}_{n}\setminus \mathcal{N}_{n}^{R})) &> R-1 = \mathrm{Sep}(\mathcal{P};(1-\varepsilon/4)/C_{\varepsilon},\varepsilon/4)\\
                &\ge \mathrm{Sep}(X_{n};(1-\varepsilon/4)/C_{\varepsilon},\varepsilon/4),
        \end{align*}
        which implies that for any $n \in \N$, we have \[\mu_{X_{n}}(B_{\varepsilon/4}(\mathcal{N}_{n}\setminus \mathcal{N}_{n}^{R})) < \varepsilon/4.\]
        This shows that 
        \begin{align*}
            \mu_{X_{n}}(B_{\varepsilon/4}(\mathcal{N}_{n}^{R})) &\ge \mu_{X_{n}}(B_{\varepsilon/4}(\mathcal{N}_{n})\setminus B_{\varepsilon/4}(\mathcal{N}_{n}\setminus \mathcal{N}_{n}^{R})) \\
                &= \mu_{X_{n}}(B_{\varepsilon/4}(\mathcal{N}_{n})) - \mu_{X_{n}}(B_{\varepsilon/4}(\mathcal{N}_{n}\setminus \mathcal{N}_{n}^{R})) \\
                &\ge 1-\varepsilon/4 - \varepsilon/4 = 1-\varepsilon/2,
        \end{align*}
        which completes the proof.
    \end{proof}
 
    Set $\Delta(\varepsilon) := \max\{C_{\varepsilon},\, 2R+\varepsilon\}$. 
    We see that for every $n \in \N$,  
    \begin{itemize}
        \item $\mathcal{N}_{n}^{R}$ is an $\varepsilon$-supporting net of $X_{n}$;
        \item $\#\mathcal{N}_{n}^{R} \le C_{\varepsilon} \le \Delta(\varepsilon)$;
        \item $\diam{\mathcal{N}_{n}^{R}} \le 2R \le \Delta(\varepsilon)$.
    \end{itemize}

    Take any $Y \in \mathcal{P}$ and fix it. 
    There exist positive number $\delta$ and $n \in \N$ such that $0 < \delta < \min\{\varepsilon/8,\, \varepsilon/(4C_{\varepsilon})\}$, $n > 4/\varepsilon$, and $Y \prec_{\delta} X_{n}$ by Lemma~\ref{lem:PXn -> P, Z in P => Z <ek Xnk}.
    Let $g:X_{n} \to Y$ be a 1-Lipschitz map up to $\delta$ such that $\dP(g_{*}\mu_{X_{n}}, \mu_{Y}) \le \delta$, and let $\widetilde{X}_{n} \subset X_{n}$ be a non-exceptional domain of $g$. 
    Set \[I_{n}^{R} := \{1\le i \le N_{n} \mid x_{n}^{i} \in \mathcal{N}_{n}^{R}\}\] and \[J_{n}^{R} := \{j \in I_{n}^{R} \mid \mu_{X_{n}}(B_{\varepsilon/4}(x_{n}^{j})) > \delta\}.\] 
    Then we have 
    \begin{align*}
        \mu_{X_{n}}(B_{\varepsilon/4}(\{x_{n}^{j}\mid j \in J_{n}^{R}\}) &= \mu_{X_{n}}(B_{\varepsilon/4}(\mathcal{N}_{n}^{R}\setminus \{x_{n}^{i}\mid i \in I_{n}^{R}\setminus J_{n}^{R}\})) \\
        &\ge \mu_{X_{n}}(B_{\varepsilon/4}(\mathcal{N}_{n}^{R})\setminus B_{\varepsilon/4}(\{x_{n}^{i}\mid i \in I_{n}^{R}\setminus J_{n}^{R}\})) \\
        &= \mu_{X_{n}}(B_{\varepsilon/4}(\mathcal{N}_{n}^{R})) - \mu_{X_{n}}(\textstyle\bigcup_{i \in I_{n}^{R} \setminus J_{n}^{R}}B_{\varepsilon/4}(\{x_{n}^{i}\})) \\
        &\ge 1-\varepsilon/2 - \delta C_{\varepsilon} \ge 1 - \frac{3}{4}\varepsilon.
    \end{align*}
    For any $ j \in J_{n}^{R}$, we see that $B_{\varepsilon/4}(x_{n}^{j})\cap \widetilde{X}_{n} \neq \emptyset$ since $\mu_{X_{n}}(B_{\varepsilon/4}(x_{n}^{j})) > \delta$.
    Take any points $\tilde{x}_{n}^{j} \in B_{\varepsilon/4}(x_{n}^{j})\cap \widetilde{X}_{n}$, $j \in J_{n}^{R}$, and set $\widetilde{\mathcal{N}}_{n}^{R} := \{\tilde{x}_{n}^{j} \mid j \in J_{n}^{R}\}$. 
    By the definition of $\widetilde{\mathcal{N}}_{n}^{R}$, we have
    \begin{align*}
        \mu_{X_{n}}(B_{\varepsilon/2}(\widetilde{\mathcal{N}}_{n}^{R})) &= \mu_{X_{n}}(\textstyle\bigcup_{j \in J_{n}^{R}}B_{\varepsilon/2}(\tilde{x}_{n}^{j})) \\
            &\ge \mu_{X_{n}}(\textstyle\bigcup_{j \in J_{n}^{R}}B_{\varepsilon/4}(x_{n}^{j})) \ge 1 - \displaystyle\frac{3}{4}\varepsilon.
    \end{align*}
    Since $g$ is a 1-Lipschitz map up to $\delta$, we have \[B_{\varepsilon/2}(\widetilde{\mathcal{N}}_{n}^{R}) \cap \widetilde{X}_{n} \subset g^{-1}(B_{\varepsilon/2 + \delta}(g(\widetilde{\mathcal{N}}_{n}^{R}))) \subset g^{-1}(B_{\varepsilon-\delta}(g(\widetilde{\mathcal{N}}_{n}^{R}))).\] 
    This implies that
    \begin{align*}
        \mu_{Y}(B_{\varepsilon}(g(\widetilde{\mathcal{N}}_{n}^{R}))) &\ge g_{*}\mu_{X_{n}}(B_{\varepsilon-\delta}(g(\widetilde{\mathcal{N}}_{n}^{R}))) - \delta\\
            &\ge \mu_{X_{n}}(B_{\varepsilon/2}(\widetilde{\mathcal{N}}_{n}^{R})\cap\widetilde{X}_{n}) - \delta \ge \mu_{X_{n}}(B_{\varepsilon/2}(\widetilde{\mathcal{N}}_{n}^{R})) - 2\delta\\
            &\ge 1 - \frac{3}{4}\varepsilon - 2 \delta \ge 1 - \varepsilon.
    \end{align*}
    This shows that the subset $g(\widetilde{\mathcal{N}}_{n}^{R})$ of $Y$ is an $\varepsilon$-supporting net of $Y$ with $\#g(\widetilde{\mathcal{N}}_{n}^{R}) \le \# J_{n}^{R} \le \Delta(\varepsilon)$. 
    Moreover, for any points $\tilde{x}_{n}^{i}, \tilde{x}_{n}^{j} \in \widetilde{\mathcal{N}}_{n}^{R}$, we have 
    \begin{align*}
        d_{Y}(g(\tilde{x}_{n}^{i}),\,g(\tilde{x}_{n}^{j})) &\le d_{X_{n}}(\tilde{x}_{n}^{i}, \tilde{x}_{n}^{j}) + \delta \le d_{X_{n}}(x_{n}^{i}, x_{n}^{j}) + \varepsilon/2 + \delta \\
            &\le 2R + \varepsilon/2 + \delta \le 2R+\varepsilon \le \Delta(\varepsilon),
    \end{align*}
    which implies that $\diam{g(\widetilde{\mathcal{N}}_{n}^{R})} \le \Delta(\varepsilon)$. 
    By Lemma~\ref{lem:box-precompactness criterion}, we see that the subset $\mathcal{P} \subset \mathcal{X}$ is compact with respect to the box topology, and by Proposition~\ref{prop:P is compact <=> P=PX}, there exists an mm-space $X$ such that $\mathcal{P} = \mathcal{P}_{X}$. 
    This completes the proof.
    \end{proof}

    \begin{proof}[Proof of Theorem~\ref{thm:when a pyrmid is an ext. mm sp.}]
    (i) $\Leftarrow$ (ii) is followed by Lemma~\ref{lem:the covering number of ext mm-space is finite}.
    We prove (i) $\Rightarrow$ (ii).
    Assume (i).
    By Lemma~\ref{lem:the condition when a pyramid of finite observable diameter is an mm-space}, it suffices to consider the case where the observable diameter of $\mathcal{P}$ is infinite.
    Assume that the observable diameter of $\mathcal{P}$ is infinite.
    By Theorem~\ref{thm:decomposition of pyramids}, there exist $A= (a_{n})_{n=1}^{N} \in \mathcal{A}$ and a sequence $\{\mathcal{P}_{n}\}_{n=1}^{N}$ of pyramids of finite observable diameter such that \[\mathcal{P} = \mathcal{X}^{1-\|A\|_{1}} + \sum_{n=1}^{N}\mathcal{P}_{n}^{a_{n}}.\]

    \begin{clm}\label{clm:Cov(P_n;r,k) <= Cov(P;r,a_n k)}
        Let $\mathcal{P} = \sum_{n=1}^{N}\mathcal{P}_{n}^{a_{n}}$ be a direct sum of pyramids. 
        For any $r > 0$, $0 < \kappa < 1$, and $n \in [N]$, we have \[\cov{\mathcal{P}_{n}}{r}{\kappa} \le \cov{\mathcal{P}}{r}{a_{n}\kappa}.\]
    \end{clm}

    \begin{proof}
        Take any $n \in [N]$ and $X \in \mathcal{P}_{n}$ and fix them.
        We set $X^{\prime} := (X^{a_{n}} + \{*\}^{1-a_{n}})_{2r}$, which belongs to $\mathcal{P}$. 
        Let $\mathcal{N} \subset X^{\prime}$ be a minimizer of $\cov{X^{\prime}}{r}{a_{n}\kappa}$.

        If $* \in \mathcal{N}$, then we have \[1-a_{n}\kappa \le \mu_{X_{^{\prime}}}(B_{r}(\mathcal{N})) \le a_{n}\mu_{X}(B_{r}(\mathcal{N}\setminus\{*\})) + 1-a_{n},\] which implies that \[\mu_{X}(B_{r}(\mathcal{N}\setminus \{*\})) \ge 1 - \kappa.\]
        On the other hand, if $* \notin \mathcal{N}$, then we see that \[1-a_{n}\kappa \le \mu_{X^{\prime}}(B_{r}(\mathcal{N})) = a_{n}\mu_{X}(B_{r}(\mathcal{N})).\] This shows that \[\mu_{X}(B_{r}(\mathcal{N})) \ge (1-a_{n}\kappa)/a_{n} \ge 1-\kappa.\]
        By the above argument, we have \[\cov{X}{r}{\kappa} \le \#\mathcal{N} = \cov{X^{\prime}}{r}{a_{n}\kappa} \le \cov{\mathcal{P}}{r}{a_{n}\kappa}.\]
        Since we choose $X \in \mathcal{P}_{n}$ arbitrarily, we obtain $\cov{\mathcal{P}_{n}}{r}{\kappa} \le \cov{\mathcal{P}}{r}{a_{n}\kappa}$. 
        This completes the proof.
    \end{proof}

    We note that $\cov{\mathcal{X}}{r}{\kappa} = \infty$ for any $r > 0$ and $0 < \kappa < 1$. 
    By Claim~\ref{clm:Cov(P_n;r,k) <= Cov(P;r,a_n k)}, we have $\|A\|_{1}=1$ and $\mathcal{P} = \sum_{n=1}^{N}\mathcal{P}_{n}^{a_{n}}$. 
    Moreover, we see that $\cov{\mathcal{P}_{n}}{r}{\kappa} \le \cov{\mathcal{P}}{r}{a_{n}\kappa} < \infty$ for any $r > 0$, $0 < \kappa < 1$, and $n \in [N]$. 
    By Lemma~\ref{lem:the condition when a pyramid of finite observable diameter is an mm-space}, there exists an mm-space $X_{n}$ such that $\mathcal{P}_{n} = \mathcal{P}_{X_{n}}$. 
    This implies that \[\mathcal{P} = \sum_{n=1}^{N}\mathcal{P}_{X_{n}}^{a_{n}} = \mathcal{P}_{\sum_{n=1}^{N}X_{n}^{a_{n}}},\] which completes the proof of (i) $\Rightarrow$ (ii).
\end{proof}
By Theorem~\ref{thm:when a pyrmid is an ext. mm sp.} and Proposition~\ref{prop:P < Q => Cov(P) < Cov(Q)}, we obtain the following corollary.
\begin{cor}\label{cor:Q < P_X, X; ext. mm-sp. => Q=P_Y}
    Let $\mathcal{Q}$ be a pyramid.
    If we have $\mathcal{Q} \subset \mathcal{P}_{X}$ for some extended mm-space $X$, then there exists an extended mm-space $Y$ such that $\mathcal{Q}=\mathcal{P}_{Y}$.
\end{cor}

\begin{rem}
    If $X$ in Corollary~\ref{cor:Q < P_X, X; ext. mm-sp. => Q=P_Y} is an mm-space, then the conclusion follows from Proposition~\ref{prop:P is compact <=> P=PX}.
\end{rem}

\bibliographystyle{abbrv}
\bibliography{Bibliography}

\begin{thebibliography}{10}

\bibitem{Bill}
P.~Billingsley.
\newblock {\em Convergence of probability measures}.
\newblock Wiley Series in Probability and Statistics: Probability and Statistics. John Wiley \& Sons, Inc., New York, second edition, 1999.
\newblock A Wiley-Interscience Publication.

\bibitem{EKM2024}
S.~Esaki, D.~Kazukawa, and A.~Mitsuishi.
\newblock Invariants for {G}romov's pyramids and their applications.
\newblock {\em Adv. Math.}, 442:Paper No. 109583, 2024.

\bibitem{Fukaya}
K.~Fukaya.
\newblock Collapsing of {R}iemannian manifolds and eigenvalues of {L}aplace operator.
\newblock {\em Invent. Math.}, 87(3):517--547, 1987.

\bibitem{Gromov}
M.~Gromov.
\newblock {\em Metric structures for {R}iemannian and non-{R}iemannian spaces}.
\newblock Modern Birkh\"auser Classics. Birkh\"auser Boston, Inc., Boston, MA, english edition, 2007.
\newblock Based on the 1981 French original, With appendices by M. Katz, P. Pansu and S. Semmes, Translated from the French by Sean Michael Bates.

\bibitem{GrM}
M.~Gromov and V.~D. Milman.
\newblock A topological application of the isoperimetric inequality.
\newblock {\em Amer. J. Math.}, 105(4):843--854, 1983.

\bibitem{Huber}
P.~J. Huber.
\newblock {\em Robust statistics}.
\newblock Wiley Series in Probability and Mathematical Statistics. John Wiley \& Sons, Inc., New York, 1981.

\bibitem{K2022}
D.~Kazukawa.
\newblock Convergence of metric transformed spaces.
\newblock {\em Israel J. Math.}, 252(1):243--290, 2022.

\bibitem{KNS2024}
D.~Kazukawa, H.~Nakajima, and T.~Shioya.
\newblock Principal bundle structure of the space of metric measure spaces.
\newblock {\em Proc. Roy. Soc. Edinburgh Sect. A}, page 1–31, 2024.

\bibitem{KY2021}
D.~Kazukawa and T.~Yokota.
\newblock Boundedness of precompact sets of metric measure spaces.
\newblock {\em Geom. Dedicata}, 215:229--242, 2021.

\bibitem{Ledoux}
M.~Ledoux.
\newblock {\em The concentration of measure phenomenon}, volume~89 of {\em Mathematical Surveys and Monographs}.
\newblock American Mathematical Society, Providence, RI, 2001.

\bibitem{Nakajima_box_dist}
H.~Nakajima.
\newblock Box distance and observable distance via optimal transport.
\newblock Israel Journal of Mathematics, to appear. arXiv:2204.04893.

\bibitem{OS2015_limit_formula}
R.~Ozawa and T.~Shioya.
\newblock Limit formulas for metric measure invariants and phase transition property.
\newblock {\em Math. Z.}, 280(3-4):759--782, 2015.

\bibitem{S2016book}
T.~Shioya.
\newblock {\em Metric measure geometry}, volume~25 of {\em IRMA Lectures in Mathematics and Theoretical Physics}.
\newblock EMS Publishing House, Z\"urich, 2016.
\newblock Gromov's theory of convergence and concentration of metrics and measures.

\bibitem{Shioya_2017}
T.~Shioya.
\newblock Metric measure limits of spheres and complex projective spaces.
\newblock In {\em Measure theory in non-smooth spaces}, Partial Differ. Equ. Meas. Theory, pages 261--287. De Gruyter Open, Warsaw, 2017.

\bibitem{VL_mm_entropy}
A.~M. Vershik and M.~A. Lifshits.
\newblock On mm-entropy of a {B}anach space with {G}aussian measure.
\newblock {\em Theory Probab. Appl.}, 68(3):431--439, 2023.
\newblock Translation of Teor. Veroyatn. Primen. {\bf 68} (2023), 532--543.

\bibitem{Y_limit_formula}
S.~Yokota.
\newblock A complete proof of the limit formula for observable diameter, 2024.
\newblock arXiv:2407.08122.

\end{thebibliography}
\end{document}